\documentclass{article}

\usepackage[preprint]{neurips_2025}

\usepackage[utf8]{inputenc} 
\usepackage[T1]{fontenc}    
\usepackage{hyperref}       
\usepackage{url}            
\usepackage{booktabs}       
\usepackage{amsfonts}       
\usepackage{nicefrac}       
\usepackage{microtype}      
\usepackage{xcolor}         

\usepackage{amsmath,amsthm,amssymb}
\usepackage{algorithm}
\usepackage{algpseudocode}
\usepackage{xcolor}
\usepackage{graphicx}
\usepackage[flushleft]{threeparttable}
\usepackage{pifont}
\newcommand{\cmark}{\ding{51}}%
\newcommand{\xmark}{\ding{55}}%
\usepackage{multirow}

\newtheorem{assumption}{Assumption}
\newtheorem{theorem}{Theorem}

\newtheorem{lemma}{Lemma}

\newtheorem{corollary}{Corollary}

\title{Decentralized Relaxed Smooth Optimization with Gradient Descent Methods}

\author{%
  Zhanhong Jiang\thanks{Corresponding author.}, Aditya Balu, Soumik Sarkar \\
  Translational AI Center\\
  Iowa State University\\
  Ames, IA 50011 \\
  \texttt{zhjiang@iastate.edu, baditya@iastate.edu, soumiks@iastate.edu} \\
}

\begin{document}

\maketitle

\begin{abstract}
    $L_0$-smoothness, which has been pivotal to advancing decentralized optimization theory, is often fairly restrictive for modern tasks like deep learning. The recent advent of relaxed $(L_0,L_1)$-smoothness condition enables improved convergence rates for gradient methods. Despite centralized advances, its decentralized extension remains unexplored and challenging.
    In this work, we propose the first general framework for decentralized gradient descent (DGD) under $(L_0,L_1)$-smoothness by introducing novel analysis techniques. For deterministic settings, our method with adaptive clipping achieves the best-known convergence rates for convex/nonconvex functions without prior knowledge of $L_0$ and $L_1$ and bounded gradient assumption. In stochastic settings, we derive complexity bounds and identify conditions for improved complexity bound in convex optimization. The empirical validation with real datasets demonstrates gradient-norm-dependent smoothness, bridging theory and practice for $(L_0,L_1)$-decentralized optimization algorithms.
\end{abstract}

\section{Introduction}\label{intro}
In this paper, we focus primarily on the unconstrained decentralized optimization on a networked system involving $N$ agents,
\begin{equation}\label{eq_1}
    F^*:=\text{min}_{x\in\mathbb{R}^d}\frac{1}{N}\sum_{i=1}^Nf^i(x),
\end{equation}
where $f^i:\mathbb{R}^d\to\mathbb{R}$ is a local smooth function associated with agent $i$. Due to the emergence of deep learning and the demand of data privacy, the assurance of efficient convergence in decentralized optimization has remained increasingly challenging yet critical. Existing methods, such as the decentralized gradient descent (DGD) and its variants, are typically dependent on $L_0$ Lipschitz-smoothness to guarantee convergence rates. However, in many modern machine learning problems, this assumption is overly restrictive or may not necessarily hold, particularly when leveraging deep learning models for parameterization. To relax such an assumption, a recent work~\cite{zhang2019gradient} empirically showed that \textit{the norm of the Hessian evidently correlates with a norm of the gradient of the loss functions} when training neural network models. This evidence motivates the authors to relax the traditional $L_0$-smoothness by introducing a new and more realistic assumption on a function class, i.e., $(L_0,L_1)$-smoothness. Immediately, the class of $(L_0,L_1)$-smooth functions comprises the class of $L_0$-smooth functions when $L_1=0$. In addition to that, they also established the convergence rate for various gradient descent methods for nonconvex functions, claiming that the normalized and clipping methods are favorable for the new class of smooth functions. Upon this work,
recent years have witnessed numerous works designing optimization methods and analyzing convergence to solve the $(L_0,L_1)$-smooth optimization problems~\cite{vankov2024optimizing,gorbunov2024methods,li2024convex,wang2023convergence,reisizadeh2023variance}. Regardless of the prosperous progress in centralized methods, to the best of our knowledge, there are no reported results on their \textit{decentralized} counterparts with this new smoothness condition. However, spurred by edge computing and multiagent systems~\cite{yang2009decentralized,terelius2011decentralized,cao2020decentralized}, decentralized optimization or learning has attracted particularly considerable attention in a variety of areas.
Additionally, the existing proof techniques for decentralized optimization cannot directly be applied to solve the $(L_0,L_1)$-smooth optimization problems, posing difficulties in understanding convergence behaviors of decentralized optimization methods with $(L_0,L_1)$-smoothness.
\begin{table}
\centering
\begin{threeparttable}
\caption{Comparison between our results and existing results.}
   \begin{tabular}{c c c c c}
    \hline
    Ref. & Smo. & Sto. & Conv. & Complexity\\
    \hline
    \cite{yuan2016convergence} &  $(L_0,0)$ & \xmark & \cmark & \small{$\mathcal{O}(\frac{1}{\alpha\epsilon})^\dagger$}\\
    \cite{zeng2018nonconvex} &  $(L_0,0)$ & \xmark & \cmark &\small{$\tilde{\mathcal{O}}(\frac{1}{\epsilon^2})^\ddagger$}\\ \hline 
    \textbf{Theorem~\ref{theorem_1}} & $(L_0,L_1)$ & \xmark & \cmark & \small{$\mathcal{O}(\textnormal{max}\{L_1R\;\text{ln}(\frac{F_0}{\epsilon}),\frac{L_0R^2}{\epsilon}\})$}\\
    \textbf{Theorem~\ref{theorem_2}}  & $(L_0,L_1)$ & \xmark & \xmark & \small{$\mathcal{O}(\textnormal{max}\{\frac{L_0F_0}{\epsilon},\frac{L_1F_0}{\epsilon}\})$}\\ \hline
    \multirow{2}{*}{\cite{koloskova2020unified}} &  $(L_0,0)$ & \cmark & \cmark & \small{$\mathcal{O}(\frac{\bar{\sigma}^2}{N\epsilon^2}+\frac{L}{\epsilon}+\frac{\sqrt{L}(\bar{\delta}+\bar{\sigma})}{\epsilon^{3/2}})^\#$}\\  & $(L_0,0)$ &\cmark & \xmark & \small{$\mathcal{O}(\frac{L\hat{\sigma}^2}{N\epsilon^2}+\frac{L}{\epsilon}+\frac{\sqrt{L}(\hat{\delta}+\hat{\sigma})}{\epsilon^{3/2}})^\#$}\\ \hline
    \cite{yuan2022revisiting}  & $(L_0,0)$ & \cmark & \xmark & \small{$\mathcal{\tilde{O}}(\textnormal{max}(\frac{\sigma^2}{n\epsilon^2},\frac{1}{\epsilon\sqrt{1-\rho^2}}))$}\\
    \hline
    \textbf{Theorem~\ref{theorem_3}} & $(L_0,L_1)$ & \cmark & \cmark & \small{$\begin{aligned}
        \mathcal{O}(\textnormal{max}\{\frac{NR^2+\sigma^2}{N\epsilon}+\frac{(L_1^2+L_1)(\sigma^2+\delta^2)}{(1-\sqrt{\rho})^2\epsilon}, \\ \frac{N^2R^4+\sigma^4}{N^2\epsilon^2}+\frac{L_0^{4/3}(\sigma^2+\delta^2)^{2/3}}{(1-\sqrt{\rho})^{4/3}\epsilon^{2/3}}+\frac{L_0(\sigma^2+\delta^2)}{(1-\sqrt{\rho})^2\epsilon}\})
    \end{aligned}$}\\
    \textbf{Theorem~\ref{theorem_4}} & $(L_0,L_1)$ & \cmark & \xmark & \small{$\mathcal{O}(\textnormal{max}\{\frac{C}{\epsilon^4}, \frac{L_1\sigma}{\epsilon^2},\frac{L_1}{\epsilon}\})^\|$}\\
    \hline
    \end{tabular}
    \begin{tablenotes}
	\small {
	\item Dec.: decentralized; Smo.: smooth; Sto: stochastic, Conv.: convex; $\epsilon$: desired accuracy; $\dagger$: only convergence to $F^*+\mathcal{O}(\alpha)$; $\ddagger$: request of a diminishing step size and bounded gradient assumption to converge to $F^*$; and they don't show explicit rate for nonconvex case; $R$: $\|\bar{x}_0-x^*\|$; $F_0$: $F(\bar{x}_0)-F^*$; \#: $\bar{\sigma}, \bar{\delta}, \hat{\sigma}, \hat{\delta}$ are due to bounded noise in Assumption 3 in~\cite{koloskova2020unified}, and we use $\tau=1$ from their work; $\sigma, \delta$ are defined in Assumption~\ref{assum_3}; $0<\rho<1$; $C$ is the constant defined in Theorem~\ref{theorem_4}; $\|$: to unify the analysis for both deterministic and stochastic, for nonconvex functions, we use the metric of $\textnormal{min}_{1\leq k\leq K}\|\nabla F(\bar{x}_k)\|\leq \epsilon$, instead of $\frac{1}{K}\sum_{k=1}^K\|\nabla F(\bar{x}_k)\|^2\leq\epsilon$ in~\cite{koloskova2020unified,yuan2022revisiting}, where $\|\cdot\|$ is Euclidean norm.
        }
    \end{tablenotes}
\end{threeparttable}
\label{tab:comparison}
\end{table}

\textbf{Contributions.} Motivated by this gap, the present work probes decentralized $(L_0,L_1)$-smooth optimization with gradient descent methods for both generally convex and nonconvex functions. We summarize the main contributions as follows and the comparison between our and existing results in Table~\ref{tab:comparison}.
    1) In deterministic setting, we develop $(L_0,L_1)$-DGD with a \textit{clipping} step size. For generally convex problems, it achieves $\mathcal{O}(\textnormal{max}\{L_1R\;\text{ln}(F_0/\epsilon),L_0R^2/\epsilon\})$ complexity bound for finding an $\epsilon$-stationary point. For nonconvex functions, our gradient methods achieve $\mathcal{O}(\textnormal{max}\{L_0F_0/\epsilon^2,L_1F_0/\epsilon\})$. $F_0=F(\bar{x}_0)-F^*$ is the initialization error and $R=\|\bar{x}_0-x^*\|$ is the initial instance to a solution, where $\bar{x}_0$ is the average of all $x^i$ at the initial step. $N$ is the number of agents. They resembles the best-known bounds in the centralized counterparts (see Theorem~\ref{theorem_1} for convex and Theorem~\ref{theorem_2} for nonconvex).
    2) To facilitate the theoretical underpinnings in the stochastic setting, we also investigate complexity bounds for decentralized stochastic gradient descent (DSGD) with $(L_0,L_1)$-smoothness, which exhibit to match the best available bounds achieved for both convex and nonconvex functions. Particularly, for convex functions, as long as $\|\nabla f^i(x)\|\geq L_0/L_1$ ($\|\nabla f^i(x)\|$ is the gradient norm), the algorithm has improved complexity bound, and when $\|\nabla f^i(x)\|< L_0/L_1$ is satisfied, $(L_0,L_1)$-DSGD has standard complexity (see Theorem~\ref{theorem_3} for convex and Theorem~\ref{theorem_4} for nonconvex).

\noindent\textbf{Significance.} In our analysis, it should be noted that the step sizes are correlated with the constants $L_0$ and $L_1$, raising concerns in designing $(L_0,L_1)$-DGD or DSGD for real-world tasks, as estimating such constants in practice remains a longstanding challenge in the optimization community. This is indeed a shared limitation with many theoretically grounded first-order methods, including those assuming standard $L_0$-smoothness. While our work adopts this assumption for theoretical tractability, we fully acknowledge this limitation. We also highlight recent works~\cite{vankov2024optimizing,gorbunov2024methods} in centralized settings that aim to circumvent this need by developing normalized or adaptive methods. This motivates us to extend such methods to decentralized settings under relaxed smoothness, which is an important next step and left as a future work. Our contribution lies in being the \textit{first to rigorously characterize convergence under decentralized relaxed smoothness}, an area previously unexplored to the best of our knowledge. While this may seem like a natural extension, our analysis required careful handling of both clipping dynamics and consensus errors in the decentralized setting, which are nontrivial challenges that are not addressed in existing literature.
\section{Related Work}

\textbf{$(L_0,L_1)$-smooth optimization.} $L_0$-smoothness~\cite{hu2023beyond,xie2024trust,erickson1981lipschitz} has played a central role in providing convergence guarantee for optimization algorithms. Particularly, in massive deep learning tasks primarily with nonconvex optimization, it ensures a critical descent loss descent property. Despite its significant successes and widespread use, this assumption can be overly restrictive, since $L_0$ must capture the worst-case smoothness. To address this issue, a recent work~\cite{zhang2019gradient} empirically showed a phenomenon that the local smoothness constant decreases during training alongside the gradient norm, contradicting the intuition that smoothness constant remains the same. Thus, $L_0$-smoothness has been relaxed to be $(L_0,L_1)$-smoothness in recent studies~\cite{zhang2020improved,koloskova2023revisiting,vankov2024optimizing,gorbunov2024methods,lobanov2024linear, tyurin2024toward, crawshaw2022robustness,khirirat2024error}. With this new smoothness assumption, systematic studies have recently been conducted primarily on convex problems with existing gradient descent methods. In~\cite{vankov2024optimizing}, the authors claimed a significantly better complexity bound for convex functions in deterministic settings, also studying the normalized gradient method and gradient method with Polyak step size. They found that neither of these methods require explicit knowledge of $L_0$ and $L_1$. However, this work fails to extend the analysis to stochastic settings to establish rigorous bounds. Concurrently, another work~\cite{gorbunov2024methods} presented the similar analysis independently to derive the tighter rates for gradient descent with smoothed clipping and extended $(L_0,L_1)$-smoothness to stochastic settings for further advancing the convergence guarantee. Another recent work~\cite{lobanov2024linear} also focused only on convex setup and claimed the linear convergence rate for $(L_0,L_1)$-gradient descent, when the gradient norm satisfied a certain condition. Likewise, this conclusion applies to normalized/clipped gradient descent and even random coordinate descent. However, these two works show no results in nonconvex problems, similar to~\cite{lobanov2025power,chezhegov2025convergence}. Other works such as~\cite{koloskova2023revisiting,takezawa2024polyak,li2023convex} formally obtained complexity bounds for gradient descent methods for both convex and nonconvex cases, but relying on additional constants that can be extremely larger than $L_0$ and $L_1$. A generalized smoothness assumption was also proposed in~\cite{tovmasyan2025revisiting} for stochastic proximal point methods with strongly-convex functions, extending both $L_0$- and $(L_0,L_1)$-smoothness. A more recent work examined the new convergence theory for adaptive gradient methods with the relaxed smoothness assumption~\cite{crawshaw2025complexity}.
To the best of our knowledge, no prior work offers comprehensive theoretical studies of $(L_0,L_1)$-smoothness in fully decentralized settings with rigorous convergence guarantees, posing a theoretical gap.

\textbf{Decentralized $L_0$-smooth optimization.} Decentralized optimization with $L_0$-smoothness~\cite{kovalev2020optimal,hendrikx2019accelerated,yuan2022revisiting,bhuyan2024communication} has become the standard theoretical analysis framework in edge computing and multiagent systems. Particularly, driven by decentralized deep learning, $L_0$-smoothness is an indispensable requirement for providing convergence guarantee. In~\cite{yuan2022revisiting}, the authors established an optimal rate for $L_0$-smooth and nonconvex functions, improving the results shown in~\cite{lu2021optimal}. Another work~\cite{koloskova2020unified} unified the theory for decentralized SGD, consolidating proof techniques for different function classes, including strongly convex, convex and nonconvex. To further tackle the communication bottleneck, compression and gradient tracking methods were leveraged in~\cite{zhao2022beer} to improve the convergence rate in~\cite{koloskova2019decentralized} from $\mathcal{O}(1/K^{2/3})$ to $\mathcal{O}(1/K)$, where $K$ is the number of iterations. A more recent work~\cite{li2025convergence} investigated the trade-offs between convergence rate, compression ratio, network connectivity, and privacy in decentralized nonconvex optimization with differential privacy.
Other works in this line of research have focused on communication efficiency~\cite{lu2020moniqua}, impact of topology~\cite{zhu2022topology}, asynchronous quantized updates~\cite{nadiradze2021asynchronous}, precise gradient computation~\cite{esfandiari2021cross}, and large-batch training~\cite{yuan2021decentlam}. More recently, decentralized learning is also applied to model merging~\cite{saadati2024dimat} and low-rank adaptation in large language models~\cite{ghiasvand2025decentralized}, highlighting its effectiveness in training large models. Despite the significance of decentralized learning/optimization, all of these works resort to $L_0$-smoothness assumption as a key requirement to ensure convergence guarantee. Additionally, as the theoretical analysis in aforementioned works is primarily obtained for deep learning models, there is a departure from practice. Since in many of these models, the adoption of activation functions (e.g., ReLU) empirically violates the $L_0$-smoothness assumption, demanding urgently a natural relaxation of it. 

\noindent\textbf{Federated $(L_0,L_1)$-smooth optimization.} Another prominent distributed learning paradigm is federated learning (FL), which has recently seen significant advances in convergence analysis under relaxed smoothness assumptions. In particular, Demidovich et al.~\cite{demidovich2024methods} have recently proposed and analyzed new methods for $(L_0,L_1)$ non-convex federated optimization, demonstrating consistency with established results for standard smooth problems. This smoothness framework was further extended to stochastic normalized error feedback algorithms to enhance communication efficiency in FL settings~\cite{khirirat2024communication}. Another study~\cite{li2024improved} examined the effects of per-sample and per-update clipping on the convergence of FedAvg under relaxed smoothness, offering a refined theoretical analysis. Motivated by the rise of large language models (LLMs), the authors of~\cite{andrei2024resilient} developed resilient federated LLMs over wireless networks, incorporating $(L_0,L_1)$-coordinate-wise smoothness into their analysis. Regardless of these developments, to the best of our knowledge, no results or theoretical guarantees have yet been reported for fully decentralized settings under relaxed smoothness.

\section{Preliminaries and Problem Formulation}\label{prelim}

\textbf{Graph Theory.}
Consider a network involving $N$ agents and denote by $\mathcal{G}=(\mathcal{V}, \mathcal{E})$ the connected topology, where $\mathcal{V} = \{1,2,...,N\}$ and $\mathcal{E}\subseteq \mathcal{V}\times \mathcal{V}$. If $(i, j)\in\mathcal{E}$, then agent $i$ is able to communicate with agent $j$. We also define the neighborhood of agent $i$ as follows: $Nb(i):=\{j\in\mathcal{V}:(i,j)\in\mathcal{E}\;or\;j=i\}$. Without loss of generality, we assume the graph $\mathcal{G}$ is \textit{connected} and \textit{undirected}. The $N$ agents jointly solve the consensus optimization problem in Eq.~\ref{eq_1}, in which consensus averaging is critical to maintain closeness among agents that learn to
achieve the shared optimal model parameter $x^*:=\text{argmin}_{x\in\mathbb{R}^d}\frac{1}{N}\sum_{i=1}^Nf^i(x)$ (we assume that $x^*$ exists). To quantify the communication protocol, we denote by $\Pi$ the mixing matrix with each element $\pi_{ij}\in [0,1]$ signifying the probability of any pair of agents communicating with each other. 
We now state the first assumption for $\Pi$, which has been utilized frequently in existing works~\cite{esfandiari2021cross,yu2019linear}. 
\begin{assumption}\label{assumption_1}
   $\Pi\in\mathbb{R}^{N\times N}$ is a symmetric doubly stochastic matrix satisfying $\lambda_1(\Pi)=1$ and
$
       \textnormal{max}\{|\lambda_2(\Pi)|, |\lambda_N(\Pi)|\}\leq \sqrt{\rho}<1,
$
$\rho\in(0,1)$, $\lambda_l(\cdot)$ is the $l$-th largest eigenvalue of the matrix.
\end{assumption}

\textbf{$(L_0,L_1)$-smoothness.} We first present what a function needs to satisfy to be $(L_0,L_1)$-smooth.
\begin{assumption}~\cite{crawshaw2025complexity}\label{definition_1}
    A continuously differentiable function $f^i$ for all $i\in\mathcal{V}$ is assumed to be $(L_0,L_1)$-smooth if there exist constants $L_0>0$ and $L_1\geq 0$, then, for all $x,y\in\mathbb{R}^d$ with $\|y-x\|\leq \frac{1}{L_1}$,
    \begin{equation}\label{eq_3}
        \|\nabla f^i(y)-\nabla f^i(x)\|\leq (L_0+L_1\|\nabla f^i(x)\|)\|y-x\|.
    \end{equation}
\end{assumption}
Assumption~\ref{definition_1} is slightly different from the symmetric smoothness assumption (Assumption 1.3 in~\cite{gorbunov2024methods}), which is a more general formulation that bounds the gradient variation using the maximum gradient norm along the path between $x$ and $y$. In contrast, our work adopts an asymmetric version, which has also been used in recent works such as~\cite{crawshaw2025complexity} and~\cite{zhang2020improved}. This form avoids evaluating a supremum along the path and instead provides a one-sided control based on the local gradient norm at $y$ which simplifies decentralized analysis and recursive inequalities involving local updates.
While the symmetric form is strictly more general, the asymmetric assumption is sufficient for our theoretical development and better suited for decentralized optimization where local information is more readily available than global gradient path behavior.
If $f^i$ is twice continuously differentiable, then Eq.~\ref{eq_3} implies the following relationship
$
    \|\nabla^2f^i(x)\|\leq L_0+L_1\|\nabla f^i(x)\|, \forall x\in\mathbb{R}^d.
$
Different from most existing works~\cite{gorbunov2024methods,zhang2020improved,vankov2024optimizing} where they have not yet considered multiple agents, Assumption~\ref{definition_1} in our work is particularly for a local agent $i$. However, in the analysis, we will also need to establish the same smoothness condition for the sum $F(x):=\frac{1}{N}\sum_{i=1}^Nf^i(x)$. To achieve this, we rewrite it into $F(x):=\frac{1}{N}\sum_{i=1}^Nf^i(x^i)$, s.t., $x_1=...=x_N$, which has been generic in numerous works on decentralized optimization~\cite{berahas2018balancing}. Intuitively, this makes sense as each individual agent learns its own model locally based on the data only known to it, while the constraint plays a central role in averaging them through some communication protocol. Following from Proposition 2.4 in~\cite{vankov2024optimizing}, we can obtain that $F(x)$ is also $(L_0,L_1)$-smooth. One may notice that in Assumption~\ref{definition_1}, we have an additional condition $\|y-x\|\leq 1/L_1$, making this assumption slightly different from the one in~\cite{vankov2024optimizing}. Though it seems a bit strong in some sense, this has widely been adopted in~\cite{zhang2020improved,koloskova2023revisiting,crawshaw2025complexity}. In the later analysis, such a condition is also key to ensure the convergence guarantee, particularly mitigating the issue of exponential factor in the loss descent.

In numerous previous decentralized optimization works, bounded gradient assumption also appears to typically characterize the consensus error, even with the assumption that the initialization of $x^i$ is 0. In deterministic settings, though the authors in~\cite{yuan2016convergence,berahas2018balancing} did not formally impose bounded gradient assumption, they still derived it and applied to their analysis. This implies the significance of bounded gradient in DGD. However, in stochastic settings, it can be replaced by the bounded variance or second moment of stochastic gradients, which is a much weaker condition. In this work, we will show a different analysis technique for addressing the convergence in deterministic settings.

\section{Deterministic Setting}\label{deterministic}
We present $(L_0,L_1)$-DGD in Algorithm~\ref{alg:dgd}. DGD has been one of the most popular algorithms in decentralized optimization and can flexibly be extended to stochastic settings. However, the majority of the existing theoretical results depend highly on the traditional smoothness assumption. We have not yet been aware of any reported results using $(L_0,L_1)$-smoothness assumption. 
In Algorithm~\ref{alg:dgd}, the step size $\alpha_k$ is somewhat similar to the \textit{gradient clipping} which has particularly been popular in training large models. Notably,
it is constrained by the larger gradient norm between $x^i_k$ and $\bar{x}_k$. This intuitively suggests that the local update for $x^i_k$ should not be too far away from the average $\bar{x}_k$. This condition may also be simplified when $k$ is relatively larger, as in the later phase of optimization, the distinction between $x^i_k$ and $\bar{x}_k$ becomes much smaller. Thereby, one can only use $\|\nabla f^i(x^i_k)\|$. 
Due to the page limit, we defer all technical proof in this section to Supplementary Materials (SM).

\begin{algorithm}
\caption{$(L_0,L_1)$-Decentralized Gradient Descent}\label{alg:dgd}
\begin{algorithmic}
\Require Mixing matrix $\Pi$, the number of epochs $K$, initialization $x^i_0, i\in\mathcal{V}$, $L_0, L_1$
\Ensure $x^i_K, i\in\mathcal{V}$
\For{from 0 to $K-1$}
    \For{each agent $i$}
        \State $\alpha_k=\textnormal{min}\{\frac{1}{2L_0},\frac{1}{3L_1\textnormal{max}_i\{\|\nabla f^i(x^i_k)\|\}}, \frac{1}{3L_1\textnormal{max}_i\{\|\nabla f^i(\bar{x}_k)\|\}}\}$ \Comment{$\bar{x}_k=\frac{1}{N}\sum_{i=1}^Nx^i_k$}
        \State $x^i_{k+1/2}=\sum_{j\in Nb(i)}\pi_{ij}x^j_k$
        \State $x^i_{k+1}=x^i_{k+1/2}-\alpha_k \nabla f^i(x^i_k)$
    \EndFor
\EndFor
\end{algorithmic}
\end{algorithm}
Though there exist abundant theoretical results for DGD, obtaining the convergence rate for the $(L_0, L_1)$-smooth functions is new and non-trivial. We recall the key loss descent property that has significantly assisted in convergence in traditional analysis:
$
    f^i(y)\leq f^i(x) +\langle\nabla f^i(x), y-x\rangle + \frac{L_0}{2}\|y-x\|^2, \forall x,y\in\mathbb{R}^d
$.
Equivalently, the last inequality leads to the Lipschitz continuous gradient in the following:
$
    \|\nabla f^i(y)-\nabla f^i(x)\|\leq L_0\|y-x\|$.
The difference between Eq.~\ref{eq_3} and Lipschitz continuous gradient is the smoothness constant. The former has the additional $L_1\|\nabla f^i(x)\|$, which poses difficulty in the analysis. Essentially, when $L_1=0$, the $(L_0,L_1)$-smoothness degenerates to $L_0$-smoothness. 
However, Assumption~\ref{definition_1} also leads to the following relationship that is critical to establish the loss descent and summarized in Lemma~\ref{lemma_1}.
\begin{lemma}\label{lemma_1}
    Suppose that $f^i$ is a twice continuously differentiable function. Then $f^i$ is $(L_0,L_1)$-smooth if and only if the following relationship holds for any $x, y\in\mathbb{R}^d$,
$
        f^i(y)\leq f^i(x) +\langle\nabla f^i(x), y-x\rangle + (L_0+L_1\|\nabla f^i(x)\|)\frac{\psi(L_1\|y-x\|)}{L_1^2},
$
where $\psi(\gamma)=e^\gamma-\gamma - 1, \gamma\geq 0$.
\end{lemma}
Lemma~\ref{lemma_1} shows the loss descent with the new smoothness condition in decentralized settings. Nevertheless, the third term on the right hand side has the exponential term. Using the existing proof techniques fails to derive the explicit convergence rate, for DGD. In this paper, we will come up with new analysis techniques to resolve this issue. The techniques are \textit{non-trivial} as we have to consider the impact of the consensus estimate among different agents on the optimality. The first key is to establish the local progress on the averaged model with $(L_0,L_1)$-smoothness and then extend it from $f^i$ to $F$.
\begin{lemma}\label{lemma_2}
    Let Assumption~\ref{definition_1} hold. Suppose that $f^i:\mathbb{R}^d\to\mathbb{R}$ is twice continuously differentiable. Let $\bar{x}_k\in\mathbb{R}^d$, where $\bar{x}_k=\frac{1}{N}\sum_{i=1}^Nx^i_k$, and $x^i_k$ is produced by Algorithm~\ref{alg:dgd} with step size $\alpha_k=\textnormal{min}\{\frac{1}{2L_0},\frac{1}{3L_1\textnormal{max}_i\{\|\nabla f^i(x^i_k)\|\}}, \frac{1}{3L_1\textnormal{max}_i\{\|\nabla f^i(\bar{x}_k)\|\}}\}$. Then, the following relationship holds:
$
        f^i(\bar{x}_k)-f^i(\bar{x}_{k+1})\geq \frac{\|\nabla f^i(\bar{x}_k)\|^2}{4L_0+6L_1\|\nabla f^i(\bar{x}_k)\|}.
$
\end{lemma}
The conclusion in Lemma~\ref{lemma_2} generalizes the conclusion of Lemma B.2 in~\cite{vankov2024optimizing} from centralized to decentralized settings, while differing in the update law and step size. 
This lemma critically suggests that with $(L_0,L_1)$-smoothness assumption and a properly selected step size, the local loss consecutively decreases, ensuring the same property for the global loss $F$ in the following lemmas. In this context, the analysis is focused on the ensemble average, $\bar{x}_k$ instead of $x^i_k$ as we will investigate the global loss induced by each local loss~\cite{saadati2024dimat,esfandiari2021cross}.
The step size in Lemma~\ref{lemma_2} acts as a clipping step size that has been quite popular and practically useful in many empirical tasks. A recent finding reported in~\cite{vankov2024optimizing} reveals that it is a simple approximation of the \textit{optimal} step sizes, ensuring the bound on the objective progress in Lemma~\ref{lemma_2}. 
Pertaining to Eq.~\ref{eq_3}, we can know that the equivalent smoothness constant has now varied with the gradient $\|\nabla f^i(\bar{x}_k)\|$ when $x=\bar{x}_k$. Hence, we investigate two scenarios, i.e., $\|\nabla f^i(\bar{x}_k)\|\geq L_0/L_1$ and $\|\nabla f^i(\bar{x}_k)\|< L_0/L_1$.

\begin{lemma}\label{lemma_3}
    Let Assumption~\ref{definition_1} hold. Suppose that $f^i:\mathbb{R}^d\to\mathbb{R}$ is twice continuously differentiable. If $\|\nabla f^i(\bar{x}_k)\|\geq L_0/L_1$,     $
        F(\bar{x}_k)-F(\bar{x}_{k+1})\geq \frac{\|\nabla F(\bar{x}_k)\|}{10L_1}
    $; if $\|\nabla f^i(\bar{x}_k)\|< L_0/L_1$, $
    F(\bar{x}_k)-F(\bar{x}_{k+1})\geq \frac{\|\nabla F(\bar{x}_k)\|^2}{10L_0}
$. $\bar{x}_k$ is defined as in Lemma~\ref{lemma_2}.
\end{lemma}
It implies that the global loss descent is dictated heavily by the value of $\|\nabla F(\bar{x}_k)\|$ due to Lemma~\ref{lemma_2}. Another interesting observation is that when $L_1\to 0$, making each $f^i$ close to $L_0$-smooth, the loss descent may be more significant for the phase with larger gradient norm. However, a larger number of agents can lead to slow convergence, which has empirically been validated in previous works~\cite{saadati2024dimat}. 
When $\|\nabla f^i(\bar{x}_k)\|< L_0/L_1$, it is equivalent to the bounded gradient such that the conclusion from Lemma~\ref{lemma_3} exhibits the typical loss descent in numerous existing works. We are now ready to state the first main result for generally convex functions by combining the two scenarios.
\begin{theorem}\label{theorem_1} (Convex, deterministic)
    Let Assumptions~\ref{assumption_1} and~\ref{definition_1} hold. Let $\{\bar{x}_k\}$ be defined as in Lemma~\ref{lemma_2} with step size $\alpha_k=\textnormal{min}\{\frac{1}{2L_0},\frac{1}{3L_1\textnormal{max}_i\{\|\nabla f^i(x^i_k)\|\}}, \frac{1}{3L_1\textnormal{max}_i\{\|\nabla f^i(\bar{x}_k)\|\}}\}$. Suppose that $f^i:\mathbb{R}^d\to\mathbb{R}$ is twice continuously differentiable convex function. Let $x^*$ be an arbitrary solution to the problem in Eq.~\ref{eq_1} and let $F_0:=F(\bar{x}_0)-F^*$. Then, $F(\bar{x}_K)-F^*\leq \epsilon$ for any given $0\leq \epsilon\leq F(\bar{x}_0)$ whenever
    $K\geq 10\textnormal{max}\{L_1R\textnormal{ln}\frac{F_0}{\epsilon},\frac{L_0R^2}{\epsilon}\}$,
    where $R=\|\bar{x}_0-x^*\|$.
\end{theorem}
    Theorem~\ref{theorem_1} explicitly shows the complexity bound for $(L_0,L_1)$-DGD with convex functions. To the best of our knowledge, this is the first time showing the established rigorous complexity bound with the relaxed smooth assumption in decentralized setting. Compared to the recent results in its centralized counterparts~\cite{lobanov2024linear,vankov2024optimizing,gorbunov2024methods}, our result reveals the similar rate, but differing in constants due to the number of agents, which causes more computational overheads in practice. Additionally, Theorem~\ref{theorem_1} implies linear convergence rate when $\|\nabla f^i(\bar{x}_k)\|\geq L_0/L_1$ and sublinear rate when $\|\nabla f^i(\bar{x}_k)\|< L_0/L_1$, which has been disclosed in~\cite{lobanov2024linear} for centralized settings. We also see that the complexity bound is influenced by the initialization error $F_0$ that could be arbitrarily poor. Thus, a reasonably good initialization enables the smaller complexity. Also, when comparing the rate in Theorem~\ref{theorem_1} to that in~\cite{yuan2016convergence} for only $L_0$-smooth convex functions, though both of them remain the same, the latter only ensures the convergence to the neighborhood of $x^*$ at the level of $\mathcal{O}(\alpha)$, where $\alpha$ is a constant step size. Another interesting observation from our work is that based on recent results of the monotonicity of gradient norm~\cite{lobanov2024linear,gorbunov2024methods}, $\alpha_k$ defined in Theorem~\ref{theorem_1} can somehow acts as a \textit{diminishing} step size, but not negatively affecting the convergence rate. However, in another work~\cite{zeng2018nonconvex}, to ensure the convergence to $F^*$, they require bounded gradient assumption and a diminishing step size in a form of $\mathcal{O}(1/k^{\varepsilon})$, where $\varepsilon\in (0,1]$, while attaining a worse bound $\tilde{\mathcal{O}}(1/\epsilon^2)$. Also, the uniform upper bound for the gradient can be arbitrarily large.
    Theorem~\ref{theorem_1} also implies that under Assumption~\ref{definition_1} with $L_0=0$, Algorithm~\ref{alg:dgd} converges to the desired accuracy $\epsilon$ after $K=\mathcal{O}(L_1R\textnormal{ln}(F_0/\epsilon))$ iterations.
    
Our complexity lower bound does not need the uniformly bounded gradient assumption, making it much tighter and more generalizable. However, the dependence on the spectral gap $1-\sqrt{\rho}$ is not explicitly shown in the complexity bound, and it is recognized as the indicator of the impact of different topologies. In previous works~\cite{berahas2018balancing,zeng2018nonconvex}, such an impact is demonstrated through the consensus error, which is characterized by $R$ in our case. Thus, to clearly illustrate the central role of the spectral gap, we present the following result by upper bounding $R$.
\begin{corollary}\label{corollary_1}
    Let Assumptions~\ref{assumption_1} and~\ref{definition_1} hold. With Theorem~\ref{theorem_1}, we have $F(\bar{x}_K)-F^*\leq \epsilon$ whenever
    $
    K\geq 10\textnormal{max}\{\sqrt{2}L_1C_0\textnormal{ln}\frac{F_0}{\epsilon}, \frac{L_0}{\epsilon}C_0^2
    \},
    $
where $C_0=\sqrt{\frac{1}{9(1-\sqrt{\rho})^2L^2_1}+\textnormal{max}_i\|x^i_0-x^*\|^2}$.
\end{corollary}
According to the conclusion in Corollary~\ref{corollary_1}, we can now explicitly observe the impact of diverse topologies on the complexity bound. The key to the above complexity bound is the clipping step size $\alpha_k\leq \frac{1}{3L_1\textnormal{max}_i\{\|\nabla f^i(\bar{x}_k)\|\}}$, which avoids the adoption of bounded gradient assumption, while in \cite{zeng2018nonconvex}, such an assumption is required. When the topology is \textit{dense}, like a fully connected network, where $\rho$ is smaller, the complexity bound is smaller. On the contrary, if the topology is \textit{sparse}, such as a ring, the complexity bound is larger.
Another observation is that the bound is affected by the maximum initialization error $\|x^i-x^*\|$ among all agents. 
We now turn to $(L_0,L_1)$-smooth nonconvex functions and study its complexity bound.
In nonconvex analysis, we still follow two scenarios as discussed before such that the conclusions from Lemma~\ref{lemma_3} applies.
\begin{theorem}\label{theorem_2} (Nonconvex, deterministic)
    Let Assumptions~\ref{assumption_1} and~\ref{definition_1} hold. Let $\{\bar{x}_k\}$ be defined as in Lemma~\ref{lemma_2} with step size $\alpha_k=\textnormal{min}\{\frac{1}{2L_0},\frac{1}{3L_1\textnormal{max}_i\{\|\nabla f^i(x^i_k)\|\}}, \frac{1}{3L_1\textnormal{max}_i\{\|\nabla f^i(\bar{x}_k)\|\}}\}$. Suppose that $f^i:\mathbb{R}^d\to\mathbb{R}$ is a twice continuously differentiable nonconvex function. Let $x^*$ be an arbitrary solution to the problem in Eq.~\ref{eq_1} and let $F_0:=F(\bar{x}_0)-F^*$. Then, $\textnormal{min}_{0\leq k\leq K}\|\nabla F(\bar{x}_k)\|\leq \epsilon$ for any give $0\leq \epsilon\leq F(\bar{x}_0)$ whenever
$K+1\geq 10\textnormal{max}\{\frac{L_1F_0}{\epsilon},\frac{L_0F_0}{\epsilon^2}\}$.
\end{theorem}
    The complexity bound in Theorem~\ref{theorem_2} still matches the bound in~\cite{vankov2024optimizing} for nonconvex functions, up to absolute constants primarily due to the number of agents in a decentralized setting. Viewing the complexity bound in Theorem~\ref{theorem_2}, we find that it is independent of $R$. However, we can still attain this through the smoothness of $F_0$ and thus include the explicit relationship with the spectral gap $1-\sqrt{\rho}$. An additional result is presented in SM for completeness. Our rate also matches the rates in~\cite{zeng2018nonconvex,tatarenko2017non}, which have different evaluation metrics (e.g., $\frac{1}{K+1}\sum_{k=0}^K\|\frac{1}{N}\sum_{i=1}^N\nabla f^i(x^i_k)\|^2\leq \epsilon$ in Theorem 1~\cite{zeng2018nonconvex} and $\|\bar{x}_k-x^*\|^2$ in Theorem 6~\cite{tatarenko2017non}). Additionally, both of them require the bounded gradient assumption for the analysis, limiting their theory. We would also like to point out due to deep learning models, numerous works established rates for stochastic settings, making the analysis for stochastic settings particularly significant. 
\section{Stochastic Setting}\label{stochastic}
Eq.~\ref{eq_1} shows a generic problem formulation for deterministic decentralized optimization, excluding any stochastic behaviors. In machine learning, due to widespread stochastic first-order optimization methods, the formulation differs primarily in the adoption of data sampling and expectation in the analysis. Thus, we reintroduce the formulation as follows:
\begin{equation}\label{eq_12}
    F^*=\underset{x\in\mathbb{R}^{d}}{\text{min}}\frac{1}{N}\sum_{i=1}^N\mathbb{E}_{\xi_i\sim\mathcal{D}_i}[\mathcal{F}^i(x;\xi_i)],
\end{equation}
where $\xi_i$ represents a data point uniformly sampled from the local dataset $\mathcal{D}_i$ only known to agent $i$ and $\mathcal{F}^i:\mathbb{R}^d\to\mathbb{R}$ is the local loss. Eq.~\ref{eq_12} can be boiled down to Eq.~\ref{eq_1} if we set $f^i(x):=\mathbb{E}_{\xi_i\sim\mathcal{D}_i}[\mathcal{F}^i(x;\xi_i)]$. We present decentralized stochastic gradient descent (DSGD) with $(L_0,L_1)$-smoothness in Algorithm~\ref{alg:dsgd} for completeness. To analyze the convergence rate for DSGD, it is typically required that the stochastic gradient is unbiased such that $\nabla f^i(x) = \mathbb{E}[\nabla\mathcal{F}^i(x;\xi_i)], \forall x\in\mathbb{R}^d$. Due to the stochastic variance and local data diversities among agents, we have the following assumption.
\begin{algorithm}
\caption{$(L_0,L_1)$-Decentralized Stochastic Gradient Descent (DSGD)}\label{alg:dsgd}
\begin{algorithmic}
\Require Mixing matrix $\Pi$, the number of epochs $K$, initialization $x^i_1, i\in\mathcal{V}$, local dataset $\mathcal{D}_i$, step size function $h(\cdot)$ (defined specifically below), $L_0, L_1, \rho, \sigma$ 
\Ensure $x^i_K, i\in\mathcal{V}$
\For{from 1 to $K$}
    \For{each agent $i$}
        \State $\alpha_k=h(L_0,L_1,\|\nabla f^i(\bar{x}_k)\|,\rho, \sigma)$ \Comment{Step size varies between convex and nonconvex}
        \State Calculate the stochastic gradient $g^i_k=\nabla \mathcal{F}^i(x^i_k;\xi_i)$
        \State $x^i_{k+1/2}=\sum_{j\in Nb(i)}\pi_{ij}x^j_k$
        \State $x^i_{k+1}=x^i_{k+1/2}-\alpha_k g^i_k$
    \EndFor
\EndFor
\end{algorithmic}
\end{algorithm}
\begin{assumption}\label{assum_3} There exist $\sigma, \delta > 0$ such that
    a) The variance of stochastic gradient in each agent is bounded: $\|\nabla \mathcal{F}^i(x;\xi)-\nabla f^i(x)\|\leq \sigma, i=1,2,...,N$; b) The gradient diversity is uniformly bounded, i.e., $\frac{1}{N}\sum_{i=1}^N\|\nabla f^i(x)-\nabla F(x)\|^2\leq \delta^2, \forall x, i=1,2,...,N$.
\end{assumption}
The above assumptions have been used in~\cite{li2019convergence,zhang2013communication,stich2018local,stich2018sparsified,yu2019parallel,crawshaw2025complexity,koloskova2020unified}. Though the bounded gradient diversity assumption has been relaxed in~\cite{koloskova2020unified}, in the recent work on federated $(L_0,L_1)$-smooth optimization~\cite{li2024improved}, it is still leveraged for convergence analysis. We also note that such an assumption remains commonplace in the decentralized optimization literature, as it offers a tractable analytical framework and intuitive interpretation, especially in early-stage research. It helps characterize how local heterogeneity impacts global convergence, and thus facilitates understanding of fundamental dynamics before more advanced tools are introduced.
In this context, $\sigma$ signifies the upper bound of variances of stochastic gradients for local agents, while $\delta^2$ quantifies the gradient dissimilarity between each agent's local objective loss $f^i(x)$, given different data distributions. We now present the main result in the following for convex functions.
\begin{theorem}\label{theorem_3} (Convex, stochastic)
    Let Assumptions~\ref{assumption_1}, \ref{definition_1} and~\ref{assum_3} hold and $\{\bar{x}_k\}$ be the iterates produced by Algorithm~\ref{alg:dsgd} with step size $\alpha_k=\textnormal{min}\{\hat{\alpha}, \frac{1}{\textnormal{max}_i\{\|\nabla f^i(\bar{x}_k)\|\}},\frac{1}{\textnormal{max}_i\{\sqrt{24L_1\|\nabla f^i(\bar{x}_k)\|}\}}\}$, where $\hat{\alpha}=\textnormal{min}\{\frac{1-\sqrt{\rho}}{\textnormal{max}\{20L_0,\sqrt{24L_0}\}},\frac{1}{4C_1}, \frac{1}{\sqrt{K}}\}$. Suppose that $f^i:\mathbb{R}^d\to\mathbb{R}$ is a twice continuously differentiable convex function. Let $x^*$ be an arbitrary solution to the problem in Eq.~\ref{eq_12}. Then, $\frac{1}{K}\sum_{k=1}^K\mathbb{E}[F(\bar{x}_k)-F^*]\leq \epsilon$ for any given $0\leq\epsilon$ whenever, 
    $
    K\geq 2\textnormal{max}\{\frac{3R^2}{\epsilon}+\frac{(288L_1^2+216L_1)(\sigma^2+\delta^2)}{(1-\sqrt{\rho})^2\epsilon}+\frac{6\sigma^2}{N\epsilon}, \frac{64R^4}{\epsilon^2}+\frac{53L_0^{4/3}(\sigma^2+\delta^2)^{2/3}}{(1-\sqrt{\rho})^{4/3}\epsilon^{2/3}}+\frac{288L_0(\sigma^2+\delta^2)}{(1-\sqrt{\rho})^2\epsilon}+\frac{64\sigma^4}{N^2\epsilon^2}\},
    $
    where $R:=\|\bar{x}_1-x^*\|, C_1:=1-5L_1-\frac{240L_1^3+180L_1^2}{(1-\sqrt{\rho})^2}>0$. 
\end{theorem}
Theorem~\ref{theorem_3} characterizes the convergence behavior of DSGD for convex $(L_0,L_1)$-smooth functions. Specifically, the first term on the right hand side of the above lower complexity bound shows that if $\|\nabla f^i(\bar{x}_k)\|\geq \frac{L_0}{L_1}$, which likely appears in the early phase of the optimization, DSGD converges with a sublinear rate that is faster than the one reported in~\cite{koloskova2020unified,lian2017can,jiang2017collaborative}. However, when the method approaches the solution ($\|\nabla f^i(\bar{x}_k)\|< L_0/L_1$), the convergence slows down to the standard sublinear rate shown in Theorem 2 in~\cite{koloskova2020unified}. $L_1=0$ degenerates $(L_0,L_1)$-smoothness into the classical one such that our worst-case complexity bound resembles the one in~\cite{koloskova2020unified}. 
We also notice that the condition that $C_1$ needs to satisfy may be strong in some sense. Specifically, to make sure the condition $0<\rho<1$ holds, $L_1$ needs to be a relatively smaller value, resulting in a larger $L_0/L_1$ value. One may question about the applicability of our conclusion. However, $C_1$ is due to the scenario of $\|\nabla f^i(\bar{x}_k)\|\geq L_0/L_1$, where the gradient norm should relatively be larger. Thus, a smaller $L_1$ indeed ensures the validity of the scenario. 
We next state the theorem for nonconvex functions. 
\begin{theorem}\label{theorem_4} (Nonconvex, stochastic)
    Let Assumptions~\ref{assumption_1}, \ref{definition_1} and~\ref{assum_3} hold and $\{\bar{x}_k\}$ be the iterates produced by Algorithm~\ref{alg:dsgd} with step size $\alpha_k=\textnormal{min}\{\hat{\alpha},\frac{1}{2(\frac{3}{2(1-\sqrt{\rho})^2}+5L_0+4L_1\|\nabla F(\bar{x}_k)\|)}, \frac{1}{(L_0+L_1\textnormal{max}_i\{\|\nabla f^i(\bar{x}_k)\|\})^2}\}$, where $\hat{\alpha}=\textnormal{min}\{\frac{1}{36L_1\sigma},\frac{1}{\sqrt{K+1}}\}$. Suppose that $f^i:\mathbb{R}^d\to\mathbb{R}$ is an $(L_0,L_1)$-smooth nonconvex function. Let $x^*$ be an arbitrary solution to the problem in Eq.~\ref{eq_12} and let $F_1:=F(\bar{x}_1)-F^*$. Then, $\textnormal{min}_{1\leq k\leq K}\mathbb{E}[\|\nabla F(\bar{x}_k)\|]\leq \epsilon$ for any given $0\leq \epsilon\leq F(\bar{x}_1)$ whenever 
    $
    K\geq C\textnormal{max}\{\frac{4C}{\epsilon^4}, \frac{144L_1\sigma}{\epsilon^2},\frac{64L_1}{\epsilon}\}$ where $C=F_1+\frac{3(\sigma^2+\delta^2)}{2(1-\sqrt{\rho})^2} + \frac{(5L_0+2L_1\sigma)\sigma^2}{N}$.
\end{theorem}
    We first remark the difference of complexity bounds between ours and those in~\cite{koloskova2020unified,yuan2022revisiting}. In most previous works, when they investigate complexity bounds for nonconvex objective functions, a more generic metric is $\frac{1}{K}\sum_{k=1}^K\mathbb{E}[\|\nabla F(\bar{x}_k)\|^2]\leq \epsilon$ such that it is typically $\mathcal{O}(1/\epsilon^2)$. However, when it is converted to $\mathbb{E}[\|\nabla F(\bar{x}_k)\|]\leq \epsilon$, the complexity still remains $\mathcal{O}(1/\epsilon^4)$. This implies that our worst-case bound matches the best available result in previous works. Additionally, when $L_1=0$, Algorithm~\ref{alg:dsgd} converges to the desired accuracy $\epsilon$ after $K=\mathcal{O}(4C^2/\epsilon^4)$ iterations. We also notice that our complexity bound also matches the one in its centralized counterpart~\cite{zhang2019gradient} for $(L_0,L_1)$-smooth nonconvex functions, without the gradient norm to be bounded globally. One can also impose the bounded gradient assumption to achieve the similar convergence rate, but its absolute constants can be arbitrarily large, making the corresponding algorithm arbitrarily slower than $(L_0,L_1)$-DSGD.

\section{Numerical Results}\label{results}
In this section, we empirically investigate $(L_0,L_1)$-smoothness in decentralized settings. Our primary focus is to illuminate the complexity bounds through theoretical analysis under this new smoothness assumption, rather than to optimize empirical performance. Accordingly, our interest lies in examining the linear relationship between smoothness parameters and the gradient norm, rather than in reporting testing accuracy.

First, we leverage logistic regression (which represents convex case) for a binary classification problem for each agent such that 
$f^i(x)=\textnormal{log}(1+\textnormal{exp}(ya_i^\top x)), \;a_i\in\mathbb{R}^d, \;y\in\{0,1\}$. As Example 1.6 in~\cite{gorbunov2024methods} shows, each logistic regression loss function $f^i$ is $(L_0,L_1)$-smooth, indicating that $F$ is also $(L_0,L_1)$-smooth. However, note that the detailed derivation for the exact constants $L_0$ and $L_1$ is fairly complicated and relies heavily on the relationship between different agent model parameters $a_i$. In this case, to make it feasible, we resort to the Hessian norm (curvature in the plots) and numerically estimate the relationship between the Hessian norm and the gradien norm, as done in~\cite{gorbunov2024methods,vankov2024optimizing}. Though in centralized settings, smoothness has empirically been shown to have the linear dependency on the gradient norm, whether it exhibits the analogous trend in a decentralized setting remains unknown. Therefore, we run $(L_0,L_1)$-DSGD and vanilla DSGD over a real dataset from LIBSVM~\cite{chang2011libsvm} - a9a, with 5 agents in a fully-connected topology. The results are presented in Figure~\ref{fig:comparison_lr}, showing that $F(\bar{x})$ has a clear linear dependence of the Hessian norm on the gradient norm by using gradient clipping, as suggested by the relaxed smoothness condition defined in Eq.~\ref{eq_3}.  
To check the behavior in nonconvex functions, we replace each agent's model with a multilayer perceptron (MLP). Similarly, based on Figure~\ref{fig:comparison_nn}, we observe the smoothness dependency on the gradient norm for the averaged model, though the nonconvexity of the model results in a looser trend of linearity between them. In Figures~\ref{fig:comparison_lr} and~\ref{fig:comparison_nn}, we also observe that without clipping, the smoothness dependency on the gradient norm is weaker than that when using gradient clipping. To validate our theoretical claims in a wider spectrum of datasets (including more complicated ones, e.g., CIFAR10 and CIFAR100), additional results are presented in SM, including training loss, testing accuracy, and performance of individual agents and ring topology.

When examining Figure 2 in~\cite{gorbunov2024methods}, one interesting observation is that though the authors showed smoothness dependency on the gradient norm in a toy scenario logistic regression, the gradient norm increases along with the number of epochs. This seemingly implies that the algorithm exhibits divergence. Likewise, our result in Figure~\ref{fig:comparison_cnn} using a more complicated 2-layer Convolutional Neural Network (CNN) and MNIST dataset~\cite{xiao2017fashion} similarly shows the gradient norm grows when the number of epochs increases. This is likely attributed to the already well-trained model after only one or a few epochs, also possibly signaling the overfitting, but not divergence (through the plots of loss and accuracy in SM). However, with clipping the dependency of smoothness on gradient norm is more significantly linear-like compared to adopting vanilla DGD, emphasizing again the $(L_0,L_1)$-smoothness in complex models.
\begin{figure}[h]
    \centering    \includegraphics[width=0.8\linewidth]{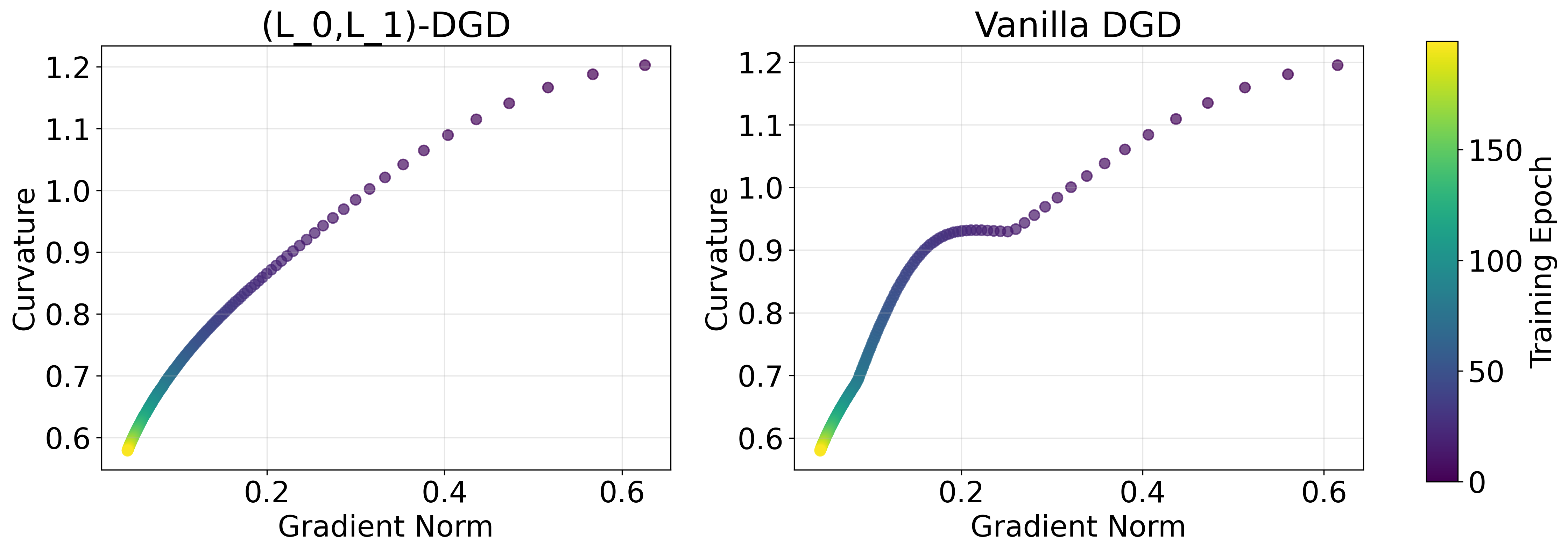}
    \caption{Smoothness vs. gradient norm for decentralized logistic regression with a9a dataset for averaged model}
    \label{fig:comparison_lr}
\end{figure}

\begin{figure}[h]
    \centering
\includegraphics[width=0.8\linewidth]{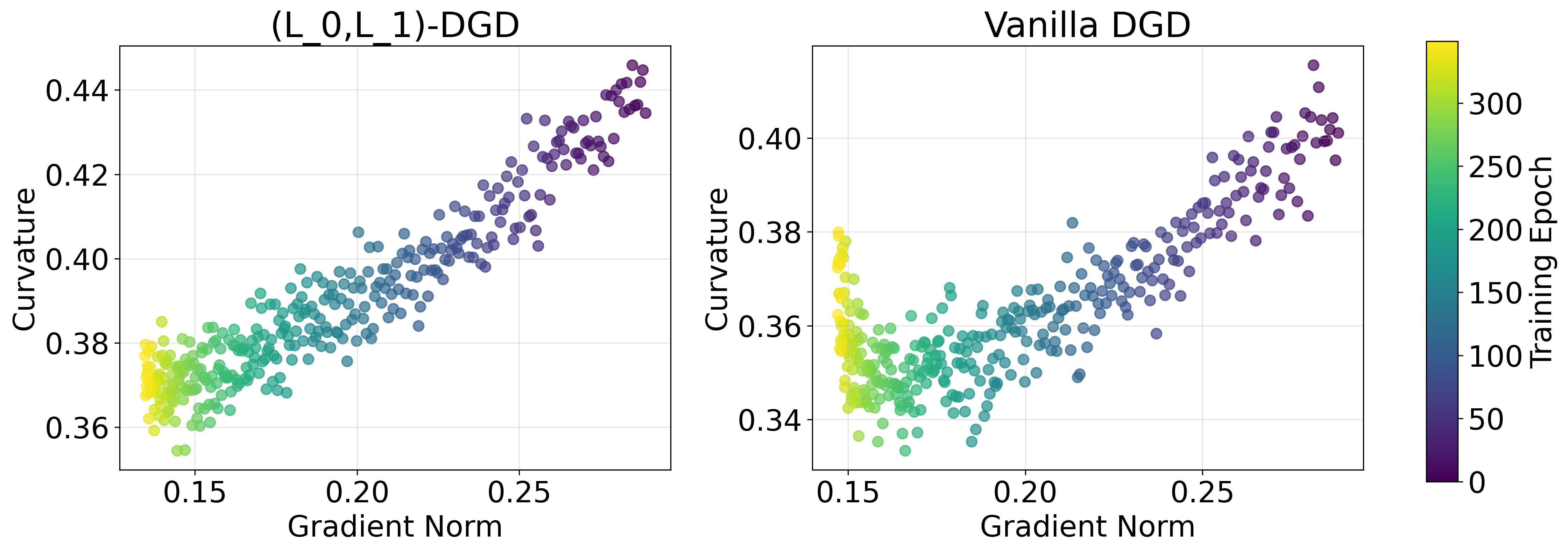}
    \caption{Smoothness vs. gradient norm for decentralized neural networks with a9a dataset for averaged model}
    \label{fig:comparison_nn}
\end{figure}

\begin{figure}[h]
    \centering
\includegraphics[width=0.8\linewidth]{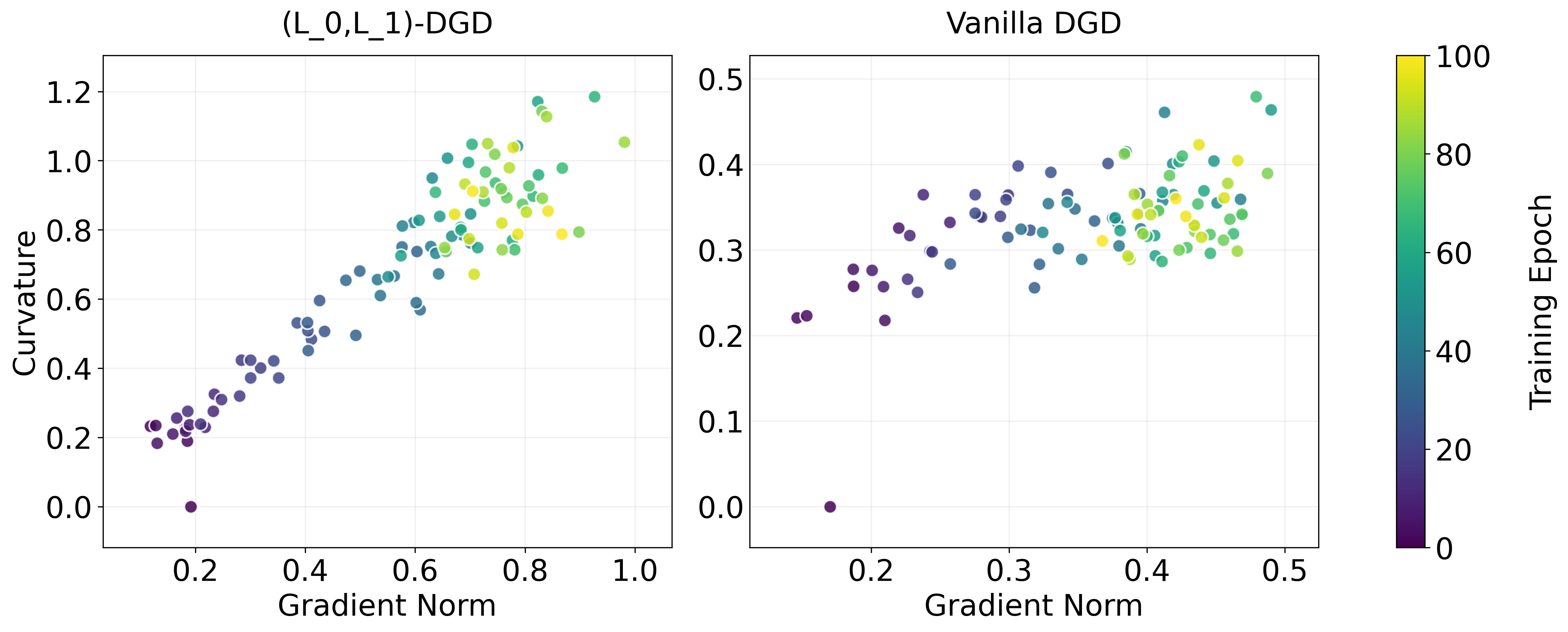}
    \caption{Smoothness vs. gradient norm for decentralized CNN with MNIST dataset for averaged model}
    \label{fig:comparison_cnn}
\end{figure}
\textbf{Limitations.} One of limitation in this work is the missing analysis of comparison between FL and decentralized learning. Though FL can be regarded as one special scenario of decentralized learning, the integration of $(L_0,L_1)$-smoothness into the existing theoretical analysis in FL has led to new and improved results. Another limitation is our theoretical claims have not yet been empirically validated through larger models such as Transformer architectures and even LLMs~\cite{ghiasvand2025decentralized}. However, the present work marks the first step towards the relaxation of traditional smoothness assumption especially in decentralized learning, so the primary focus is on establishing complexity bounds for different function classes. The validation in larger models remains a critical part of our future work.

\section{Conclusion and Impact Statement}\label{conclusion}
This work investigates decentralized gradient methods for $(L_0,L_1)$-smooth optimization problems in deterministic and stochastic settings. In deterministic setting, we have proposed the $(L_0,L_1)$-DGD and established the improved complexity bound $\mathcal{O}(\textnormal{max}\{L_1R\;\text{ln}(F_0/\epsilon),L_0R^2/\epsilon\})$ for convex function with the clipping step size, which resembles the nearly-optimal step size. In the nonconvex setting, we have also achieved the best-known complexity bound for DGD with $\mathcal{O}(\textnormal{max}\{L_0F_0/\epsilon^2,L_1F_0/\epsilon\})$. When extending to the stochastic setting, we have presented the formal convergence guarantee and provided the complexity bounds with $(L_0,L_1)$-smoothness for convex and nonconvex functions, matching the best available bounds with a properly selected step size. To the best of our knowledge, this is the first theoretical result to be established for DSGD. Our extensive empirical results have validated the smoothness dependency on the gradient norm in the decentralized setting. An interesting open question is whether we can improve the complexity bounds for the stochastic setting. Additionally, it remains unknown to us if what the exact bounds are for adaptive gradient descent in decentralized settings. 
The broader impact of this work includes facilitating the theoretical underpinnings of relaxing classical smoothness assumption in decentralized setting and providing constructive insights for designing new optimization algorithms. Particularly, the general analysis framework provided in this work sheds light on how convergence guarantee can be achieved under the relaxed $(L_1,L_0)$-smoothness assumption and the difference compared to existing works, either centralized or decentralized optimization algorithms.

\clearpage
\bibliographystyle{unsrt}
\bibliography{references}
\clearpage

\appendix

\section{Technical Appendices and Supplementary Material}
In this section, we present additional analysis and results as well as missing proof in the main contents. We also include the details of datasets and experimental setup for completeness.
\subsection{Additional analysis in Section~\ref{deterministic}}\label{additional_analysis_deterministic}
\subsubsection{Proof of Lemma~\ref{lemma_1}}
\begin{proof}
    Leveraging the similar proof techniques from the proof for Lemma 2.5 in~\cite{vankov2024optimizing} and applying it to each local loss function $f^i$ completes our proof.
\end{proof}
\subsubsection{Proof of Lemma~\ref{lemma_2}}
\begin{proof}
    Based on the core update law in Algorithm~\ref{alg:dgd}, we know that
    \begin{equation}\label{eq_8}
        \bar{x}_{k+1} = \bar{x}_k-\alpha_k\frac{1}{N}\sum_{i=1}^N\nabla f^i(x^i_k),
    \end{equation}
which is the average of all agents and is because all mixing matrices preserve the average (see Proposition 1 in~\cite{koloskova2020unified}). By setting $y=\bar{x}_{k+1}$ and $x=\bar{x}_k$ in the conclusion of Lemma~\ref{lemma_1}, and equaling $z:=L_0+L_1\|\nabla f^i(\bar{x}_k)\|>0$, we can obtain the following relation:
\begin{equation}\label{eq_9}
    f^i(\bar{x}_{k+1})\leq \underbrace{f^i(\bar{x}_k) +\langle\nabla f^i(\bar{x}_k), \bar{x}_{k+1}-\bar{x}_k\rangle + z\frac{\phi(L_1\|\bar{x}_{k+1}-\bar{x}_k\|)}{L_1^2}}_{T_0}.
\end{equation}
Viewing the above equation and leveraging the classical idea in optimization theory, we know that the gradient method is to choose the next iterate $\bar{x}_{k+1}$ by minimizing the global upper bound on the objective established around the current iterate $\bar{x}_k$. Therefore, our goal is to minimize the right hand side of Eq.~\ref{eq_9} in $\bar{x}_{k+1}$. Through Eq.~\ref{eq_8}, it is immediately obtained that $\bar{x}_{k+1}-\bar{x}_k=\frac{1}{N}\sum_{i=1}^N\nabla f^i(x^i_k)$. Luckily, we aim to derive a relationship between $\bar{x}_k$ and $\bar{x}_{k+1}$ after minimizing $T_0$ such that $\frac{1}{N}\sum_{i=1}^N\nabla f^i(x^i_k)$ will not explicit show in the relationship. As the last term in $T_0$ is only dependent on $\|\bar{x}_{k+1}-\bar{x}_k\|$, the optimal point $\bar{x}_{k+1}=\mathcal{T}(\bar{x}_k)$ is the outcome of the gradient step $\mathcal{T}(\bar{x}_k)=\bar{x}_{k}-\eta^*\frac{\nabla f^i(\bar{x}_k)}{\|\nabla f^i(\bar{x}_k)\|}$ for some $\eta^*\geq 0$ (Please see the auxiliary Lemma~\ref{lemma_5} in Section~\ref{auxiliary_results} for example illustration). This ensures the following progress in decreasing the function value:
\begin{equation}
    f^i(\bar{x}_k)-f^i(\mathcal{T}(\bar{x}_{k}))\geq \underset{r\geq 0}{\textnormal{max}}\bigg\{\|\nabla f^i(\bar{x}_k\eta-\frac{z}{L_1^2}\psi(L_1\eta))\|\bigg\}=\frac{z}{L_1^2}\psi_*\bigg(\frac{L_1\|\nabla f^i(\bar{x}_k)\|}{z}\bigg),
\end{equation}
where $\psi_*$ is the conjugate function~\cite{fenchel2014conjugate} to $\psi$. We can know that $\eta^*$ is exactly the solution to the above optimization problem, satisfying $L_1\|\nabla f^i(\bar{x}_k)\|=a\psi'(L_1\eta^*)$. Solving this equation and using $(\psi')^{-1}(\gamma)=\psi'_*(\gamma)=\textnormal{ln}(1+\gamma)$, we obtain 
\begin{equation}
    \eta^*=\frac{1}{L_1}\psi'_*(\frac{L_1\|\nabla f^i(\bar{x}_k)\|}{z})=\frac{1}{L_1}\textnormal{ln}(1+\frac{L_1\|\nabla f^i(\bar{x}_k)\|}{z}),
\end{equation}
which results in the optimal choice of step size 
\begin{equation}
    \alpha_k^*=\frac{1}{\|\nabla f^i(\bar{x}_k)\|}\textnormal{ln}\bigg(1+\frac{L_1\|\nabla f^i(\bar{x}_k)\|}{L_0+L_1\|\nabla f^i(\bar{x}_k)\|}\bigg).
\end{equation}
We then follow the similar proof for Case (3) from Lemma B.1 in~\cite{vankov2024optimizing} by choosing the clipping step size 
\begin{equation}
    \alpha_k=\textnormal{min}\bigg\{\frac{1}{2L_0},\frac{1}{3L_1\|\nabla f^i(\bar{x}_k)\|}\bigg\},
\end{equation}
which gives us the conclusion. To ease the use of step size for each local agent, we augment it with $\frac{1}{\|\nabla f^i(x^i_k)\|}$ such that it becomes 
\begin{equation}
    \alpha_k=\textnormal{min}\bigg\{\frac{1}{2L_0},\frac{1}{3L_1\|\nabla f^i(\bar{x}_k)\|},\frac{1}{3L_1\|\nabla f^i(x^i_k)\|}\bigg\}.
\end{equation}
Based on $\alpha_k$, we know that it is essentially associated with agent $i$, which makes local updates more complex. To mitigate this issue, we resort to the maximum operator such that it becomes
\begin{equation}
\alpha_k=\textnormal{min}\bigg\{\frac{1}{2L_0},\frac{1}{3L_1\textnormal{max}_i\{\|\nabla f^i(\bar{x}_k)\|\}},\frac{1}{3L_1\textnormal{max}_i\{\|\nabla f^i(x^i_k)\|\}}\bigg\},
\end{equation}
which implies that $\alpha_k\leq \frac{1}{3L_1\|\nabla f^i(\bar{x}_k)\|}$ and $\alpha_k\leq \frac{1}{3L_1\|\nabla f^i(x^i_k)\|}$.
\end{proof}
\subsubsection{Proof of Lemma~\ref{lemma_3}}
\begin{proof}
    According to Lemma~\ref{lemma_2}, for any $k\geq 0$, we have
    \begin{equation}
        f^i(\bar{x}_k)-f^i(\bar{x}_{k+1})\geq \frac{\|\nabla f^i(\bar{x}_k)\|^2}{4L_0+6L_1\|\nabla f^i(\bar{x}_k)\|}.
    \end{equation}
As $\|\nabla f^i(\bar{x}_k)\|\geq \frac{L_0}{L_1}$, the above equation can be rewritten as
\begin{equation}
    f^i(\bar{x}_k)-f^i(\bar{x}_{k+1})\geq \frac{\|\nabla f^i(\bar{x}_k)\|^2}{10L_1\|\nabla f^i(\bar{x}_k)\|}=\frac{\|\nabla f^i(\bar{x}_k)\|}{10L_1}
\end{equation}
Summing the above relationship over all $i\in\mathcal{V}$ and dividing both sides by $N$ produces the following:
\begin{equation}
    F(\bar{x}_k)-F(\bar{x}_{k+1})\geq \frac{\sum_{i=1}^N\|\nabla f^i(\bar{x}_k)\|}{10NL_1}\geq \frac{\|\sum_{i=1}^N\nabla f^i(\bar{x}_k)\|}{10NL_1}=\frac{\|\nabla F(\bar{x}_k)\|}{10L_1},
\end{equation}
Following the proof techniques and having $\|\nabla f^i(\bar{x}_k)\|\leq \frac{L_0}{L_1}$ grants us the following:
    \begin{equation}
        f^i(\bar{x}_k)-f^i(\bar{x}_{k+1})\geq\frac{\|\nabla f^i(\bar{x}_k)\|^2}{10L_0}
    \end{equation}
Summing the above inequality over all $i\in\mathcal{V}$ and dividing both sides by $N$ generates the following:
\begin{equation}
    F(\bar{x}_k)-F(\bar{x}_{k+1})\geq\frac{\sum_{i=1}^N\|\nabla f^i(\bar{x}_k)\|^2}{10NL_0}\geq \frac{\|\frac{1}{N}\sum_{i=1}^N\nabla f^i(\bar{x}_k)\|^2}{10L_0}=\frac{\|\nabla F(\bar{x}_k)\|^2}{10L_0},
\end{equation}
which completes the proof.
\end{proof}
\subsubsection{Proof of Theorem~\ref{theorem_1}}
\begin{proof}
    To show the conclusion in Theorem~\ref{theorem_1}, we will look into two cases that have been discussed in Lemma~\ref{lemma_3}.
    
    \textbf{Case 1: $\|\nabla f^i(\bar{x}_k)\|\geq\frac{L_0}{L_1}$.}
    Define $F_k:=F(\bar{x}_k)-F^*$ and $R_k:=\|\bar{x}_k-x^*\|$. We can know from Lemma~\ref{lemma_6} that $R_k$ is nonincreasing. With this, we can immediately obtain that $R_k\leq R$, where $R:=R_0$. Therefore, recalling the convexity property of $F$, we have
    \begin{equation}
        F(\bar{x}_k) - F^*\leq \langle\nabla F(\bar{x}_k),\bar{x}_k-x^*\rangle\leq \|\nabla F(\bar{x}_k)\|\|\bar{x}_k-x^*\|\leq\|\nabla F(\bar{x}_k)\|R.
    \end{equation}
    The second inequality is due to Cauchy-Schwartz inequality. To ease the analysis, we define $h(\gamma):=\frac{\gamma}{10L_1}$.
    Therefore, combining the conclusion from Lemma~\ref{lemma_3} and the above inequality produces the following:
    \begin{equation}
        F_k-F_{k+1}\geq h(\|\nabla F(\bar{x}_k)\|).
    \end{equation}
    As $h$ is an increasing function, we have
    \begin{equation}\label{eq_a3}
        F_k-F_{k+1}\geq h(\frac{F_k}{R}).
    \end{equation}
    Thus,
    \begin{equation}
        1\leq \frac{F_k-F_{k+1}}{h(\frac{F_k}{R})}\leq \int_{F_{k+1}}^{F_k}\frac{d\tau}{h(\frac{\tau}{R})}=\int_{F_{k+1}}^{F_k}\bigg(\frac{10L_1R}{\tau}\bigg)d\tau=10L_1R\textnormal{ln}\frac{F_k}{F_{k+1}}.
    \end{equation}
    We then sum the above inequalities for all $0\leq k\leq K-1$, obtaining
    \begin{equation}
        K\leq 10L_1R\textnormal{ln}\frac{F_0}{F_K}
    \end{equation}
    Hence, $F_k\leq \epsilon$ whenever
    \begin{equation}
        K\geq 10L_1R\textnormal{ln}\frac{F_0}{\epsilon}.
    \end{equation}

    \textbf{Case 2: $\|\nabla f^i(\bar{x}_k)\|<\frac{L_0}{L_1}$.} We follow the similar analysis techniques and define $h$ slightly different such that $h(\gamma):=\frac{\gamma^2}{10L_0}$. As $h$ is still an increasing function in this case, combining the conclusion from Lemma~\ref{lemma_3} and the convexity property of $F$ will produce the same relationship in Eq.~\ref{eq_a3}. Thus, we have
    \begin{equation}
        1\leq \frac{F_k-F_{k+1}}{h(\frac{F_k}{R})}\leq \int_{F_{k+1}}^{F_k}\frac{d\tau}{h(\frac{\tau}{R})}=\int_{F_{k+1}}^{F_k}\bigg(\frac{10L_0R^2}{\tau^2}\bigg)d\tau=10L_0R^2\bigg(\frac{1}{F_{k+1}}-\frac{1}{F_k}\bigg).
    \end{equation}
    Similarly, summing up these inequalities for all $0\leq k\leq K-1$ and dropping the negative $\frac{1}{F_0}$ term, we have
    \begin{equation}
        K\leq \frac{10L_0R^2}{F_K},
    \end{equation}
    which yields $K\geq \frac{10L_0R^2}{\epsilon}$ for us. Taking the maximum between them completes the proof.
\end{proof}
\subsubsection{Proof of Corollary~\ref{corollary_1}}
\begin{proof}
    As we have
    \begin{equation}
        R^2=\|\bar{x}_0-x^*\|^2=\|\bar{x}_0-x^i_0+x^i_0-x^*\|^2\leq 2(\|\bar{x}_0-x^i_0\|^2+\|x^i_0-x^*\|^2)
    \end{equation}
    The last inequality follows from the fact that $\|a+b\|^2\leq 2(\|a\|^2+\|b\|^2)$. We now study the upper bound for $\|\bar{x}_0-x^i_0\|^2$, which is the typical consensus estimate in most distributed optimization works. According to Lemma~4.2 (bounded deviation from mean) in~\cite{berahas2018balancing}, we can immediately obtain that
    \begin{equation}
        \|\bar{x}_0-x^i_0\|^2\leq\frac{\alpha^2_0D^2}{(1-\sqrt{\rho})^2}.
    \end{equation}
    $D$ typically is a bounded constant for the gradient norm $\|\nabla f^i(x)\|$. However, due to the step size we have adopted in this work satisfying the condition $\alpha_k=\textnormal{min}\bigg\{\frac{1}{2L_0},\frac{1}{3L_1\textnormal{max}_i\{\|\nabla f^i(\bar{x}_k)\|\}},\frac{1}{3L_1\textnormal{max}_i\{\|\nabla f^i(x^i_k)\|\}}\bigg\}$, there is no necessity for us to impose such an assumption. Therefore, we can obtain the following relationship
    \begin{equation}
        \|\bar{x}_0-x^i_0\|^2\leq\frac{1}{9L_1^2(1-\sqrt{\rho})^2}.
    \end{equation}
    Therefore, $R^2$ is bounded by the following
    \begin{equation}
        R^2\leq 2\bigg(\frac{1}{9L_1^2(1-\sqrt{\rho})^2}+\textnormal{max}_i\|x^i_0-x^*\|^2\bigg).
    \end{equation}
    Substituting the above upper bound into the conclusion from Theorem~\ref{theorem_1} grants us the desirable result.
\end{proof}
\subsubsection{Proof of Theorem~\ref{theorem_2}}
\begin{proof}
    Similar to the proof we have done for convex function, the proof for nonconvex function also involves two cases. 

    \textbf{Case 1: $\|\nabla f^i(\bar{x}_k)\|\geq\frac{L_0}{L_1}$.} Analogously, we can obtain $F_k-F_{k+1}\geq h(\frac{F_k}{R})$. Summing up these inequalities from all $0\leq k\leq K$ and denoting $\nabla F^*_K=\underset{0\leq k\leq K}{\textnormal{min}}\|\nabla F(\bar{x}_k)\|$, we can attain the following relationship
    \begin{equation}
        F_0\geq F_0-F_K\geq\sum_{k=0}^Kh(\|\nabla F(\bar{x}_k)\|)\geq (K+1)h(\nabla F^*_K).
    \end{equation}
    The last inequality is due to an increasing function $h$. We denote by $h^{-1}$ the inverse function of $h$ such that we have
    \begin{equation}
        \nabla F^*_K\leq h^{-1}\bigg(\frac{F_0}{K+1}\bigg)\leq \epsilon.
    \end{equation}
    whenever $\frac{F_0}{K+1}\leq h(\epsilon)$.
    This implies that $K+1\leq\frac{F_0}{h(\epsilon)}$. Equivalently, we have $K+1\geq\frac{10F_0L_1}{\epsilon}$.

    \textbf{Case 2: $\|\nabla f^i(\bar{x}_k)\|<\frac{L_0}{L_1}$.} Likewise, we apply the same analysis from the last scenario in this scenario and obtain that $K+1\geq \frac{10F_0L_0}{\epsilon^2}$. 

    Combining the above two cases and taking the maximum between them completes the proof.
\end{proof}
\subsubsection{Corollary~\ref{corollary_2} for Complexity with Consensus}
\begin{corollary}\label{corollary_2}
    Let Assumptions~\ref{assumption_1} and~\ref{definition_1} hold. With Theorem~\ref{theorem_2}, we have $\textnormal{min}_{0\leq k\leq K}\|\nabla F(\bar{x}_k)\|\leq \epsilon$ whenever
    $
    K\geq 15\textnormal{max}\{\frac{L_0L_1C^2_0}{\epsilon}, \frac{L^2_0C^2_0}{\epsilon^2}
    \},
    $ where $C_0=\sqrt{\frac{1}{9(1-\sqrt{\rho})^2L^2_1}+\textnormal{max}_i\|x^i_0-x^*\|^2}$.
\end{corollary}
\begin{proof}
    From Theorem~\ref{theorem_2}, to explicitly show the relationship between the complexity and the spectral gap, we needs to upper bound $F_0$, by connecting $F(\bar{x}_0)-F(x^*)$ with $R:=\|\bar{x}_0-x^*\|$. Assumption~\ref{definition_1} implies that 
    \begin{equation}
        \|\nabla f^i(y)-\nabla f^i(x)\|\leq (L_0+L_1\textnormal{sup}_{u\in[y,x]}\|\nabla f^i(u)\|)\|y-x\|.
    \end{equation}
    Based on Lemma 2.1 in~\cite{gorbunov2024methods}, we know the above inequality implies the following relationship:
    \begin{equation}
        f^i(y)\leq f^i(x)+\langle\nabla f^i(x),y-x\rangle + \frac{L_0+L_1\|\nabla f^i(x)\|}{2}e^{L_1\|y-x\|}\|y-x\|^2.
    \end{equation}
    Letting $y=\bar{x}_k$, $x=x^*$ and leveraging $\|y-x\|\leq\frac{1}{L}$ yields the following:
    \begin{equation}
        f^i(\bar{x}_k)-f^i(x^*)\leq \frac{L_0e}{2}\|\bar{x}_k-x^*\|^2
    \end{equation}
    With this in hand, we have
    \begin{equation}
        F(\bar{x}_0)-F^*\leq\frac{3L_0\frac{1}{N}\sum_{i=1}^N\|\bar{x}_0-x^*\|^2}{2}.
    \end{equation}
    Hence, we have $F_0\leq \frac{3L_0C_0}{2}$. Substituting this into the conclusion in Theorem~\ref{theorem_2} completes the proof.
\end{proof}
\subsection{Additional analysis in Section~\ref{stochastic}}
In this section, we primarily cover the missing proof in Section~\ref{stochastic}. The analysis is non-trivial, which requires multiple technical auxiliary lemmas that will be presented in Section~\ref{auxiliary_results}. We start with the proof for Theorem~\ref{theorem_3}.
\subsection{Proof of Theorem~\ref{theorem_3}}
In our analysis, to consider multiple agents, we will have the expanded mixing matrix of the Kronecker product between $\Pi$ and $I_d$, $\mathbf{P}=\Pi\otimes I_d$, but the magnitudes of eigenvalues of $\mathbf{P}$ remain the same through a known result presented in the sequel and the fact that eigenvalue of $I_d$ is 1. 
\begin{theorem}\cite{schacke2004kronecker}\label{kronecker_prod}
    Let $\mathbf{C}\in\mathbb{R}^{N\times N}$ and $\mathbf{D}\in\mathbb{R}^{d\times d}$, with eigenvalue $\lambda\in s(\mathbf{C})$ with corresponding eigenvector $x\in\mathbb{C}^{N}$, and $\mu\in s(\mathbf{D})$ with corresponding eigenvector $y\in\mathbb{C}^{d}$, where $s(\cdot)$ signifies the spectrum of a matrix. Then $\lambda\mu$ is an eigenvalue of $\mathbf{C}\otimes \mathbf{D}$ with corresponding eigenvector $x\otimes y\in\mathbb{C}^{dN}$. Any eigenvalue of $\mathbf{C}\otimes \mathbf{D}$ arises as such a product of eigenvalues of $\mathbf{C}$ and $\mathbf{D}$.
\end{theorem}
One immediate outcome from Theorem~\ref{kronecker_prod} is that $\textnormal{max}\{|\lambda_2(\mathbf{P})|, |\lambda_{dN}(\mathbf{P})|\}\leq \sqrt{\rho}<1$.

\begin{proof}
Recalling the update law for average iterate as follows, we have
\begin{equation}\label{eq_46}
\begin{split}
    \|\bar{x}_{k+1}-x^*\|^2&=\|\bar{x}_k-\alpha_k\frac{1}{N}\sum_{i=1}^N g^i_k-x^*\|^2\\&=\|\bar{x}_k-x^*-\alpha_k\frac{1}{N}\sum_{i=1}^N \nabla f^i(x^i_k)+\alpha_k\frac{1}{N}\sum_{i=1}^N \nabla f^i(x^i_k)-\alpha_k\frac{1}{N}\sum_{i=1}^N g^i_k\|^2\\&
    =\|\bar{x}_k-x^*-\alpha_k\frac{1}{N}\sum_{i=1}^N \nabla f^i(x^i_k)\|^2 + \alpha^2_k\|\frac{1}{N}\sum_{i=1}^N \nabla f^i(x^i_k)-\frac{1}{N}\sum_{i=1}^N g^i_k\|^2\\&+\frac{2\alpha_k}{N}\langle\bar{x}_k-x^*-\frac{\alpha_k}{N}\sum_{i=1}^N\nabla f^i(x^i_k),\sum_{i=1}^N\nabla f^i(x^i_k)-\sum_{i=1}^Ng^i_k\rangle
\end{split}
\end{equation}
The last term is zero in expectation. as $\mathbb{E}[g^i_k]=\nabla f^i(x^i_k)$. The second term can be bounded as
\begin{equation}\label{eq_47}
    \alpha^2_k\mathbb{E}[\|\frac{1}{N}\sum_{i=1}^N \nabla f^i(x^i_k)-\frac{1}{N}\sum_{i=1}^N g^i_k\|^2]\leq\frac{\alpha^2_k\sigma^2}{N}
\end{equation}
We next analyze the first term. The first term can be rewritten as
\begin{equation}
    \|\bar{x}_k-x^*-\alpha_k\frac{1}{N}\sum_{i=1}^N \nabla f^i(x^i_k)\|^2=\|\bar{x}_k-x^*\|^2+\alpha^2_k\|\frac{1}{N}\sum_{i=1}^N\nabla f^i(x^i_k)\|^2 - 2\alpha_k\langle\bar{x}_k-x^*,\frac{1}{N}\sum_{i=1}^N\nabla f^i(x^i_k)\rangle
\end{equation}
We upper bound $\|\frac{1}{N}\sum_{i=1}^N\nabla f^i(x^i_k)\|^2$ in the next.
\begin{equation}
    \begin{split}
        \|\frac{1}{N}\sum_{i=1}^N\nabla f^i(x^i_k)\|^2&= \|\frac{1}{N}\sum_{i=1}^N(\nabla f^i(x^i_k)-\nabla f^i(\bar{x}_k)+\nabla f^i(\bar{x}_k)-\nabla f^i(x^*))\|^2\\&\leq \frac{2}{N}\sum_{i=1}^N\|\nabla f^i(x^i_k)-\nabla f^i(\bar{x}_k)\|^2+\frac{2}{N}\sum_{i=1}^N\|\nabla f^i(\bar{x}_k)-\nabla f^i(x^*)\|^2\\&\leq \frac{2}{N}\sum_{i=1}^N(L_0+L_1\|\nabla f^i(\bar{x}_k)\|)^2\|x^i_k-\bar{x}_k\|^2\\&+\frac{2}{N}\sum_{i=1}^N(2L_0+3L_1\|\nabla f^i(\bar{x}_k)\|)(f^i(\bar{x}_k)-f^i_*)
    \end{split}
\end{equation}
The second inequality follows from $\|a+b\|^2\leq 2\|a\|^2+2\|b\|^2$ and Corollary~\ref{corollary_3} in SM~\ref{auxiliary_results}.
In the following, we bound $2\alpha_k\langle\bar{x}_k-x^*,\frac{1}{N}\sum_{i=1}^N\nabla f^i(x^i_k)\rangle$.
\begin{equation}
    \begin{split}
        &-\frac{1}{\alpha_k}\cdot 2\alpha_k\langle\bar{x}_k-x^*,\frac{1}{N}\sum_{i=1}^N\nabla f^i(x^i_k)\rangle=\frac{-2}{N}\sum_{i=1}^N[\langle\bar{x}_k-x^i_k,\nabla f^i(x^i_k)\rangle+\langle x^i_k-x^*,\nabla f^i(x^i_k)\rangle]
    \end{split}
\end{equation}
According to the convexity, $(L_0,L_1)$-smoothness condition and the constraint $\|\bar{x}_k-x^i_k\|\leq\frac{1}{L_1}$, we have
\begin{equation}
    \begin{split}
        &-\frac{1}{\alpha_k}\cdot 2\alpha_k\langle\bar{x}_k-x^*,\frac{1}{N}\sum_{i=1}^N\nabla f^i(x^i_k)\rangle\\&\leq \frac{-2}{N}\sum_{i=1}^N[f^i(\bar{x}_k)-f^i(x^i_k)-\frac{e}{2}(L_0+L_1\|\nabla f^i(x^i_k\|)\|\bar{x}_k-x^i_k\|^2+f^i(x^i_k)-f^i_*]\\&\leq -2(F(\bar{x}_k)-F^*)+\frac{1}{N}\sum_{i=1}^N 3(L_0+L_1\|\nabla f^i(x^i_k)\|)\|\bar{x}_k-x^i_k\|^2
    \end{split}
\end{equation}
We can obtain that
\begin{equation}
    -2\alpha_k\langle\bar{x}_k-x^*,\frac{1}{N}\sum_{i=1}^N\nabla f^i(x^i_k)\rangle\leq -2(F(\bar{x}_k)-F^*)+\frac{1}{N}\sum_{i=1}^N 3(L_0+L_1\|\nabla f^i(x^i_k)\|)\|\bar{x}_k-x^i_k\|^2.
\end{equation}
Therefore, the following can be attained:
\begin{equation}
    \begin{split}
        &\|\bar{x}_k-x^*-\alpha_k\frac{1}{N}\sum_{i=1}^N \nabla f^i(x^i_k)\|^2\leq \|\bar{x}_k-x^*\|^2\\&+\alpha^2_k\bigg[\frac{2}{N}\sum_{i=1}^N(L_0+L_1\|\nabla f^i(\bar{x}_k)\|)^2\|x^i_k-\bar{x}_k\|^2\\&+\frac{2}{N}\sum_{i=1}^N(2L_0+3L_1\|\nabla f^i(\bar{x}_k)\|)(f^i(\bar{x}_k)-f^i_*)\bigg]\\&-2\alpha_k(F(\bar{x}_k)-F^*)+\frac{\alpha_k}{N}\sum_{i=1}^N 3(L_0+L_1\|\nabla f^i(x^i_k)\|)\|\bar{x}_k-x^i_k\|^2
    \end{split}
\end{equation}
Substituting the above equation and Eq.~\ref{eq_47} into Eq.~\ref{eq_46} yields the following:
\begin{equation}\label{eq_54}
    \begin{split}
        \|\bar{x}_{k+1}-x^*\|^2&\leq \|\bar{x}_k-x^*\|^2\\&+\alpha^2_k\bigg[\frac{2}{N}\sum_{i=1}^N(L_0+L_1\|\nabla f^i(\bar{x}_k)\|)^2\|x^i_k-\bar{x}_k\|^2\\&+\frac{2}{N}\sum_{i=1}^N(2L_0+3L_1\|\nabla f^i(\bar{x}_k)\|)(f^i(\bar{x}_k)-f^i_*)\bigg]\\&-2\alpha_k(F(\bar{x}_k)-F^*)+\frac{\alpha_k}{N}\sum_{i=1}^N 3(L_0+L_1\|\nabla f^i(x^i_k)\|)\|\bar{x}_k-x^i_k\|^2 + \frac{\alpha^k\sigma^2}{N}
    \end{split}
\end{equation}
To continue the analysis, we divide it into two scenarios we have discussed before. However, we extend slightly such that $\textnormal{max}\{\|\nabla f^i(\bar{x}_k)\|,\|\nabla f^i(x^i_k)\|\}\geq \frac{L_0}{L_1}$ and $\textnormal{max}\{\|\nabla f^i(\bar{x}_k),\|\nabla f^i(x^i_k)\|\}< \frac{L_0}{L_1}$.

\textbf{Case 1: $\textnormal{max}\{\|\nabla f^i(\bar{x}_k)\|,\|\nabla f^i(x^i_k)\|\}\geq \frac{L_0}{L_1}$.}
Eq.~\ref{eq_54} can be rewritten as
\begin{equation}
    \begin{split}
        \|\bar{x}_{k+1}-x^*\|^2&\leq\|\bar{x}_k-x^*\|^2+\frac{1}{N}\sum_{i=1}^N(8\alpha_kL_1^2\|\nabla f^i(\bar{x}_k)\|+6\alpha_kL_1\|\nabla f^i(\bar{x}_k)\|)\|\bar{x}_k-x^i_k\|^2\\&+\frac{2}{N}\sum_{i=1}^N(5L_1\|\nabla f^i(\bar{x}_k)\|\alpha_k^2-\alpha_k)(f^i(\bar{x}_k)-f^i_*)+\frac{\alpha^2_k\sigma^2}{N}
    \end{split}
\end{equation}
As $\alpha_k\leq\frac{1}{\|\nabla f^i(\bar{x}_k)\|}$, the above inequality can be rewritten as
\begin{equation}\label{eq_56}
\begin{split}
    \|\bar{x}_{k+1}-x^*\|^2&\leq\|\bar{x}_k-x^*\|^2+\frac{1}{N}\sum_{i=1}^N(8L_1^2+6L_1)\|\bar{x}_k-x^i_k\|^2+\frac{2}{N}\sum_{i=1}^N(5L_1\alpha_k-\alpha_k)(f^i(\bar{x}_k)-f^i_*)\\&+\frac{\alpha^2_k\sigma^2}{N}
\end{split}
\end{equation}
Recall the conclusion in Lemma~\ref{lemma_8} such that
\begin{equation}\label{eq_57}
\begin{split}
    &\sum_{k=1}^K\frac{1}{N}\sum_{i=1}^N\|x^i_k-\bar{x}_k\|^2\leq \sum_{k=1}^{K}\frac{6\alpha^2_k(\sigma^2+\delta^2)}{(1-\sqrt{\rho})^2}+\sum_{k=1}^K\frac{60\alpha_kL_1}{(1-\sqrt{\rho})^2}(F(\bar{x}_k)-F^*)
\end{split}
\end{equation}
We sum Eq.~\ref{eq_56} over $\{1,...,K\}$, obtaining
\begin{equation}\label{eq_58}
    \begin{split}
        \sum_{k=1}^K\|\bar{x}_{k+1}-x^*\|^2&\leq\sum_{k=1}^K\|\bar{x}_{k}-x^*\|^2 + (8L_1^2+6L_1)\sum_{k=1}^K\frac{1}{N}\sum_{i=1}^N\|\bar{x}_k-x^i_k\|^2\\&+2\sum_{k=1}^K(5L_1\alpha_k-\alpha_k)(F(\bar{x}_k)-F^*)+\sum_{k=1}^K\frac{\alpha^2_k\sigma^2}{N}
    \end{split}
\end{equation}
We then substitute Eq.~\ref{eq_57} into Eq.~\ref{eq_58} to acquire the following:
\begin{equation}
    \begin{split}
        \sum_{k=1}^K\|\bar{x}_{k+1}-x^*\|^2&\leq\sum_{k=1}^K\|\bar{x}_{k}-x^*\|^2 \\& + (8L_1^2+6L_1)\bigg[\sum_{k=1}^{K}\frac{6\alpha^2_k(\sigma^2+\delta^2)}{(1-\sqrt{\rho})^2}+\sum_{k=1}^K\frac{60\alpha_kL_1}{(1-\sqrt{\rho})^2}(F(\bar{x}_k)-F^*)\bigg]\\&+2\sum_{k=1}^K(5L_1\alpha_k-\alpha_k)(F(\bar{x}_k)-F^*)+\sum_{k=1}^K\frac{\alpha^2_k\sigma^2}{N}
    \end{split}
\end{equation}
Rearranging the above inequality yields the next inequality:
\begin{equation}
    \begin{split}
        \sum_{k=1}^K\|\bar{x}_{k+1}-x^*\|^2&\leq\sum_{k=1}^K\|\bar{x}_{k}-x^*\|^2 \\& +(8L_1^2+6L_1)\sum_{k=1}^{K}\frac{6\alpha^2_k(\sigma^2+\delta^2)}{(1-\sqrt{\rho})^2}\\&+\sum_{k=1}^K\bigg(\frac{(8L_1^2+6L_1)60\alpha_kL_1}{(1-\sqrt{\rho})^2}+10L_1\alpha_k-2\alpha_k\bigg)(F(\bar{x}_k)-F^*)\\&+\sum_{k=1}^K\frac{\alpha^2_k\sigma^2}{N}
    \end{split}
\end{equation}
Based on the condition defined for step size $\alpha_k$, we know that $\bigg(\frac{(8L_1^2+6L_1)60\alpha_kL_1}{(1-\sqrt{\rho})^2}+10L_1\alpha_k-2\alpha_k\bigg)\leq -\frac{1}{2}$, such that
\begin{equation}
    \begin{split}
        \frac{1}{2}\sum_{k=1}^K(F(\bar{x}_k)-F^*)&\leq \sum_{k=1}^K\|\bar{x}_{k}-x^*\|^2-\sum_{k=1}^K\|\bar{x}_{k+1}-x^*\|^2+\\&(8L_1^2+6L_1)\sum_{k=1}^{K}\frac{6\alpha^2_k(\sigma^2+\delta^2)}{(1-\sqrt{\rho})^2}+\sum_{k=1}^K\frac{\alpha^2_k\sigma^2}{N}
    \end{split}
\end{equation}
Dividing the last equation on both sides by $\frac{1}{K}$, applying that $\alpha_k\leq\frac{1}{\sqrt{K}}$, and taking the expectation yields the following:
\begin{equation}
    \frac{1}{K}\sum_{k=1}^K\mathbb{E}[(F(\bar{x}_k)-F^*)]\leq\frac{2\|\bar{x}_1-x^*\|^2}{K}+\frac{(96L_1^2+72L_1)(\sigma^2+\delta^2)}{(1-\sqrt{\rho})^2K}+\frac{2\sigma^2}{NK}
\end{equation}

\textbf{Case 2: $\textnormal{max}\{\|\nabla f^i(\bar{x}_k)\|,\|\nabla f^i(x^i_k)\|\}< \frac{L_0}{L_1}$.}
In this case, Eq.~\ref{eq_54} is rewritten as
\begin{equation}
    \begin{split}
        \|\bar{x}_{k+1}-x^*\|^2&\leq \|\bar{x}_k-x^*\|^2+\frac{1}{N}\sum_{i=1}^N\bigg(8L_0^2\alpha^2_k+6\alpha_kL_0\bigg)\|\bar{x}_k-x^i_k\|^2\\&(10L_0\alpha^2_k-2\alpha_k)(F(\bar{x}_k)-F^*) + \frac{\alpha^2_k\sigma^2}{N}
    \end{split}
\end{equation}
Summing last inequality over $\{1,...,K\}$ yields the following:
\begin{equation}
    \begin{split}
        \sum_{k=1}^K\|\bar{x}_{k+1}-x^*\|^2&\leq\sum_{k=1}^K\|\bar{x}_{k}-x^*\|^2 +\sum_{k=1}^K\frac{1}{N}\sum_{i=1}^N\bigg(8L_0^2\alpha^2_k+6\alpha_kL_0\bigg)\|\bar{x}_k-x^i_k\|^2\\&+
        \sum_{k=1}^K(10L_0\alpha^2_k-2\alpha_k)(F(\bar{x}_k)-F^*)+\sum_{k=1}^K\frac{\alpha^2_k\sigma^2}{N}
    \end{split}
\end{equation}
According to Lemma~\ref{lemma_8}, the above inequality can be rewritten as:
\begin{equation}
    \begin{split}
        \sum_{k=1}^K\|\bar{x}_{k+1}-x^*\|^2&\leq\sum_{k=1}^K\|\bar{x}_{k}-x^*\|^2 + \sum_{k=1}^K\bigg(8L_0^2\alpha^2_k+6\alpha_kL_0\bigg)\frac{6\alpha^2_k(\sigma^2+\delta^2)}{(1-\sqrt{\rho})^2}\\&+\sum_{k=1}^K\bigg(8L_0^2\alpha^2_k+6\alpha_kL_0\bigg)\frac{60\alpha^2_kL_0}{(1-\sqrt{\rho})^2}(F(\bar{x}_k)-F^*)\\&+\sum_{k=1}^K(10L_0\alpha^2_k-2\alpha_k)(F(\bar{x}_k)-F^*)+\sum_{k=1}^K\frac{\alpha^2_k\sigma^2}{N}
    \end{split}
\end{equation}
We rewrite the last inequality as follows:
\begin{equation}
    \begin{split}
        \sum_{k=1}^K\|\bar{x}_{k+1}-x^*\|^2&\leq\sum_{k=1}^K\|\bar{x}_{k}-x^*\|^2 +\sum_{k=1}^K\bigg(8L_0^2\alpha^2_k+6\alpha_kL_0\bigg)\frac{6\alpha^2_k(\sigma^2+\delta^2)}{(1-\sqrt{\rho})^2}\\&+\sum_{k=1}^K\bigg(\bigg(8L_0^2\alpha^2_k+6\alpha_kL_0\bigg)\frac{60\alpha^2_kL_0}{(1-\sqrt{\rho})^2}+10L_0\alpha^2_k-2\alpha_k\bigg)(F(\bar{x}_k)-F^*)\\&+\sum_{k=1}^K\frac{\alpha^2_k\sigma^2}{N}
    \end{split}
\end{equation}
As $\alpha_k\leq\frac{1-\sqrt{\rho}}{20L_0}$, then we have $\bigg(8L_0^2\alpha^2_k+6\alpha_kL_0\bigg)\frac{60\alpha^2_kL_0}{(1-\sqrt{\rho})^2}+10L_0\alpha^2_k-2\alpha_k\leq -\frac{1}{2}\alpha_k$. Then, the above inequality becomes 
\begin{equation}
    \begin{split}
        \sum_{k=1}^K\frac{1}{2}\alpha_k(F(\bar{x}_k)-F^*)&\leq\sum_{k=1}^K\|\bar{x}_{k}-x^*\|^2-\sum_{k=1}^K\|\bar{x}_{k+1}-x^*\|^2\\&+\bigg(8L_0^2\alpha^2_k+6\alpha_kL_0\bigg)\frac{6\alpha^2_k(\sigma^2+\delta^2)}{(1-\sqrt{\rho})^2}+\sum_{k=1}^K\frac{\alpha^2_k\sigma^2}{N}
    \end{split}
\end{equation}
We divide both sides by $\frac{1}{2}\alpha_k$ and $K$, apply the condition $\alpha_k\leq\frac{1}{\sqrt{K}}$ and take the expectation to obtain the following relationship
\begin{equation}
\begin{split}
    \frac{1}{K}\sum_{k=1}^K\mathbb{E}[F(\bar{x}_k)-F^*]&\leq \frac{2\|\bar{x}_1-x^*\|^2}{\sqrt{K}}+\frac{96L_0^2(\sigma^2+\delta^2)}{K^{3/2}(1-\sqrt{\rho})^2}+\frac{72L_0(\sigma^2+\delta^2)}{K(1-\sqrt{\rho})^2}+\frac{2\sigma^2}{N\sqrt{K}}.
\end{split}
\end{equation}
As the desired accuracy is $\epsilon$, combining both cases and conduct some simple mathematical manipulation can attain the desirable results.
\end{proof}
\subsection{Proof for Theorem~\ref{theorem_4}}
\begin{proof}
    We first parameterize the path between $\bar{x}_{k+1}$ and $\bar{x}_k$ as follows:
    \begin{equation}
        l(t) = t(\bar{x}_{k+1}-\bar{x}_k) + \bar{x}_k, \; \forall t\in[0,1]
    \end{equation}
    Since $\bar{x}_{k+1}=\bar{x}_k-\alpha_k\frac{1}{N}\sum_{i=1}^Ng^i_k$, resorting to Taylor's Theorem, the Triangle Inequality, and Cauthy-Schwartz Inequality, we have the following relationship
    \begin{equation}
        F(\bar{x}_{k+1})\leq F(\bar{x}_k) + \alpha_k\langle\nabla F(\bar{x}_k),\frac{1}{N}\sum_{i=1}^Ng^i_k\rangle+\frac{\|\bar{x}_k-\bar{x}_k\|^2}{2}\int_{0}^1\|\nabla^2F(l(t))\|dt.
    \end{equation}
    As $\|\nabla F(l(t))\|\leq 4(\frac{L_0}{L_1}+\|\nabla F(\bar{x}_k)\|)$ based on Lemma~\ref{lemma_10}, with the fact that $F(x)$ is $(L_0,L_1)$-smooth, we can obtain the following descent inequality, taking expectation on both sides:
    \begin{equation}\label{eq_15}
        \mathbb{E}[F(\bar{x}_{k+1})]\leq F(\bar{x}_k) -\underbrace{\mathbb{E}[\alpha_k\langle\nabla F(\bar{x}_k),\frac{1}{N}\sum_{i=1}^Ng^i_k\rangle]}_{T_1}+\underbrace{\frac{5L_0+4L_1\|\nabla F(\bar{x}_k)\|}{2}\mathbb{E}[\alpha_k^2\|\frac{1}{N}\sum_{i=1}^Ng^i_k\|^2]}_{T_2}.
    \end{equation}
    Different from the analysis in centralized setting, our goal in decentralized setting is to find out the complexity bound for $K$ when $\underset{0\leq k\leq K}{\textnormal{min}}\|\nabla F(\bar{x})_k\|\leq \epsilon$. However, due to the distinction among $\nabla F(\bar{x}_k), \frac{1}{N}\sum_{i=1}^Ng^i_k$, and $\frac{1}{N}\sum_{i=1}^N\nabla f^i(x^i_k)$, existing techniques for centralized settings in~\cite{vankov2024optimizing,gorbunov2024methods,lobanov2024linear} cannot be directly applied. Therefore, we need to bound separately $T_1$ and $T_2$. 
    Based on Lemma~\ref{lemma_11}, we can obtain the lower bound for $T_1$ is as follows:
    \begin{equation}\label{eq_16}
    \begin{split}
        T_1&\geq \frac{1}{2}\mathbb{E}[\alpha_k\|\nabla F(\bar{x}_k)\|^2]+\frac{1}{2}\mathbb{E}[\alpha_k\|\frac{1}{N}\sum_{i=1}^N\nabla f^i(x^i_k)\|^2]\\&-\frac{1}{2}\mathbb{E}[\alpha_k\frac{1}{N}\sum_{i=1}^N(L_0+L_1\|\nabla f^i(\bar{x}_k)\|)^2\|x^i_k-\bar{x}_k\|^2].
    \end{split}
    \end{equation}
    As $\alpha_k\leq\frac{1}{(L_0+L_1\|\nabla f^i(\bar{x}_k)\|)^2}$, we have
    \begin{equation}\label{eq_17}
        T_1\geq \frac{1}{2}\mathbb{E}[\alpha_k\|\nabla F(\bar{x}_k)\|^2]+\frac{1}{2}\mathbb{E}[\alpha_k\|\frac{1}{N}\sum_{i=1}^N\nabla f^i(x^i_k)\|^2]-\frac{1}{2}\underbrace{\mathbb{E}[\frac{1}{N}\sum_{i=1}^N\|x^i_k-\bar{x}_k\|^2]}_{T_3}.
    \end{equation}
    We next bound $T_3$. Based on Eq.~\ref{eq_43} (in Lemma~\ref{lemma_8}) in the analysis for the convex case, it is obtained that:
    \begin{equation}\label{eq_18}
        T_3\leq \frac{3(\sigma^2+\delta^2)\alpha_k^2}{(1-\sqrt{\rho})^2}+\frac{3\alpha_k^2}{(1-\sqrt{\rho})^2}\mathbb{E}[\|\frac{1}{N}\sum_{i=1}^N\nabla f^i(x^i_k)\|^2].
    \end{equation}
    We now analyze $T_2$. According to Lemma~\ref{lemma_12}, the following relationship is obtained:
    \begin{equation}\label{eq_19}
        T_2\leq \frac{(5L_0+2L_1\sigma)\alpha_k^2\sigma^2}{N}+\frac{1}{4}\mathbb{E}[\|\nabla F(\bar{x}_k)\|^2\alpha_k] + (5L_0+4L_1\|\nabla F(\bar{x}_k)\|)\mathbb{E}[\alpha_k^2\|\frac{1}{N}\sum_{i=1}^N\nabla f^i(x^i_k)\|^2]
    \end{equation}
    Substituting Eq.~\ref{eq_18} into Eq.~\ref{eq_17}, and substituting Eqs.~\ref{eq_17},~\ref{eq_19} into Eq.~\ref{eq_15} yield the following relationship:
    \begin{equation}\label{eq_20}
        \begin{split}
            \mathbb{E}[F(\bar{x}_{k+1})]&\leq F(\bar{x}_k) - \frac{1}{2}\mathbb{E}[\alpha_k\|\nabla F(\bar{x}_k)\|^2]-\frac{1}{2}\mathbb{E}[\alpha_k\|\frac{1}{N}\sum_{i=1}^N\nabla f^i(x^i_k)\|^2]\\&+\frac{3(\sigma^2+\delta^2)\alpha_k^2}{2(1-\sqrt{\rho})^2}+\frac{3}{2(1-\sqrt{\rho})^2}\mathbb{E}[\alpha_k^2\|\frac{1}{N}\sum_{i=1}^N\nabla f^i(x^i_k)\|^2]\\&+\frac{(5L_0+2L_1\sigma)\alpha_k^2\sigma^2}{N}+\frac{1}{4}\mathbb{E}[\|\nabla F(\bar{x}_k)\|^2\alpha_k]\\&+(5L_0+4L_1\|\nabla F(\bar{x}_k)\|)\mathbb{E}[\alpha_k^2\|\frac{1}{N}\sum_{i=1}^N\nabla f^i(x^i_k)\|^2].
        \end{split}
    \end{equation}
    As $\frac{1}{2}-\frac{3}{2(1-\sqrt{\rho})^2}\alpha_k-(5L_0+4L_1\|\nabla F(\bar{x}_k)\|)\alpha_k\geq 0$, Eq.~\ref{eq_20} can be rewritten as:
    \begin{equation}
        \begin{split}
            \mathbb{E}[F(\bar{x}_{k+1})]&\leq F(\bar{x}_k)- \frac{1}{4}\mathbb{E}[\alpha_k\|\nabla F(\bar{x}_k)\|^2]+\frac{3(\sigma^2+\delta^2)\alpha_k^2}{2(1-\sqrt{\rho})^2}+\frac{(5L_0+2L_1\sigma)\alpha_k^2\sigma^2}{N}.
        \end{split}
    \end{equation}
    The rest of the proof follows similarly from the proof of Theorem 7 in~\cite{zhang2019gradient}. 
\end{proof}

\subsection{Experimental Setup and Additional Results}
\subsubsection{Dataset}
In this section, we provide the additional information for all the datasets we use in this work.

\textbf{a9a:} This is a dataset for classification problem, which involves 2 classes. The number of training data samples is 32561 and the number of testing data samples is 16281. It has 123 features. The data is for training logistic regression and multilayer perceptron (MLP) models. The data source link is: \texttt{https://www.csie.ntu.edu.tw/~cjlin/libsvmtools/datasets/binary.html}.

\textbf{MNIST.} This is a handwritten digit dataset involves 10 classes (from 0 to 9). The number of training samples is 60000 and the number of testing samples is 10000. The link to access the data is: \texttt{http://yann.lecun.com/exdb/mnist/}.

\textbf{Breast cancer.} This is classification dataset. Features are computed from a digitized image of a fine needle aspirate (FNA) of a breast mass. They describe characteristics of the cell nuclei present in the image. Class distribution: 357 benign, 212 malignant. The number of data samples is 569 and the number features is 30. The link to access the data is: \texttt{https://archive.ics.uci.edu/dataset/17/breast+cancer+wisconsin+diagnostic}.

\textbf{FashionMNIST.} Fashion-MNIST is a dataset of Zalando's article images—consisting of a training set of 60000 examples and a test set of 10000 examples. Each example is a 28x28 grayscale image, associated with a label from 10 classes. The access to this data: \texttt{https://github.com/zalandoresearch/fashion-mnist}.

\textbf{KMNIST.} Kuzushiji-MNIST is a drop-in replacement for the MNIST dataset (28x28 grayscale, 70,000 images), provided in the original MNIST format as well as a NumPy format. KMNIST has 10 classes as well. The link to access the data is: \texttt{https://github.com/rois-codh/kmnist?tab=readme-ov-file}. 

\textbf{Mushrooms.} This data set includes descriptions of hypothetical samples corresponding to 23 species of gilled mushrooms in the Agaricus and Lepiota Family (pp. 500-525). It has 8124 data samples and 22 features. The link to access the data is: \texttt{https://archive.ics.uci.edu/dataset/73/mushroom}. 

\textbf{Covtype.} This dataset regarding classification of pixels into 7 forest cover types based on attributes such as elevation, aspect, slope, hillshade, soil-type, and more. The number of instances is 581012 and the number of classes is 7. There are 54 features. The link to access the data and more information is \texttt{https://archive.ics.uci.edu/dataset/31/covertype}.

\textbf{IJCNN1.} This dataset is from IJCNN 2001 neural network competition. The number of classes is 2. The training data includes 49990 samples while the testing data includes 91701 samples. The number of features is also 22. The link to access the data and more information is \texttt{https://www.csie.ntu.edu.tw/~cjlin/libsvmtools/datasets/binary.html\#ijcnn1}. 

\textbf{KDD CUP 99.} This is the data set used for The Third International Knowledge Discovery and Data Mining Tools Competition, which was held in conjunction with KDD-99 The Fifth International Conference on Knowledge Discovery and Data Mining. The competition task was to build a network intrusion detector, a predictive model capable of distinguishing between ``bad" connections, called intrusions or attacks, and ``good" normal connections. This database contains a standard set of data to be audited, which includes a wide variety of intrusions simulated in a military network environment. There are 4000000 instances. The description for the data files is here: \texttt{https://archive.ics.uci.edu/dataset/130/kdd+cup+1999+data}.

\textbf{CIFAR10.} The CIFAR10 dataset consists of 60000 32x32 colour images in 10 classes, with 6000 images per class. There are 50000 training images and 10000 test images. The link to access the data and more information is: \texttt{https://www.cs.toronto.edu/\~kriz/cifar.html}

\textbf{CIFAR100.} The CIFAR100 dataset has 100 classes containing 600 images each. There are 50000 training images and 10000 testing images totally. The 100 classes in the CIFAR-100 are grouped into 20 superclasses. Each image comes with a ``fine" label (the class to which it belongs) and a ``coarse" label (the superclass to which it belongs). The link to access the data and more information is: \texttt{https://www.cs.toronto.edu/\~kriz/cifar.html}

\textbf{Beijing PM2.5.} This is an hourly dataset that contains the PM2.5 data of US Embassy in Beijing. Meanwhile, meteorological data from Beijing Capital International Airport are also included. The number of instances is 43824 and the number features is 11. The data's time period is between Jan 1st, 2010 to Dec 31st, 2014. Missing data are denoted as ``NA". The link to access the data is \texttt{https://archive.ics.uci.edu/dataset/381/beijing+pm2+5+data}.

\textbf{Individual Household Electric Power Consumption.} This dataset contains measurements of electric power consumption in one household with a one-minute sampling rate over a period of almost 4 years. Different electrical quantities and some sub-metering values are available. The number of instances is 2075259 and the number of features is 9. This dataset is gathered in a house located in Sceaux (7km of Paris, France) between December 2006 and November 2010 (47 months). The link to access the data is \texttt{https://archive.ics.uci.edu/dataset/235/individual+household+electric+power+consumption}.

\subsubsection{Experimental Setup}
In this section, we provide the detailed information about the models and hyperparamter setup in our experiments to expedite the reproducibility of our results. We also detail the topologies in each experiment. For the hyperparameter setting, we manually tune and obtain the nearly-optimal performance. 

\textbf{Experiment Set 1: Decentralized logistic regression.} In this experiment, there are 5 agents respectively for fully-connected and ring topologies. The learning rate is 0.1 and the clipping threshold for gradient clipping is 1.0. The number of epochs is 200. We split the dataset uniformly into 5 agents. For comparison, we run $(L_0,L_1)$-DGD and vanilla DGD algorithms on \texttt{a9a, breast cancer, mushrooms, covtype, ijcnn1, and kddcup99} datasets. 

\textbf{Experiment Set 2: Decentralized Neural Network.} In this experiment, there are 5 agents respectively for fully-connected and ring topologies. The learning rate is 0.01 and the clipping threshold for gradient clipping is 1.0. The number of epochs is 350. We split the dataset uniformly into 5 agents. For comparison, we run $(L_0,L_1)$-DGD and vanilla DGD algorithms on \texttt{a9a, breast cancer, covtype, ijcnn1, and kddcup99} datasets. The model is a ultilayer perceptron (MLP) with 2 hidden layers and the activation function is \texttt{tanh}. The hidden dimension is 32 (a9a) or 64 (breast cancer). The batch size for the experiment is 16 (a9a) or 32 (breast cancer). For all other datasets, we increase the batch size to 512 due to their larger sizes of data, while the hidden dimension remains 32.

\textbf{Experiment Set 3: Decentralized Convlutional Neural Networks.} In this experiment, there are 5 agents respectively for fully-connected and ring topologies. The learning rate is 0.1 and the clipping threshold for gradient clipping is 1.0. The number of epochs is 100 (while we also adjust slightly if the dataset is large and complex, e.g., 80 for CIFAR10 and CIFAR100). We split the dataset uniformly into 5 agents. For comparison, we run $(L_0,L_1)$-DGD and vanilla DGD algorithms on \texttt{MNIST, FashionMNIST, KMNIST, CIFAR10, and CIFAR100} datasets. The batch size is 64. The model is CNN with 2 convolutional layers and 1 dropout layer (0.5). The output channels in the first convolutional layer is 32 while 64 in the second. The activation function is ReLU. Each convolutional layer is followed by a max pooling layer. There are two additional linear layers. The first linear layer has (1600, 128) for the input and output while the second has (128, 10). The output probability is through a log softmax function.

\textbf{Experiment Set 4: Decentralized Gated Recurrent Unit.} All the above experiment sets are specifically for classification. However, another popular task is regression such that in this work we use a popular standard deep recurrent network, gated recurrent unit. for implementing two regression tasks. In this experiment, there are 5 agents for fully connected topologies. The learning rate is set 0.1, and the batch size is 64. The model is 1-layer Gated Recurrent Unit (GRU), with the hidden size being 64.
The number of training epochs 350 and the clipping threshold still remains 1.0. The input sequence length is 24 and the length of output is 1. For comparison, we run $(L_0,L_1)$-DGD and vanilla DGD algorithms on \texttt{Beijing PM2.5 and household power consumption} datasets.

Note that our paper focuses primarily on shedding light on the complexity bounds in terms of theoretical analysis with the new $(L_0,L_1)$-smoothness assumption in the decentralized setting. That is the reason why we resort to relatively smaller models for validation. Additionally, we are primarily interested in the linear dependency between the smoothness and gradient norm, instead of the testing accuracy. 

\subsubsection{Hardware Compute Environment}
The hardware compute environment for all experiments is \texttt{Intel(R) Core(TM) i7-10750H CPU @ 2.60GHz   2.59 GHz} and the installed RAM is 32.0 GB (31.8 GB usable).

\subsubsection{Additional Results}
In this context, curvature, smoothness, and the Hessian norm are referred to as the same meaning in Assumption~\ref{definition_1}. We interchangeably use these three notations in the result discussion.

\textbf{Loss and Accuracy.} In this section, we include additional empirical results to validate the relationship between the smoothness and the gradient norm. We also examine the training loss and testing accuracy by comparing $(L_0,L_1)$-DGD and vanilla DGD. Figure~\ref{fig:loss_accuracy_a9a_nn} shows the training loss and testing accuracy corresponding to the experiment in Figure~\ref{fig:comparison_nn}. We can observe that $(L_0,L_1)$-DGD outperforms vanilla DGD. Also, the training loss shows that with $(L_0,L_1)$-smoothness, the convergence resembles the sublinear rate as vanilla DGD, well validating our theoretical claim. We also use another dataset breast cancer~\cite{agarap2018breast} to test the performance between $(L_0,L_1)$-DGD and vanilla DGD. The comparison in Figure~\ref{fig:loss_accuracy_breast_cancer_nn} implies the significance of gradient clipping, which assists in stabilizing the testing performance remarkably. Figure~\ref{fig:loss_accuracy_mnist_cnn} also empirically demonstrates the similar training loss and testing accuracy between $(L_0,L_1)$-DGD and vanilla DGD with the decentralized CNN and the MNIST dataset, showing the decent model performance, though we know that the former enables more significantly linear depenedency between the Hessian norm and gradient norm in Figure~\ref{fig:comparison_cnn}.

\textbf{Impact of Topology.} To investigate the impact of different topologies on the relationship between the smoothness and the gradient norm, we also implement the experiments on ring topology. Figures~\ref{fig:comparison_lr_ring_a9a}-\ref{fig:comparison_cnn_ring_minist} show empirical results on different models and datasets. Overall, the linear trend is much weaker than that in the fully-connected topology. Particularly, for logistic regression and CNN model, the negative impact is the most significant. One observation is that the values of Hessian norms are much larger, which is likely attributed to under-trained models in each case and more different data distributions. The connection in ring topology is sparser than that in fully-connected topology, demanding more time for each local model to learn. This is theoretically reflected by a larger complexity bounds due to a smaller spectral gap in our claims. However, by carefully comparing $(L_0,L_1)$-DGD and vanilla DGD, the former still results in slightly stronger linear dependency of the smoothness on the gradient norm.

\textbf{More Datasets.} To validate the dependency in more challenging datasets, we resort to two more datasets, including FashionMNIST~\cite{xiao2017fashion} and KMNIST~\cite{clanuwat2018deep} and test on CNN models. Results in Figures~\ref{fig:comparison_cnn_fashionmnist} and~\ref{fig:comparison_cnn_kmnist} still reveal the stronger linear dependency of the smoothness on the gradient norm when using $(L_0,L_1)$-DGD. To better empirically validate the linear dependency between the Hessian norm and the gradient norm in decentralized settings, we resort to a few more datasets, including Covtype, IJCNN1 and KDD CUP 99~\cite{chang2011libsvm}. Figure~\ref{fig:average_comparison_lr_covtype}-Figure~\ref{fig:average_comparison_nn_kddcup99} show the relationships between the Hessian norm and the gradient norm. In particular, with Covtype data, when using logistic regression, the linear dependency in Figure~\ref{fig:average_comparison_lr_covtype} is much weaker, and we can observe that the Hessian norm is large, even if the gradient norm remains smaller values. This is likely because the logistic regression for each local agent is not a good model, failing to capture the complex and highly nonlinear patterns in the data. When the model becomes a neural network, as shown in Figure~\ref{fig:average_comparison_nn_covtype}, the linear dependency is notably significant in both $(L_0,L_1)$-DGD and Vanilla DGD. When turning to the performance with IJCNN1, both Figure~\ref{fig:average_comparison_lr_ijcnn} and Figure~\ref{fig:average_comparison_nn_ijcnn} display nearly linear dependency, whereas the values of Hessian norm in Figure~\ref{fig:average_comparison_lr_ijcnn} are much larger, implying that logistic regression may be incapable of modeling the data well. A similar observation is also attained from Figure~\ref{fig:average_comparison_lr_kddcup99} and Figure~\ref{fig:average_comparison_nn_kddcup99}. Other two more complex datasets CIFAR10 and CIFAR 100 are also used for testing if the linear dependency holds true when the task becomes more difficult. Based on Figures~\ref{fig:average_comparison_cnn_cifar10} and~\ref{fig:average_comparison_cnn_cifar100}, it is evidently shown that $(L_0,L_1)$-DGD results in slightly more significant linear dependency than vanilla DGD, but the results stress again the variation of the smoothness constant along with the gradient norm. Though the linear dependency has been well validated via multiple datasets already, all of the above tasks are classification-based. We next turn to regression task and examine if such an dependency is still valid. Figure~\ref{fig:average_comparison_gru_pm2.5} and Figure~\ref{fig:average_comparison_gru_power_com} present the performance of decentralized GRU models on two time-series datasets, including PM2.5 and household power consumption prediction. Compared to the previous classification tasks, though both tasks do not show clearly the linear dependency between the Hessian norm and gradient norm, overall, $(L_0,L_1)$-DGD leads to slightly stronger linear relationship than Vanilla DGD. The weaker dependency may be attributed to the temporal correlations in the data, which can influence the local smoothness. 

\textbf{Performance of Individual Agents.} We also look into the dependency of each local model in multiple experiments with the fully-connected topology. In Figures~\ref{fig:comparison_nn_individual_clip}, \ref{fig:comparison_lr_breast_cancer_individual}, and~\ref{fig:comparison_lr_mushroom_individual}, each local agent clearly shows the linear-dependency between Hessian norm and gradient norm. Particularly, when the model is convex, such a dependency is more remarkable since the models are easy to learn. Finally, by comparing Figures~\ref{fig:comparison_nn_individual_clip} and \ref{fig:comparison_nn_individual_non_clip} for each local agent with different optimization methods, it also emphasizes again the strong dependency in $(L_0,L_1)$-DGD. To summarize, the additional results have further validated the smoothness dependency on the gradient norm in decentralized setting and supported our theoretical claims.

\textbf{Impact of Scalability.} When increasing the number of agents in a topology, we wonder if the linear dependency between the Hessian norm and the gradient norm can still maintain. Thus, we implement experiments on 10 and 20 agents. As shown in Figure~\ref{fig:average_comparison_lr_a9a_10} and Figure~\ref{fig:average_comparison_lr_a9a_20}, an immediate observation is that increasing the number of agents does not change the linear dependency, though with more agents, the convergence can be harder to achieve. When the local model becomes a neural network, surprisingly, the linear dependency exhibits a stronger trend when more agents are involved based on Figure~\ref{fig:average_comparison_nn_a9a_10} and Figure~\ref{fig:average_comparison_nn_a9a_20}. However, when the classification task becomes more difficult such as in Figure~\ref{fig:average_comparison_cnn_mnist_10} and Figure~\ref{fig:average_comparison_cnn_mnist_20}, the linear dependency becomes weaker when the number of agents is larger. Also, $(L_0,L_1)$-DGD still shows the linear smoothness dependency on the gradient norm. 

\begin{figure}
    \centering
    \includegraphics[width=1\linewidth]{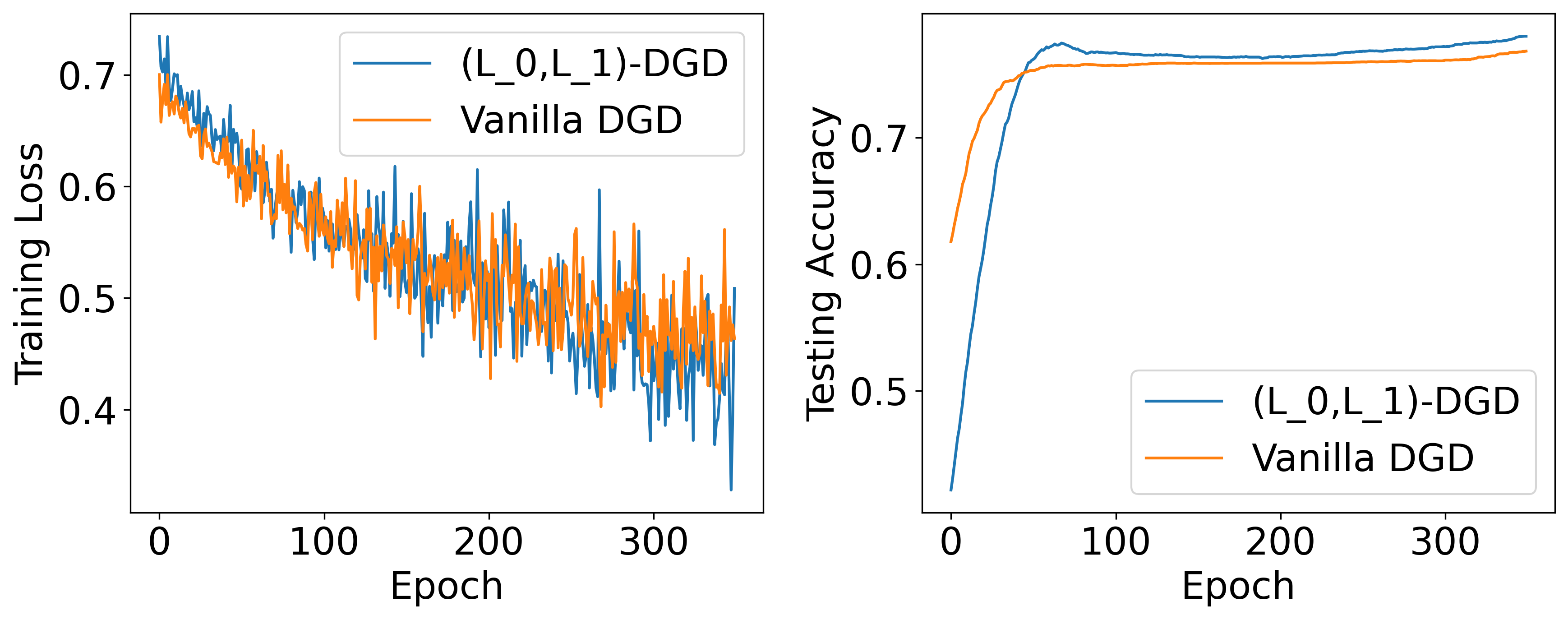}
    \caption{Training loss and testing accuracy for decentralized neural networks with a9a dataset with different methods with a fully-connected topology}
    \label{fig:loss_accuracy_a9a_nn}
\end{figure}

\begin{figure}
    \centering
    \includegraphics[width=1\linewidth]{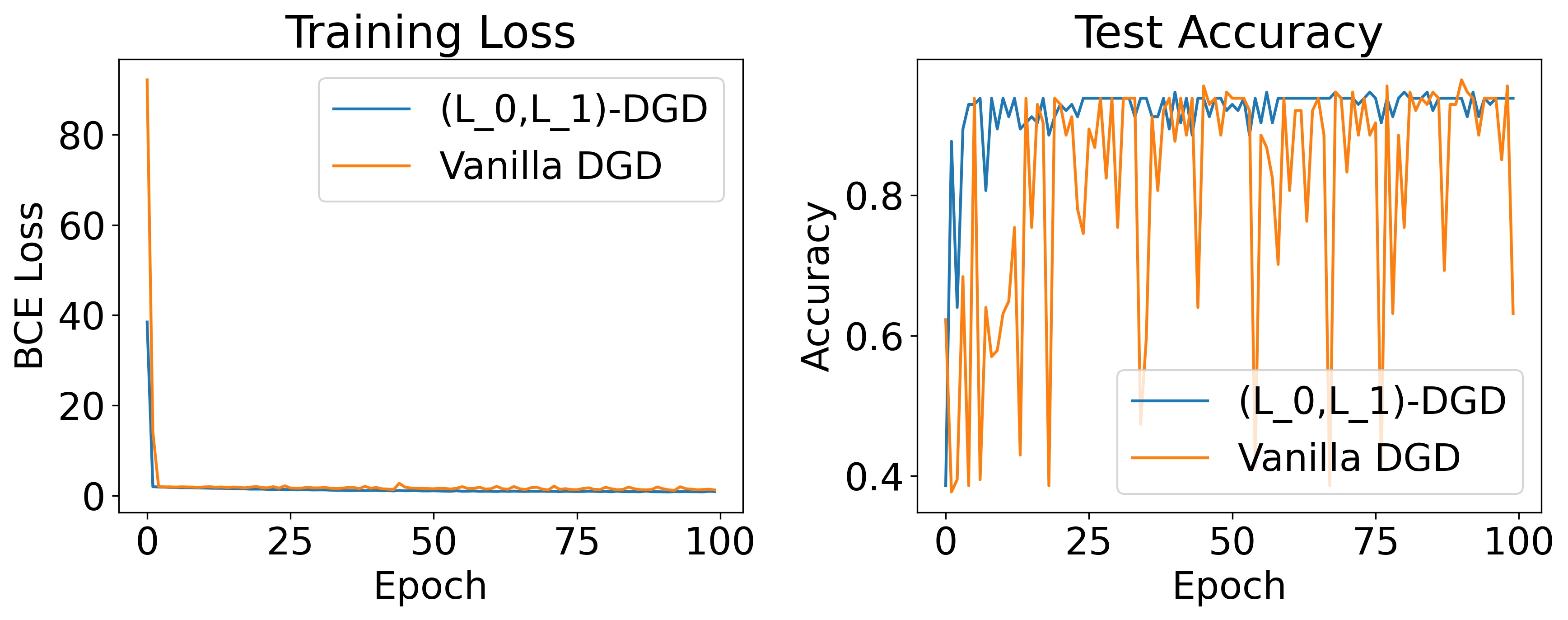}
    \caption{Training loss and testing accuracy for decentralized neural networks with breast cancer dataset~\cite{agarap2018breast} with different methods with a fully-connected topology}
    \label{fig:loss_accuracy_breast_cancer_nn}
\end{figure}

\begin{figure}
    \centering
    \includegraphics[width=1\linewidth]{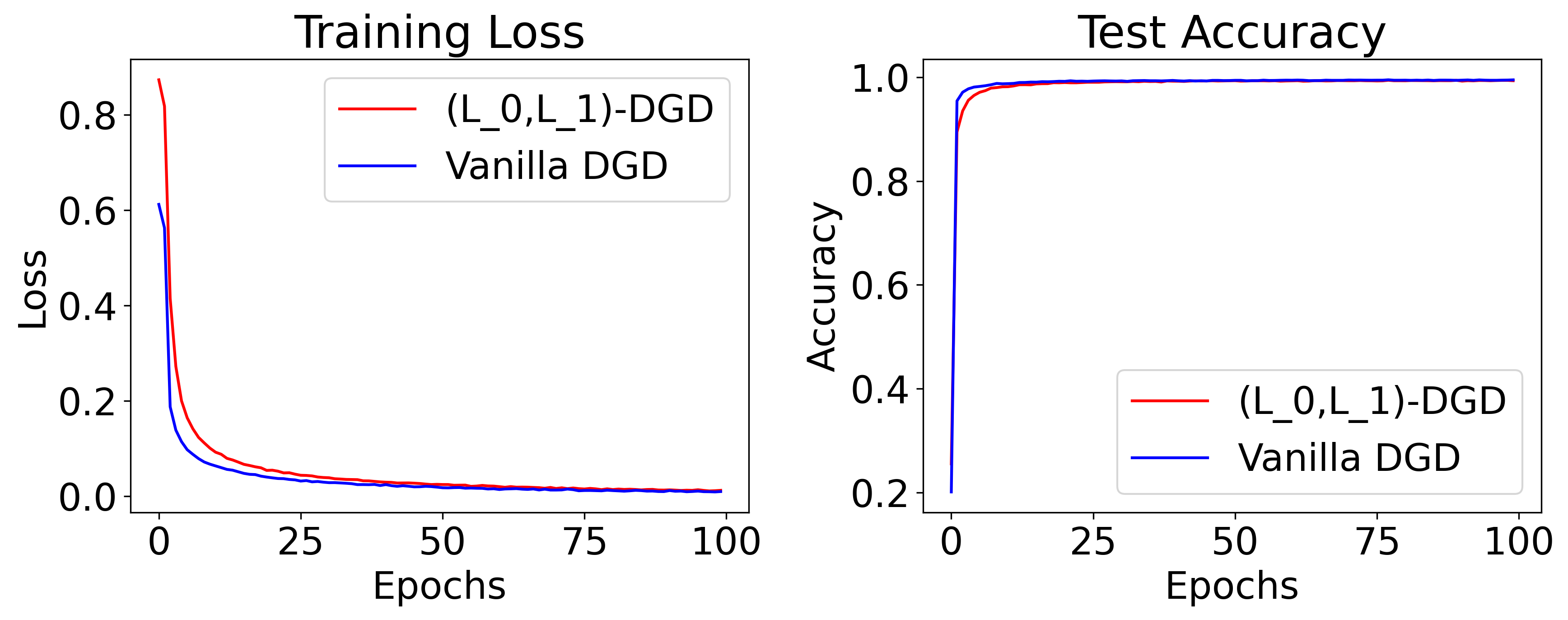}
    \caption{Training loss and testing accuracy for decentralized CNN with MNIST dataset with different methods with a fully-connected topology}
    \label{fig:loss_accuracy_mnist_cnn}
\end{figure}

\begin{figure}
    \centering
    \includegraphics[width=1\linewidth]{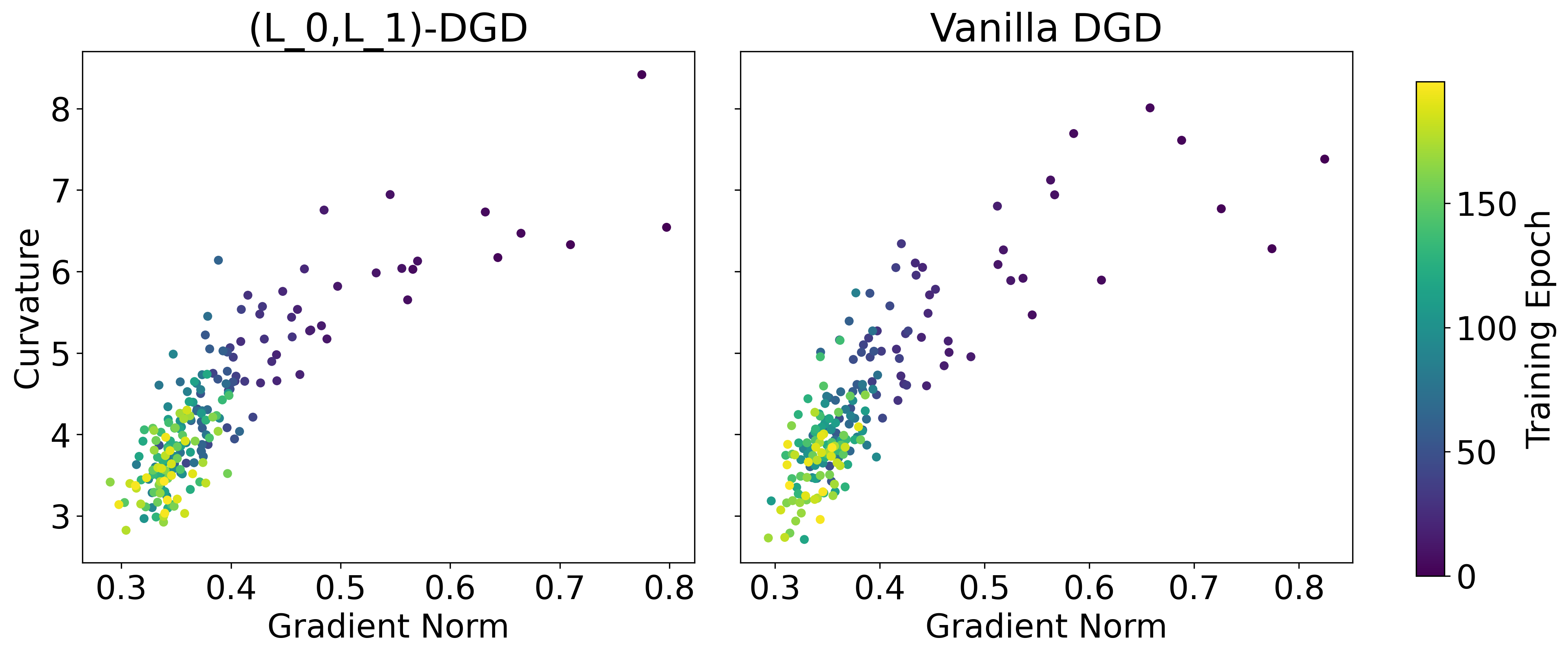}
    \caption{Smoothness dependency on the gradient norm for decentralized logistic regression with a9a dataset for averaged model with a ring topology}
    \label{fig:comparison_lr_ring_a9a}
\end{figure}

\begin{figure}
    \centering
    \includegraphics[width=1\linewidth]{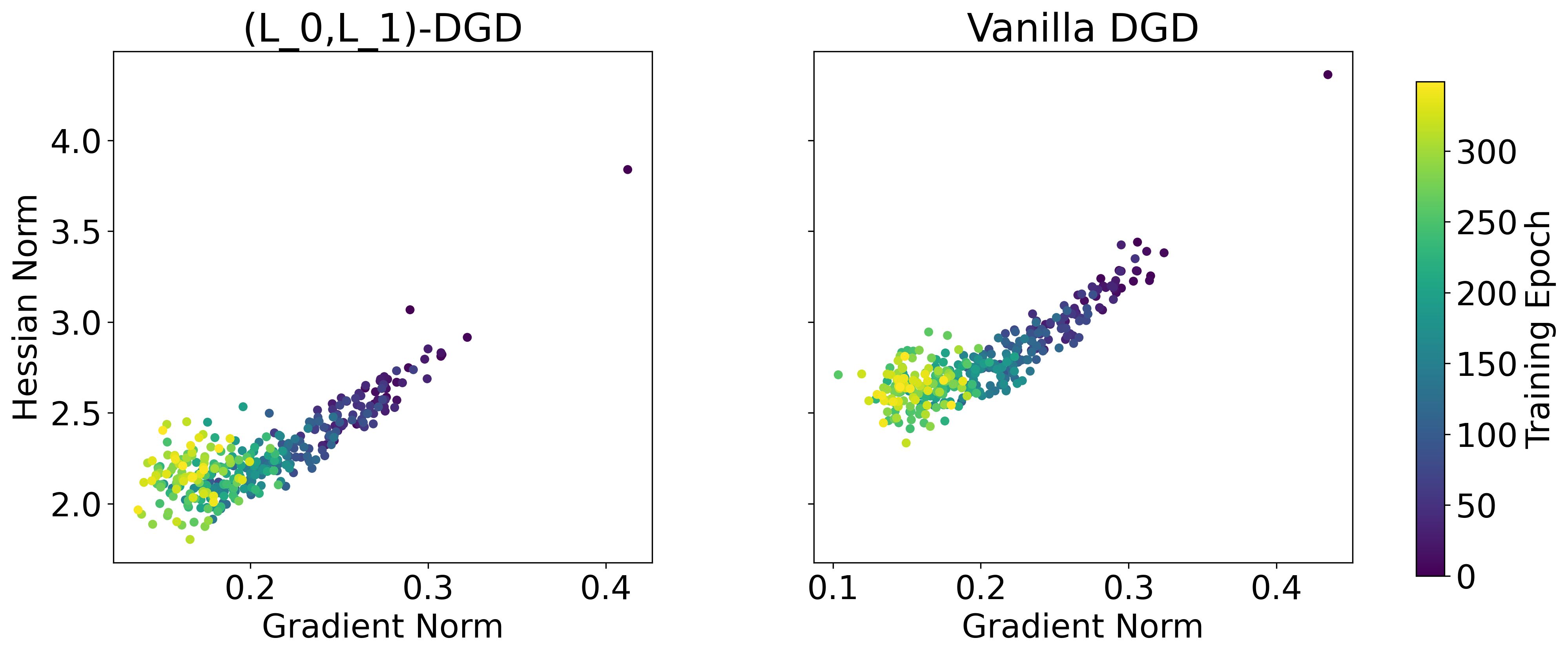}
    \caption{Smoothness dependency on the gradient norm for decentralized neural networks with a9a dataset for averaged model with a ring topology}
    \label{fig:comparison_nn_ring_a9a}
\end{figure}

\begin{figure}
    \centering
    \includegraphics[width=1\linewidth]{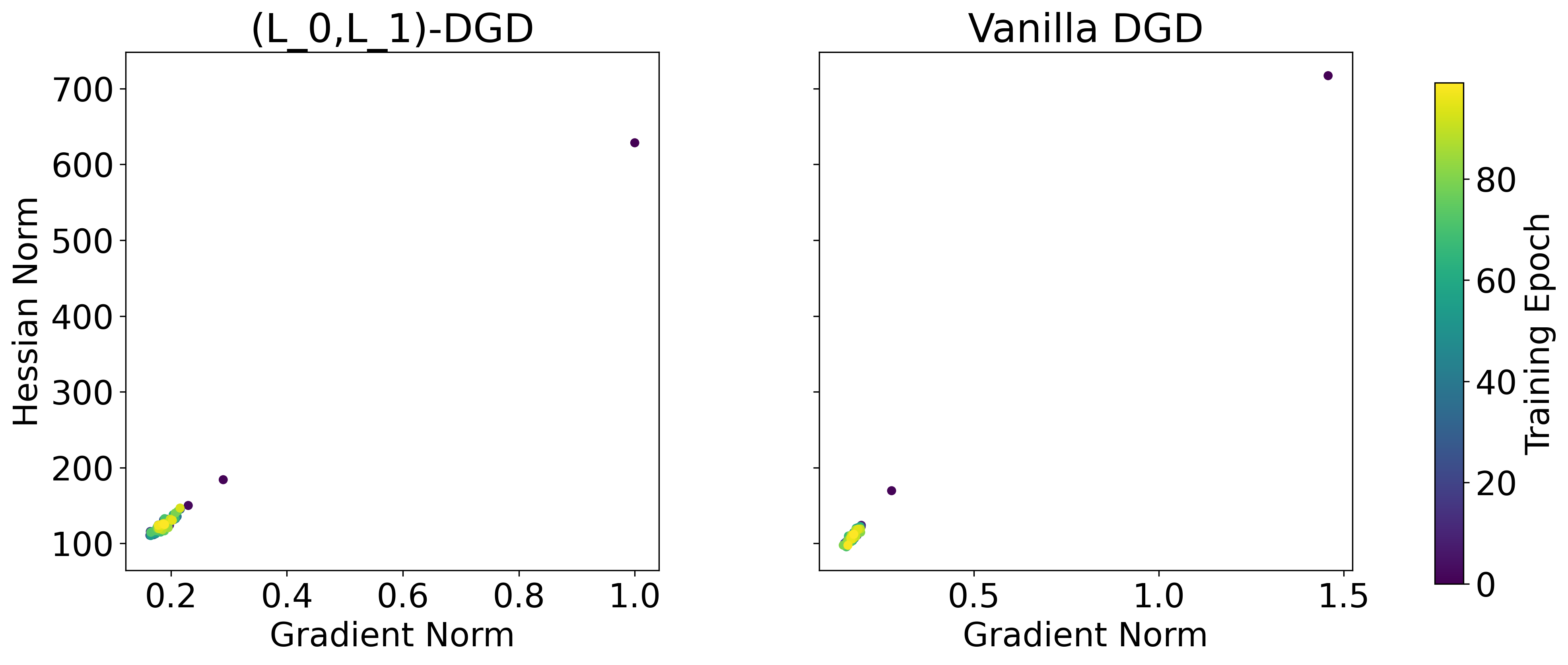}
    \caption{Smoothness dependency on the gradient norm for decentralized CNN with MNIST dataset for averaged model with a ring topology}
    \label{fig:comparison_cnn_ring_minist}
\end{figure}

\begin{figure}
    \centering
    \includegraphics[width=1\linewidth]{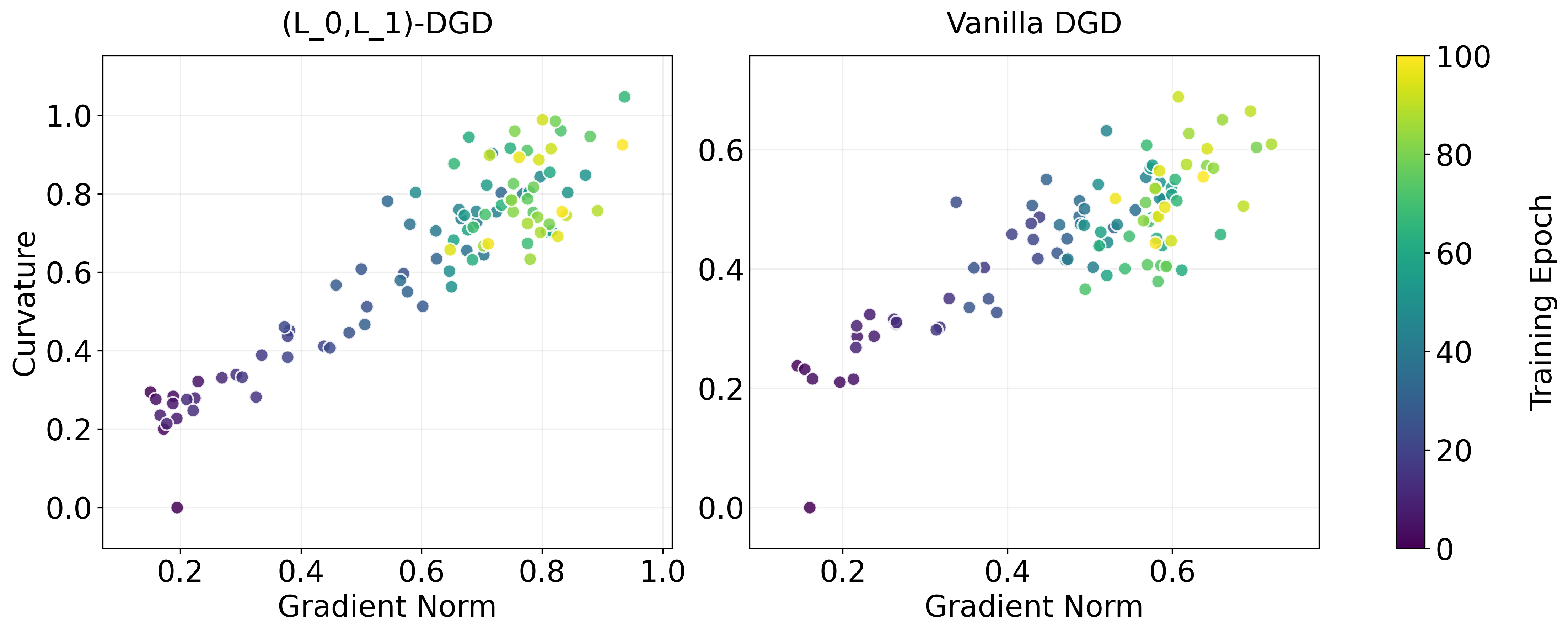}
    \caption{Smoothness dependency on the gradient norm for decentralized CNN with FashionMNIST~\cite{xiao2017fashion} dataset for averaged model with a fully-connected topology}
    \label{fig:comparison_cnn_fashionmnist}
\end{figure}

\begin{figure}
    \centering
    \includegraphics[width=1\linewidth]{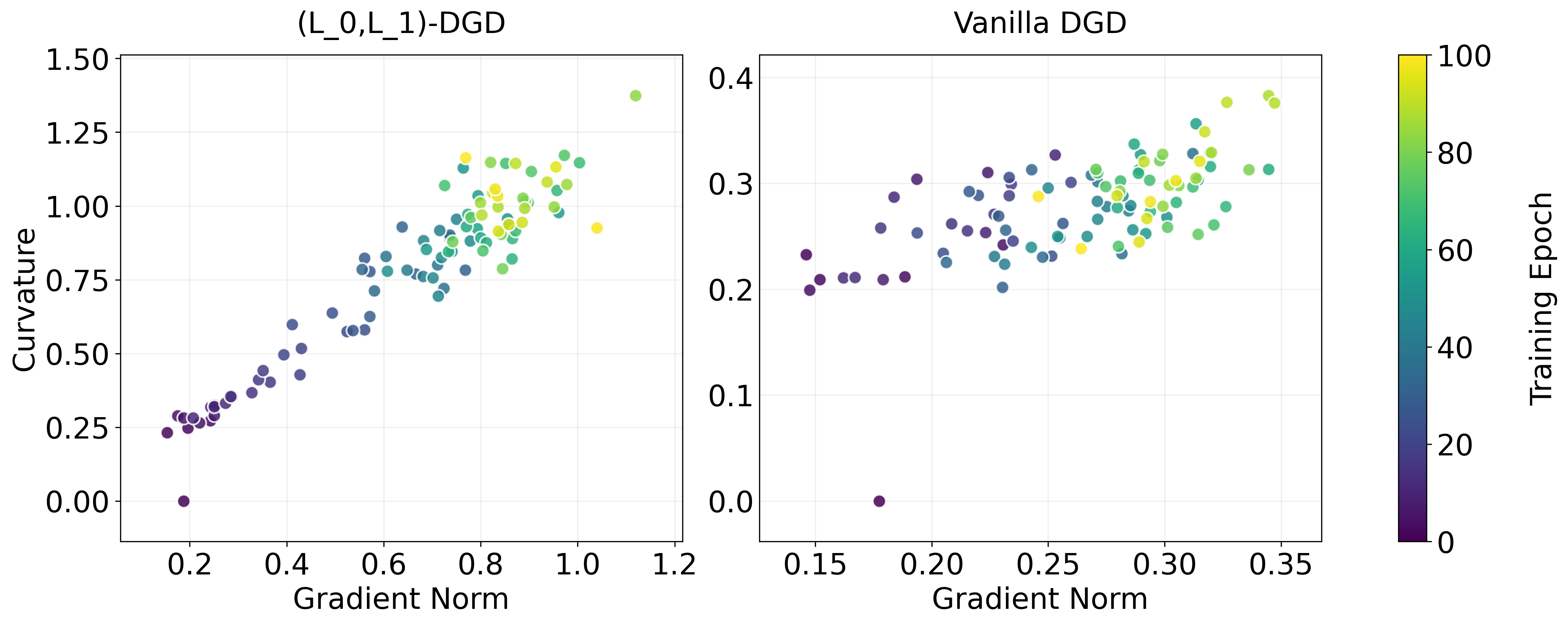}
    \caption{Smoothness dependency on the gradient norm for decentralized CNN with KMNIST~\cite{clanuwat2018deep} dataset for averaged model with a fully-connected topology}
    \label{fig:comparison_cnn_kmnist}
\end{figure}

\begin{figure}
    \centering
    \includegraphics[width=1\linewidth]{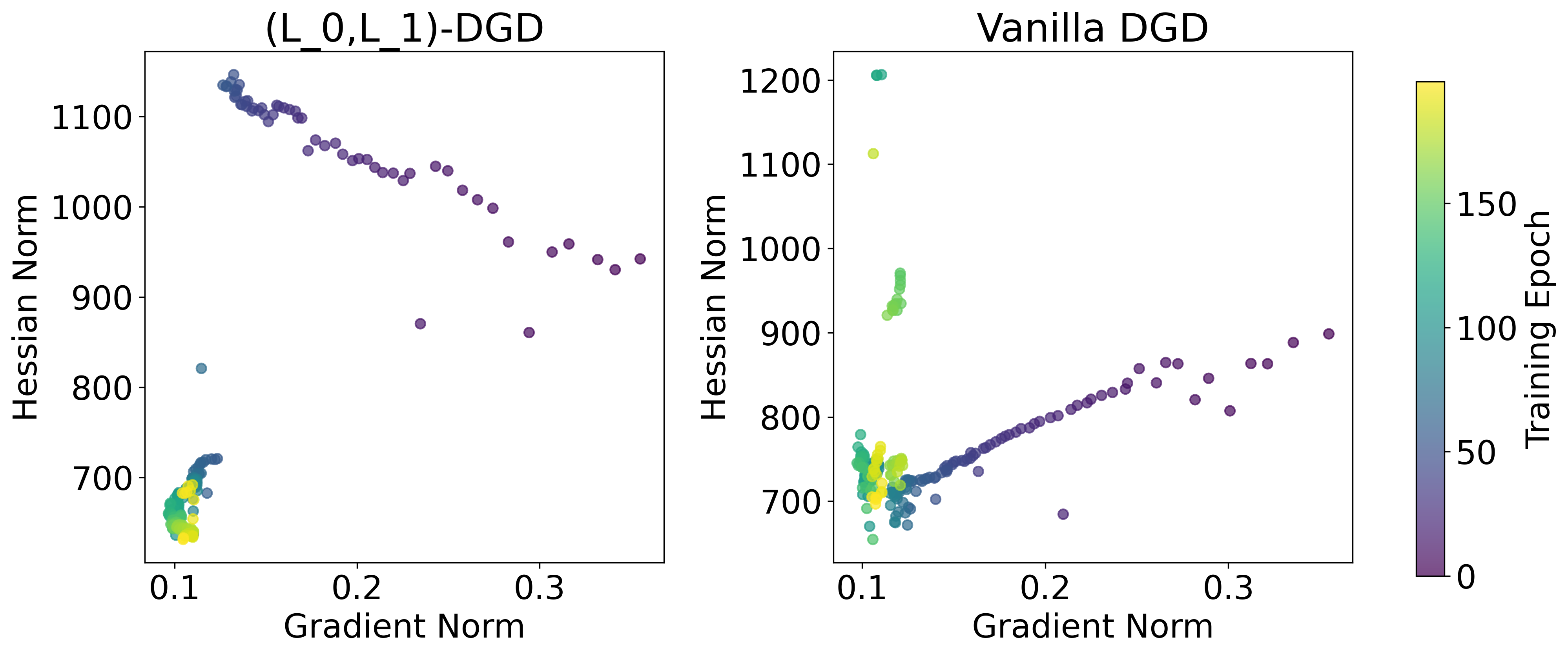}
    \caption{Smoothness dependency on the gradient norm for decentralized logistic regression with Covtype dataset for averaged model with a fully-connected topology and 5 agents}
    \label{fig:average_comparison_lr_covtype}
\end{figure}

\begin{figure}
    \centering
    \includegraphics[width=1\linewidth]{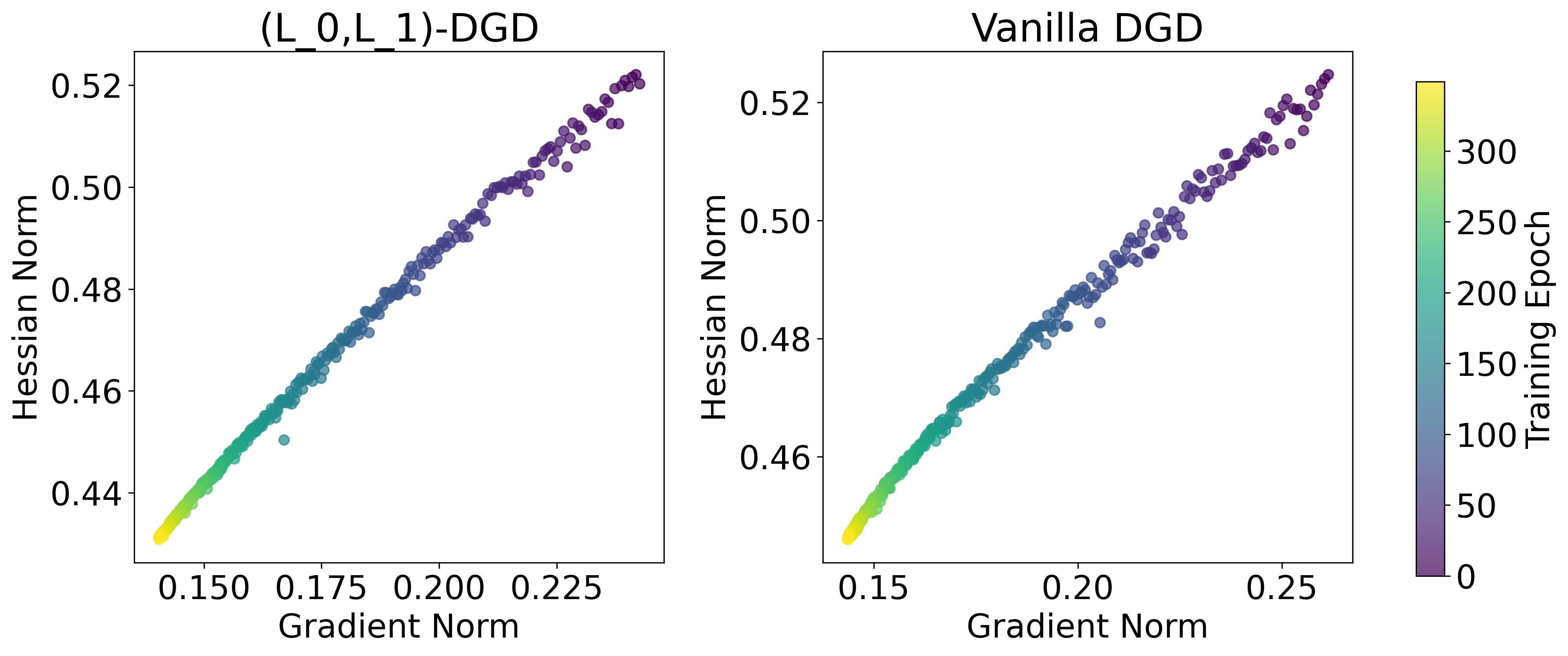}
    \caption{Smoothness dependency on the gradient norm for decentralized neural networks with Covtype dataset for averaged model with a fully-connected topology and 5 agents}
    \label{fig:average_comparison_nn_covtype}
\end{figure}

\begin{figure}
    \centering
    \includegraphics[width=1\linewidth]{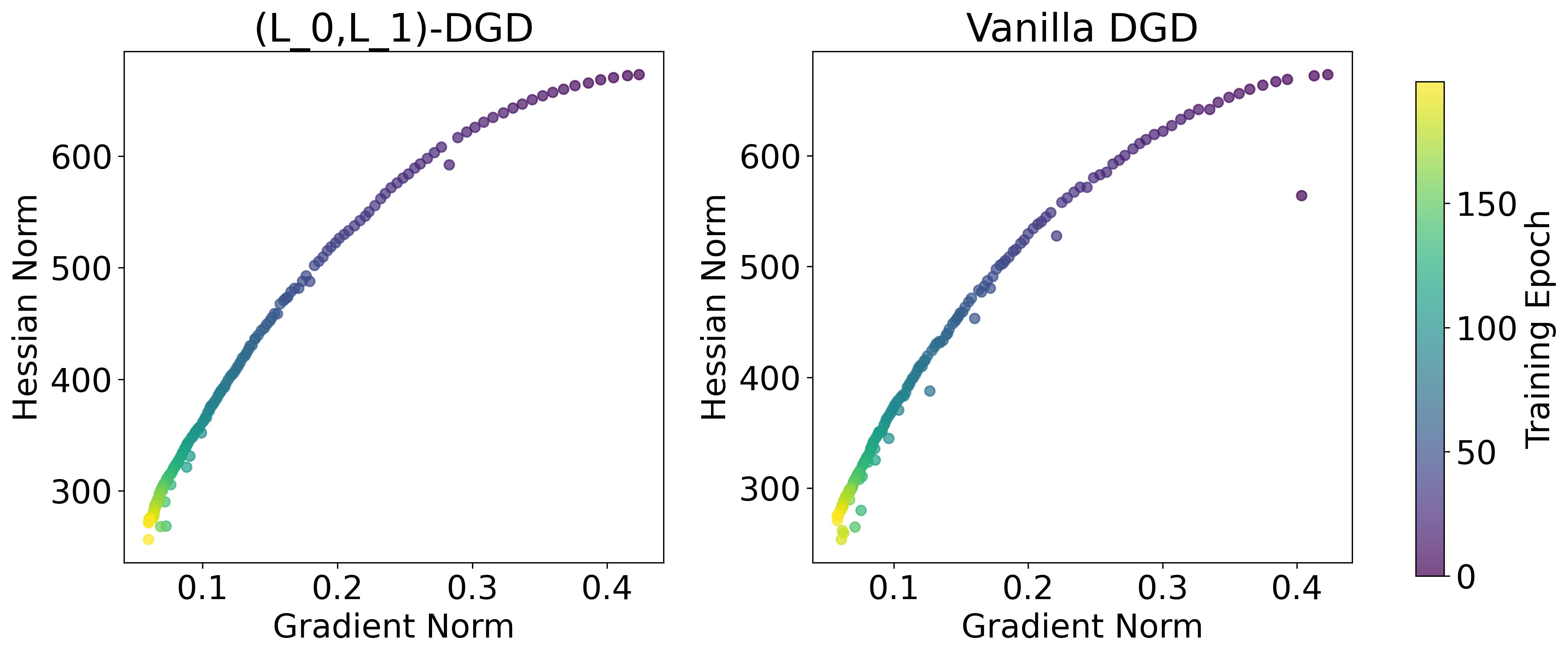}
    \caption{Smoothness dependency on the gradient norm for decentralized logistic regression with IJCNN1 dataset for averaged model with a fully-connected topology and 5 agents}
    \label{fig:average_comparison_lr_ijcnn}
\end{figure}

\begin{figure}
    \centering
    \includegraphics[width=1\linewidth]{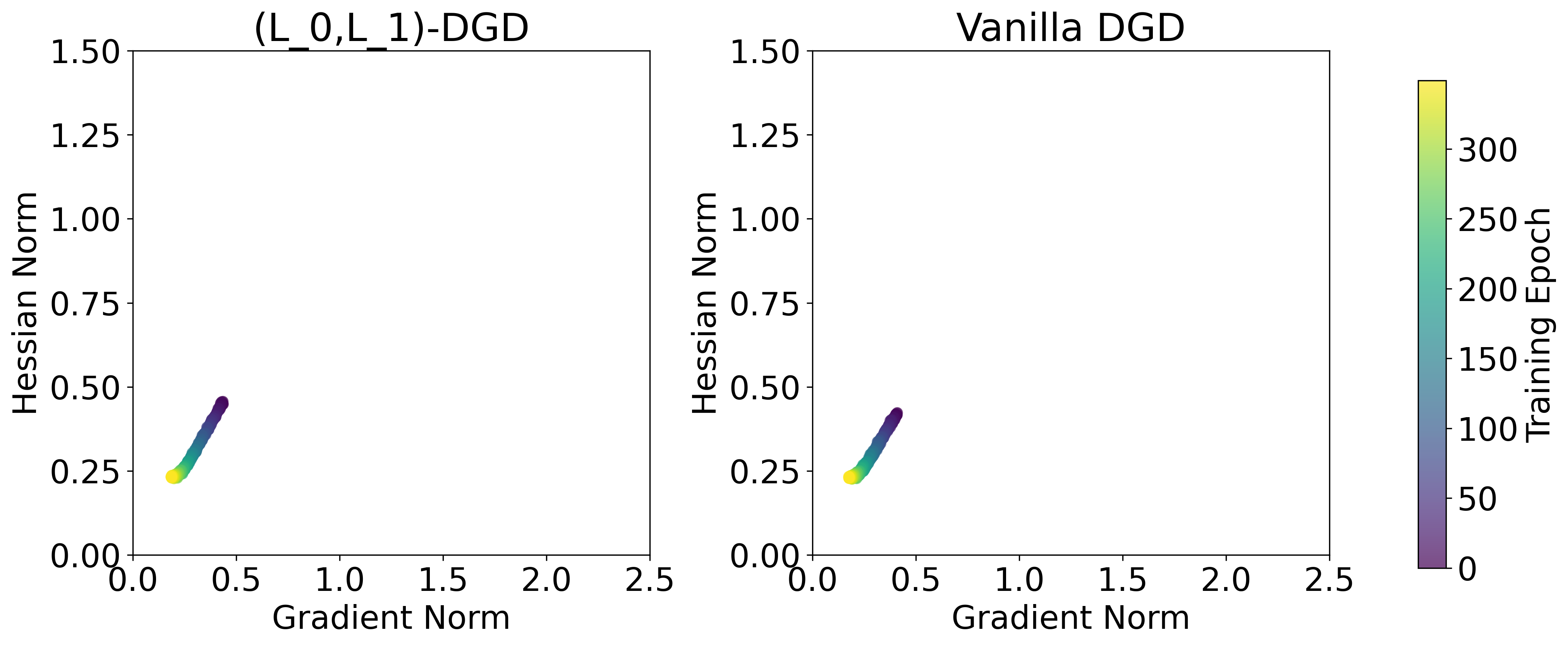}
    \caption{Smoothness dependency on the gradient norm for decentralized neural networks with IJCNN1 dataset for averaged model with a fully-connected topology and 5 agents}
    \label{fig:average_comparison_nn_ijcnn}
\end{figure}

\begin{figure}
    \centering
    \includegraphics[width=1\linewidth]{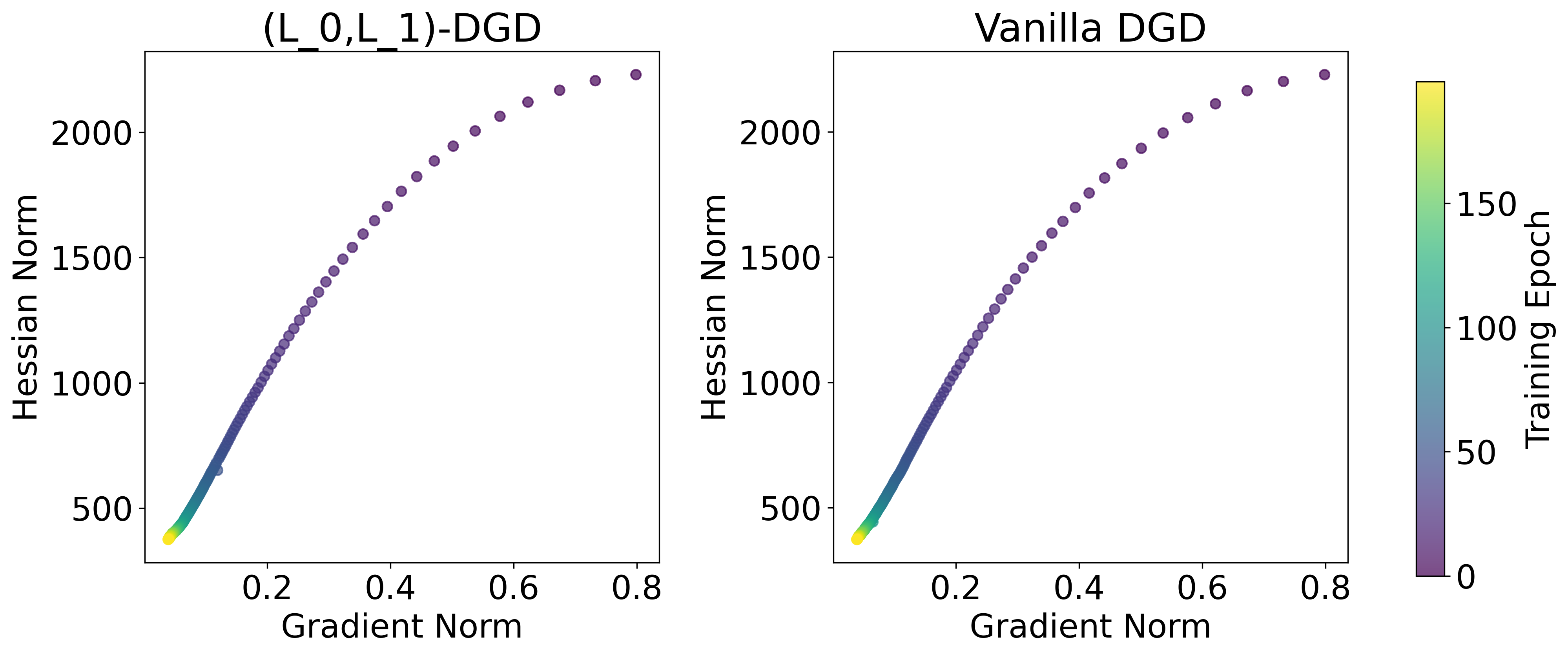}
    \caption{Smoothness dependency on the gradient norm for decentralized logistic regression with KDD CUP 99 dataset for averaged model with a fully-connected topology and 5 agents}
    \label{fig:average_comparison_lr_kddcup99}
\end{figure}

\begin{figure}
    \centering
    \includegraphics[width=1\linewidth]{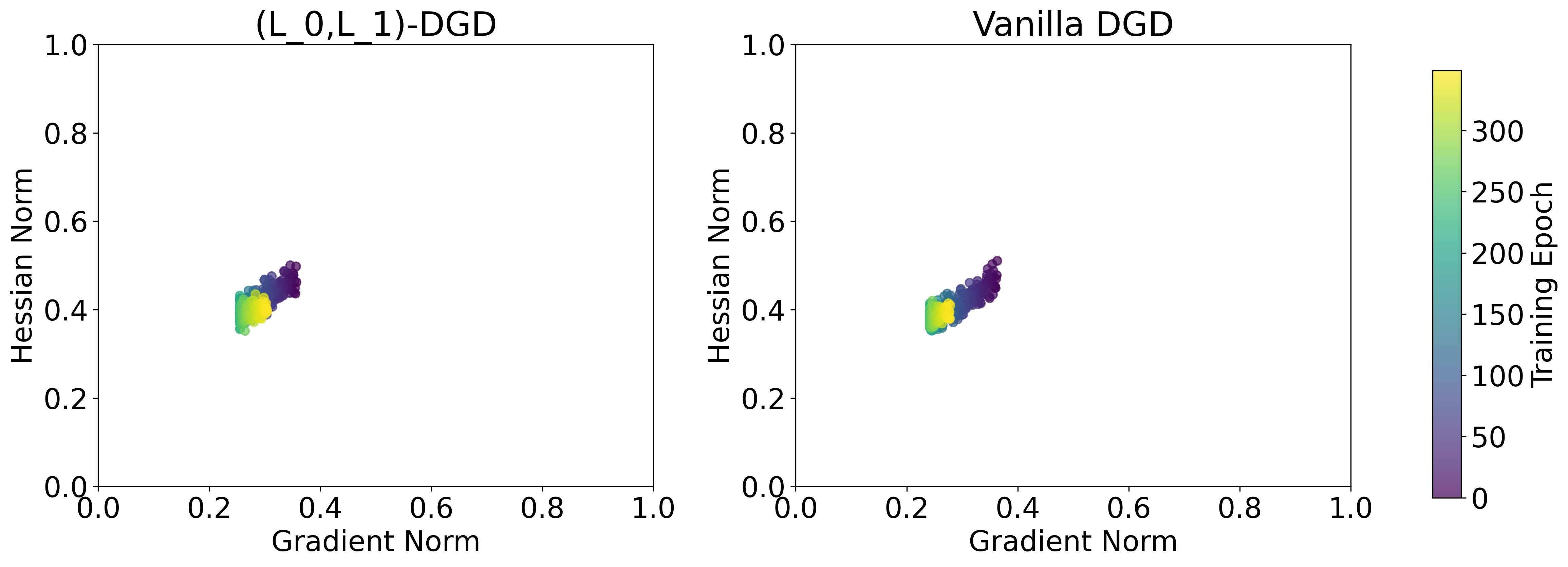}
    \caption{Smoothness dependency on the gradient norm for decentralized neural networks with KDD CUP 99 dataset for averaged model with a fully-connected topology and 5 agents}
    \label{fig:average_comparison_nn_kddcup99}
\end{figure}

\begin{figure}
    \centering
    \includegraphics[width=1\linewidth]{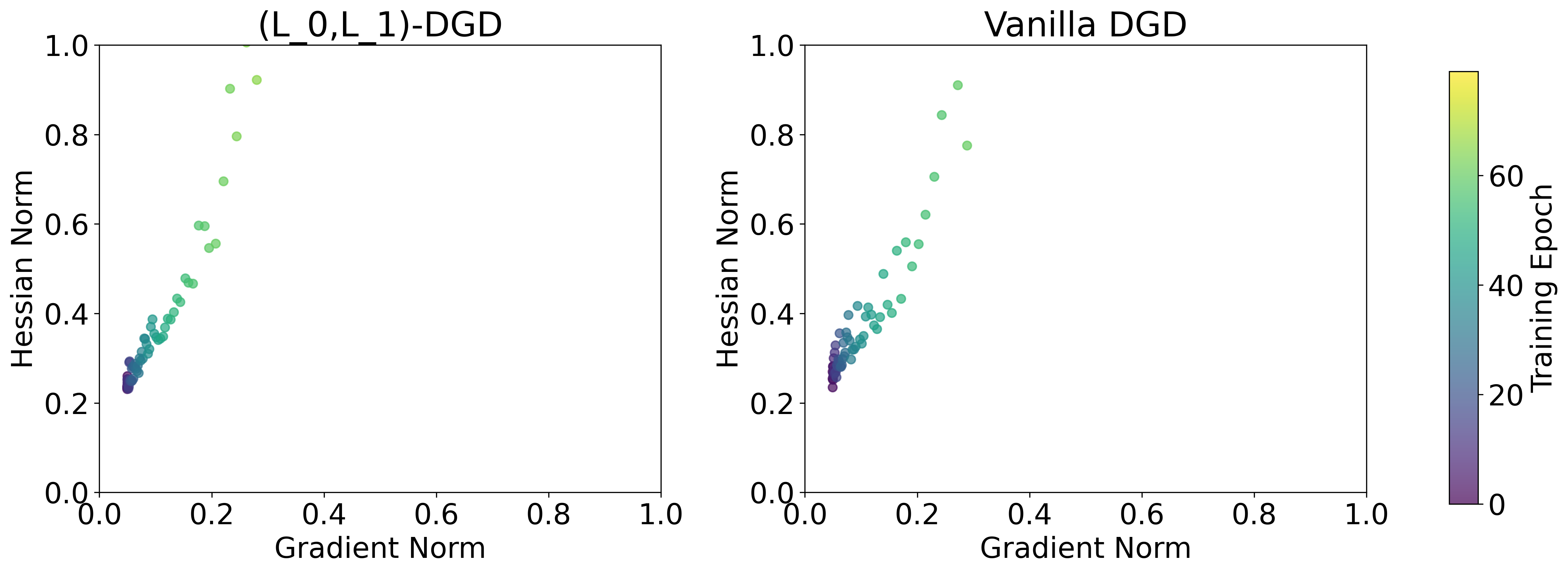}
    \caption{Smoothness dependency on the gradient norm for decentralized CNN with CIFAR 10 dataset~\cite{recht2018cifar} for averaged model with a fully-connected topology and 5 agents}
    \label{fig:average_comparison_cnn_cifar10}
\end{figure}

\begin{figure}
    \centering
    \includegraphics[width=1\linewidth]{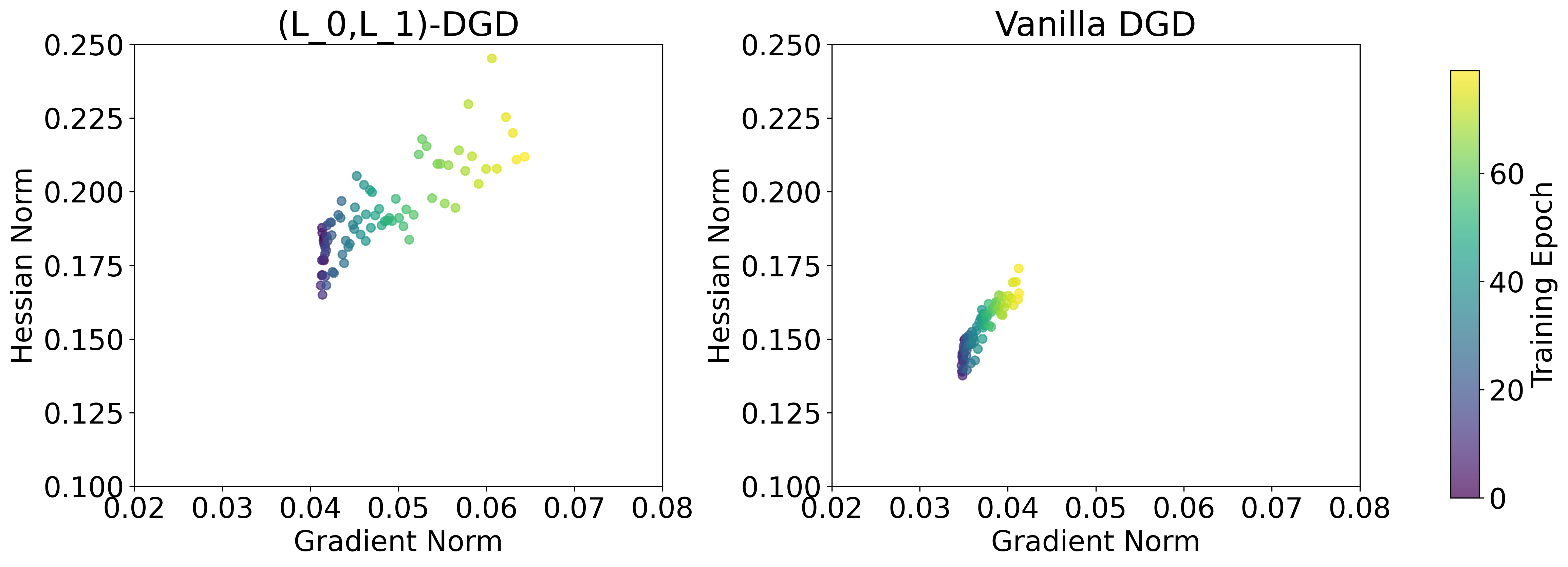}
    \caption{Smoothness dependency on the gradient norm for decentralized CNN with CIFAR 100 dataset~\cite{singla2021improved} for averaged model with a fully-connected topology and 5 agents}
    \label{fig:average_comparison_cnn_cifar100}
\end{figure}

\begin{figure}
    \centering
    \includegraphics[width=1\linewidth]{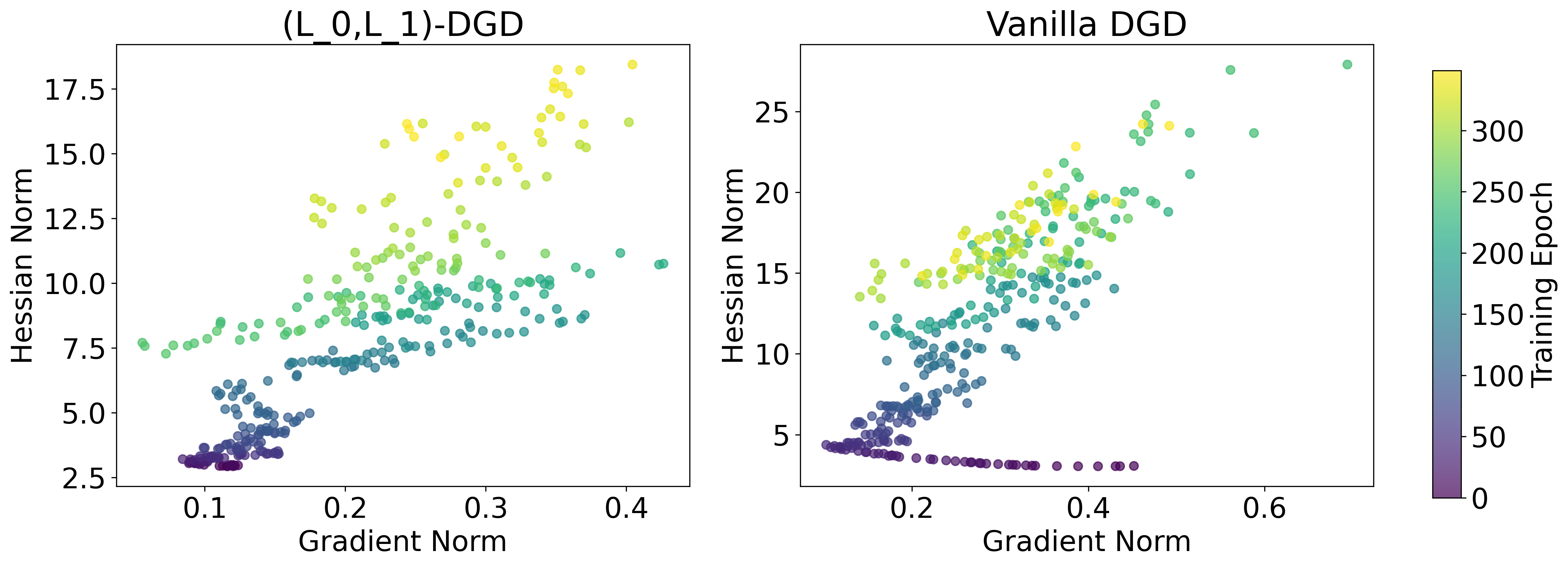}
    \caption{Smoothness dependency on the gradient norm for decentralized GRU with Beijing PM2.5 dataset~\cite{wang2018characteristics} for averaged model with a fully-connected topology and 5 agents}
    \label{fig:average_comparison_gru_pm2.5}
\end{figure}

\begin{figure}
    \centering
    \includegraphics[width=1\linewidth]{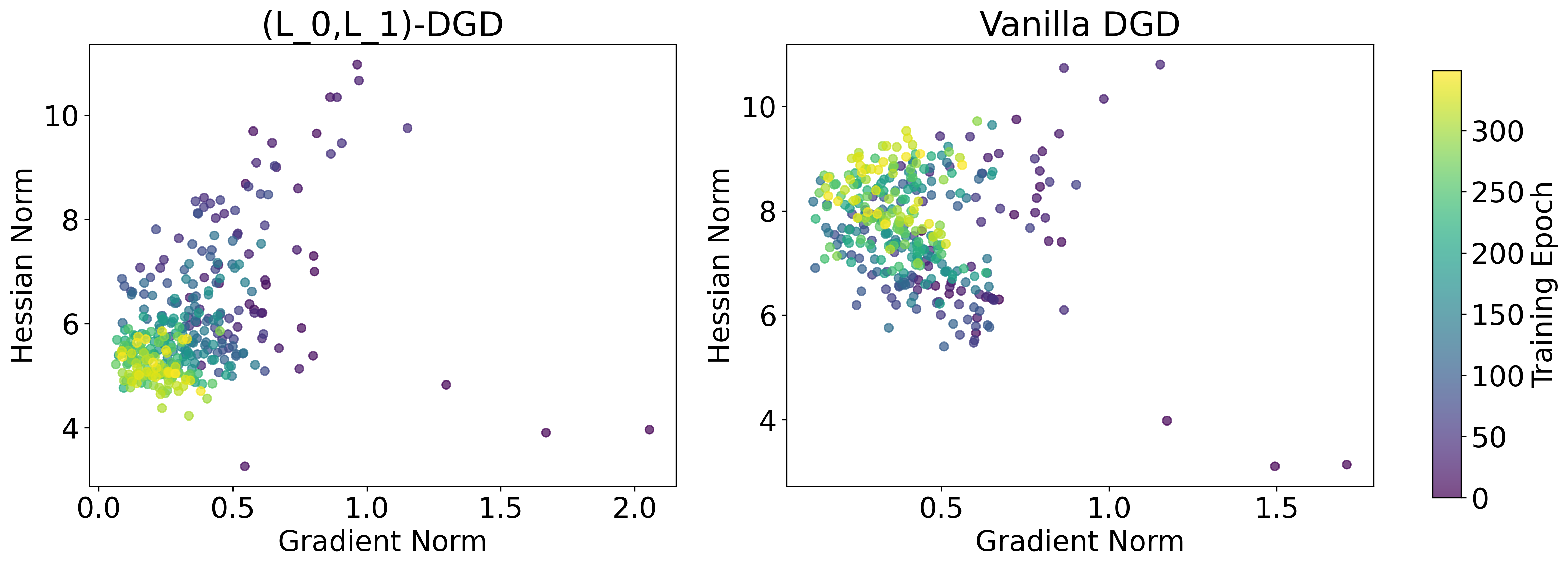}
    \caption{Smoothness dependency on the gradient norm for decentralized GRU with individual household power consumption dataset~\cite{yang2020household} for averaged model with a fully-connected topology and 5 agents}
    \label{fig:average_comparison_gru_power_com}
\end{figure}

\begin{figure}
    \centering
    \includegraphics[width=1\linewidth]{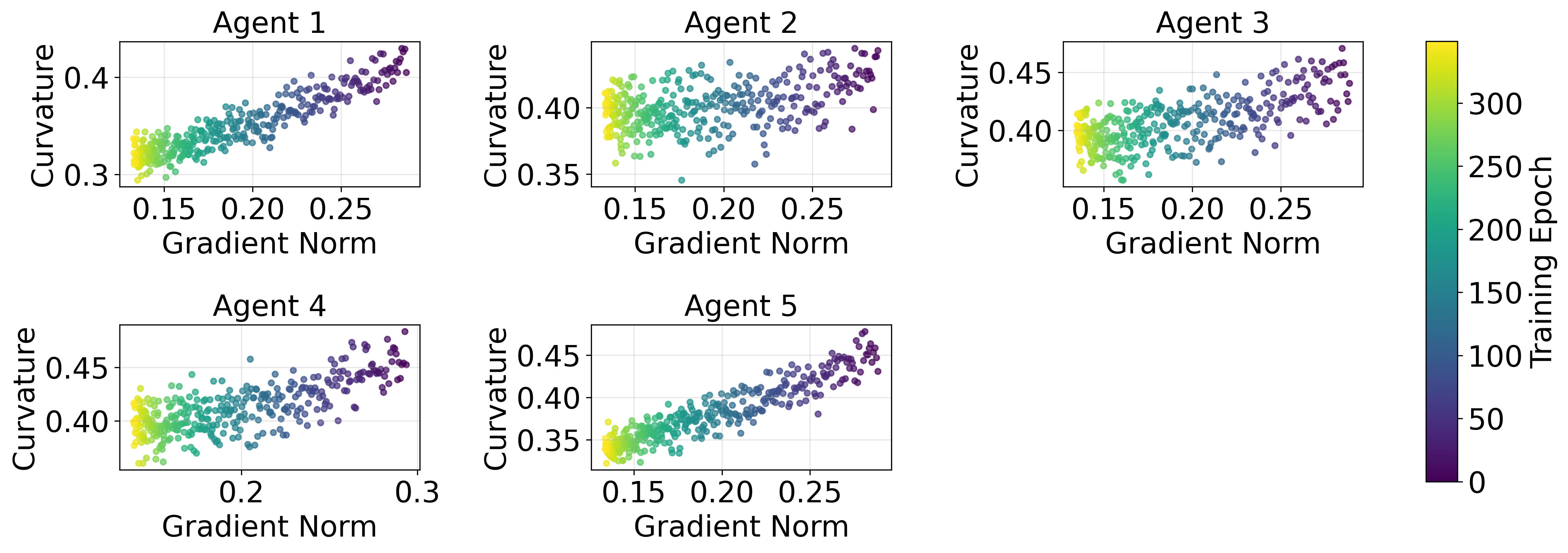}
    \caption{Smoothness dependency on the gradient norm for decentralized neural networks with a9a dataset for individual agents with a fully-connected topology and $(L_0,L_1)$-DGD}
    \label{fig:comparison_nn_individual_clip}
\end{figure}

\begin{figure}
    \centering
    \includegraphics[width=1\linewidth]{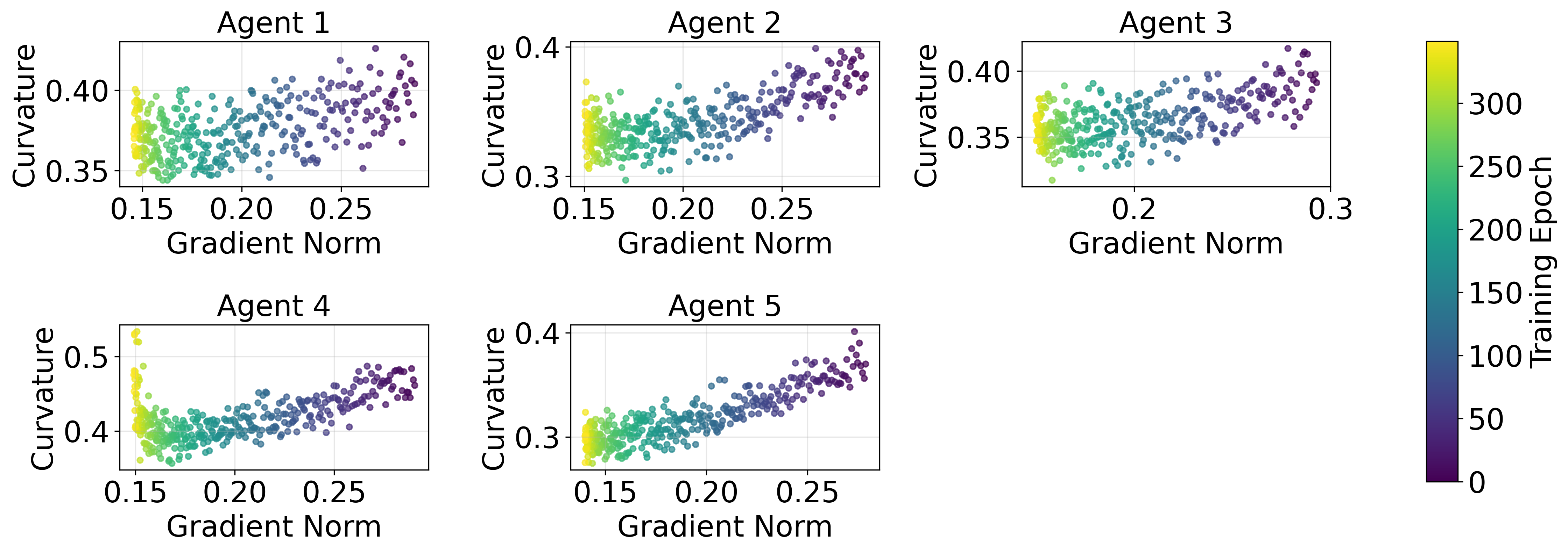}
    \caption{Smoothness dependency on the gradient norm for decentralized neural networks with a9a dataset for individual agents with a fully-connected topology and vanilla DGD}
    \label{fig:comparison_nn_individual_non_clip}
\end{figure}

\begin{figure}
    \centering
    \includegraphics[width=1\linewidth]{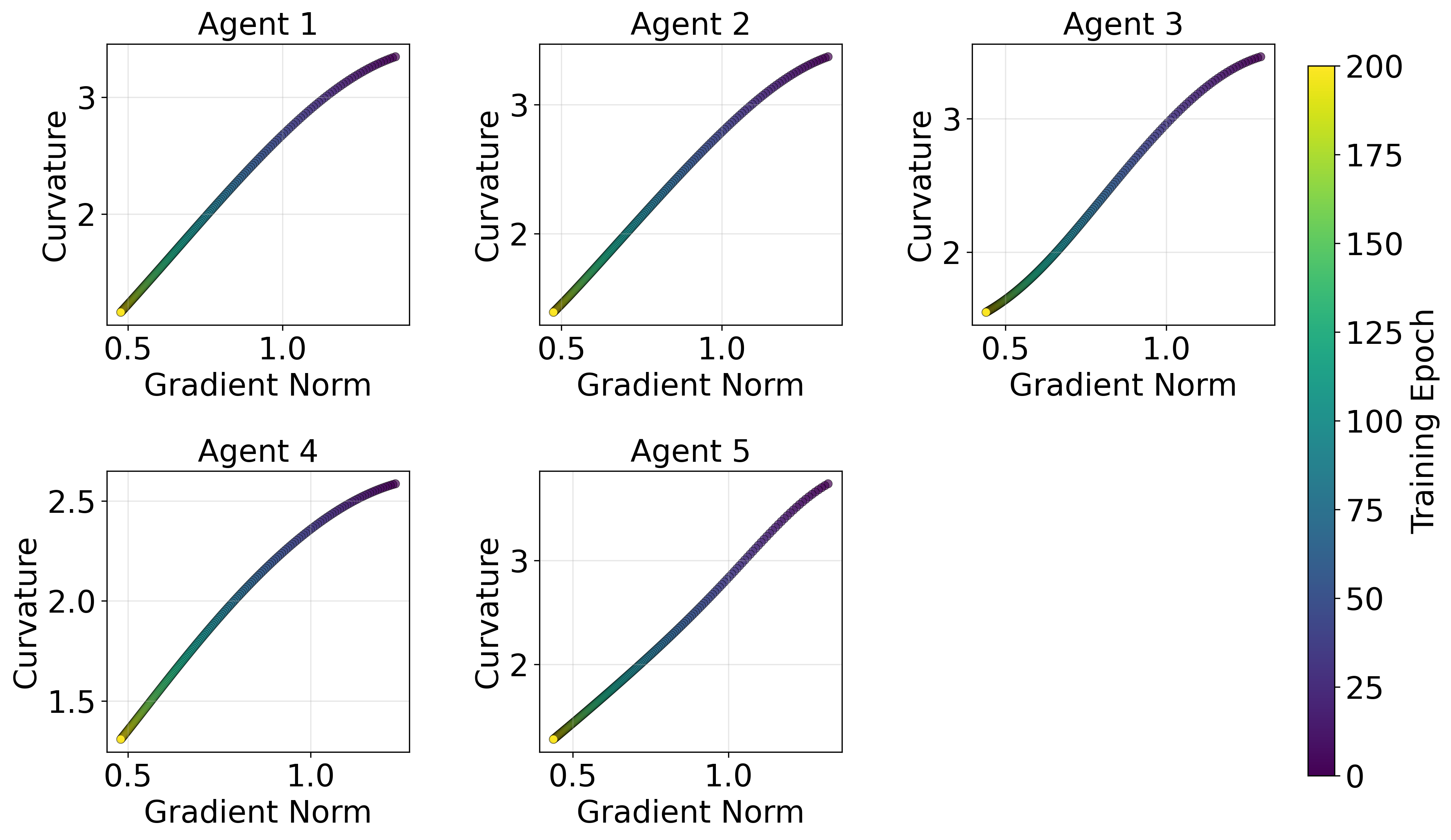}
    \caption{Smoothness dependency on the gradient norm for decentralized logistic regression with breast cancer dataset for individual agent with a fully-connected topology}
    \label{fig:comparison_lr_breast_cancer_individual}
\end{figure}

\begin{figure}
    \centering
    \includegraphics[width=1\linewidth]{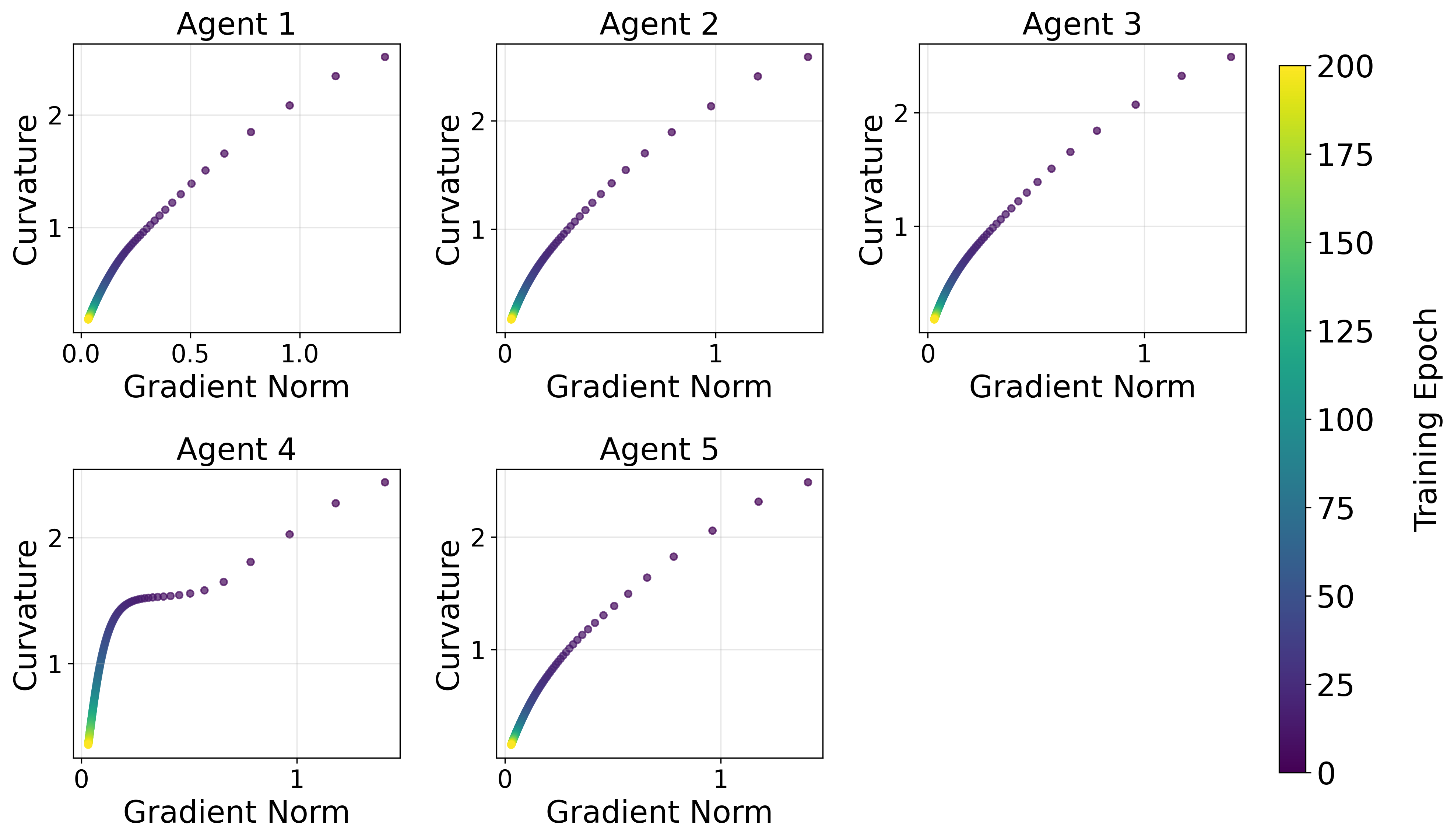}
    \caption{Smoothness dependency on the gradient norm for decentralized logistic regression with mushroom~\cite{chang2011libsvm} dataset for individual agent with a fully-connected topology}
    \label{fig:comparison_lr_mushroom_individual}
\end{figure}

\begin{figure}
    \centering
    \includegraphics[width=1\linewidth]{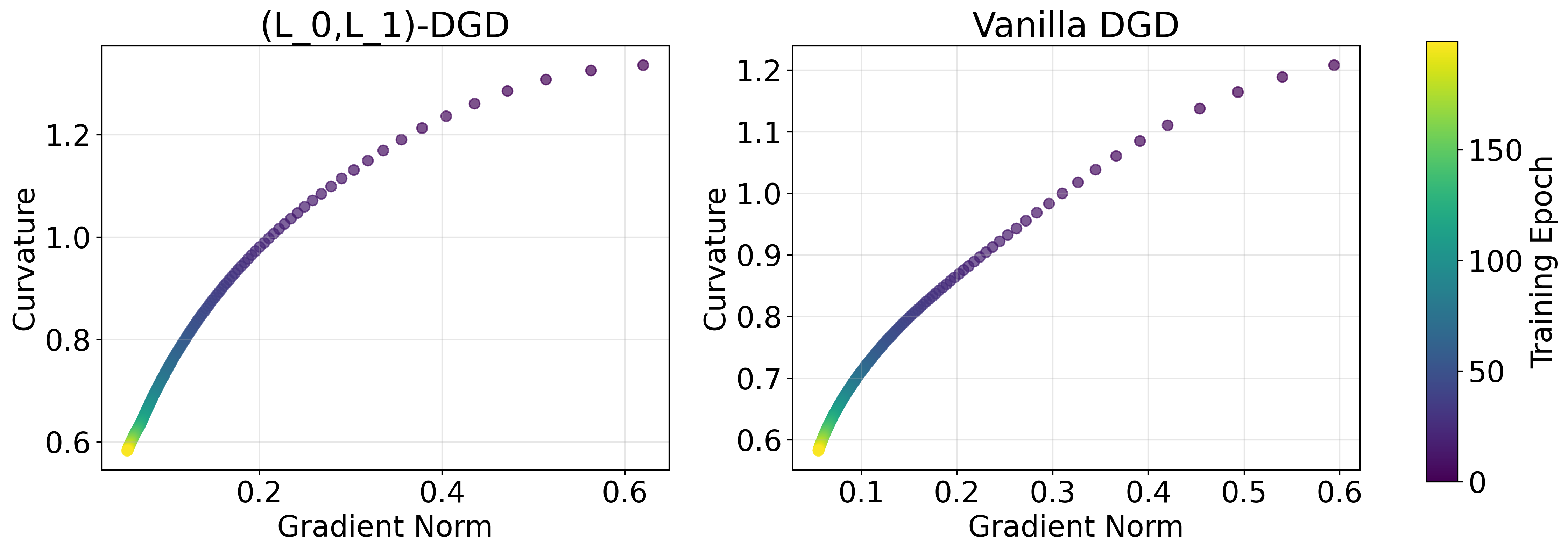}
    \caption{Smoothness dependency on the gradient norm for decentralized logistic regression with a9a dataset for averaged model with a fully-connected topology and 10 agents}
    \label{fig:average_comparison_lr_a9a_10}
\end{figure}

\begin{figure}
    \centering
    \includegraphics[width=1\linewidth]{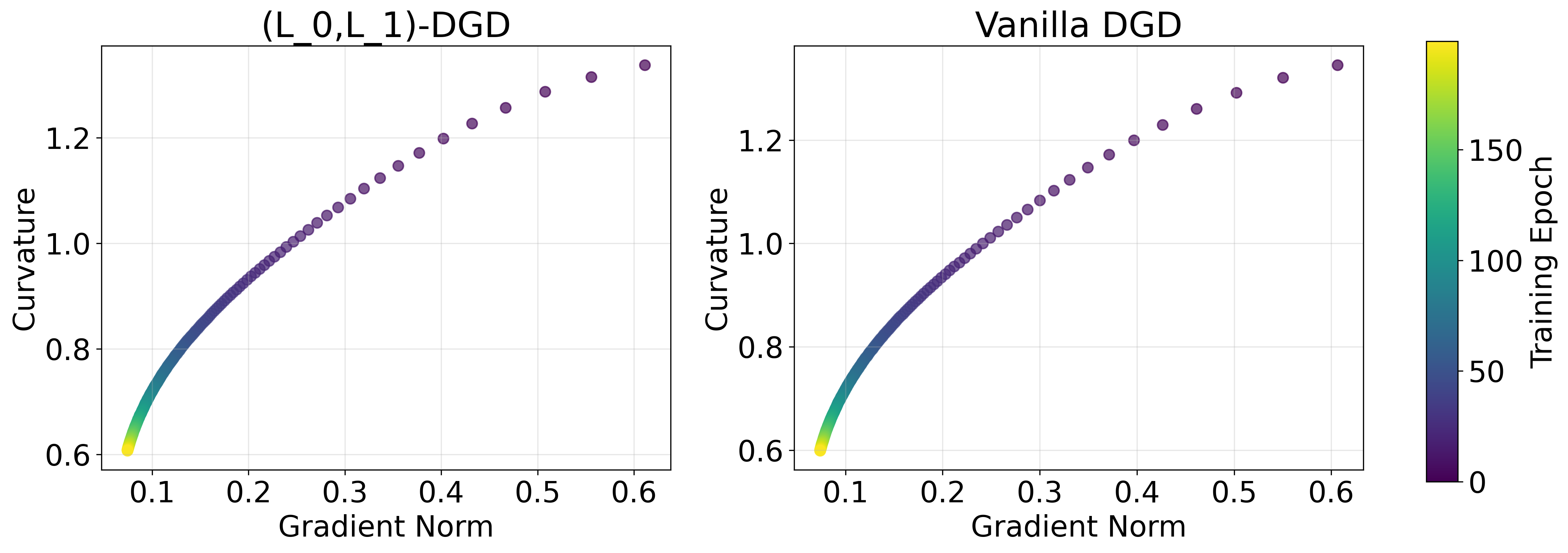}
    \caption{Smoothness dependency on the gradient norm for decentralized logistic regression with a9a dataset for averaged model with a fully-connected topology and 20 agents}
    \label{fig:average_comparison_lr_a9a_20}
\end{figure}

\begin{figure}
    \centering
    \includegraphics[width=1\linewidth]{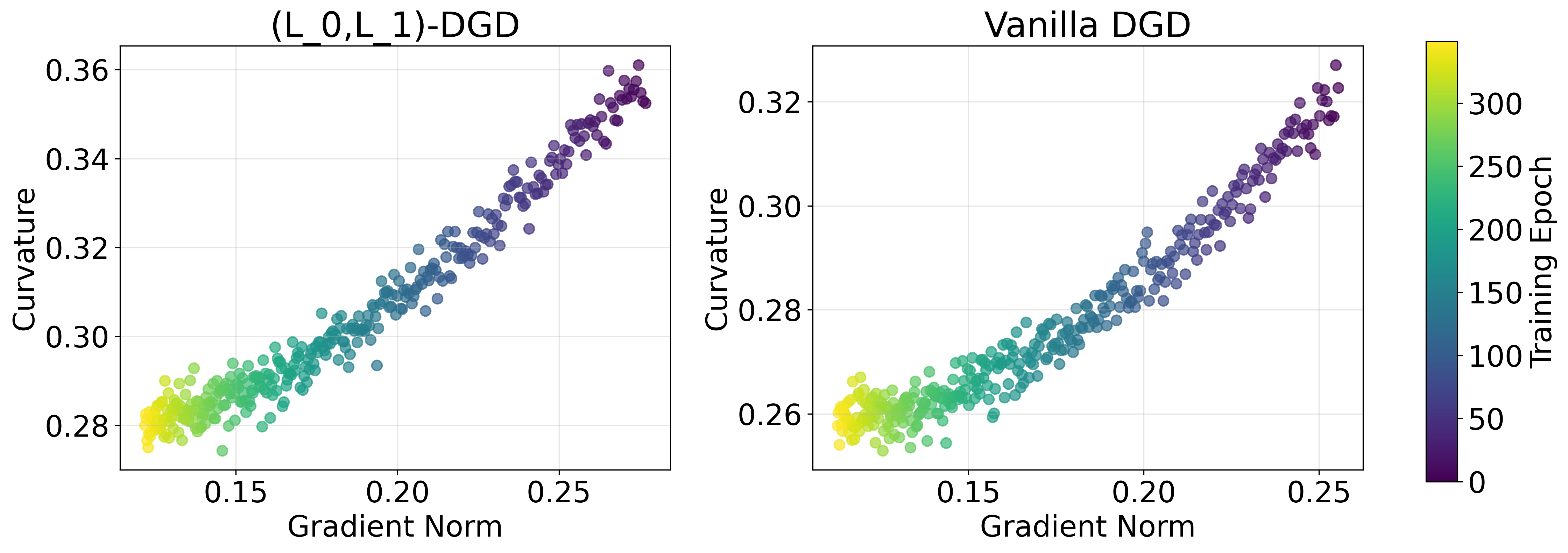}
    \caption{Smoothness dependency on the gradient norm for decentralized neural networks with a9a dataset for averaged model with a fully-connected topology and 10 agents}
    \label{fig:average_comparison_nn_a9a_10}
\end{figure}

\begin{figure}
    \centering
    \includegraphics[width=1\linewidth]{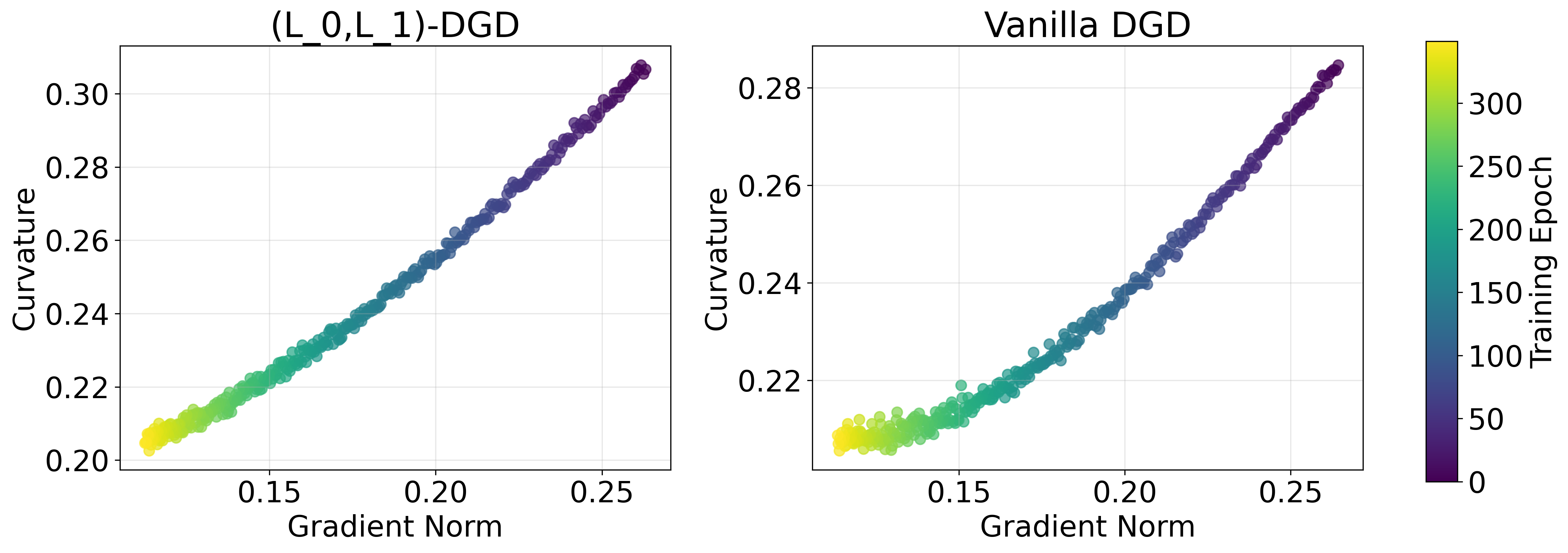}
    \caption{Smoothness dependency on the gradient norm for decentralized neural networks with a9a dataset for averaged model with a fully-connected topology and 20 agents}
    \label{fig:average_comparison_nn_a9a_20}
\end{figure}

\begin{figure}
    \centering
    \includegraphics[width=1\linewidth]{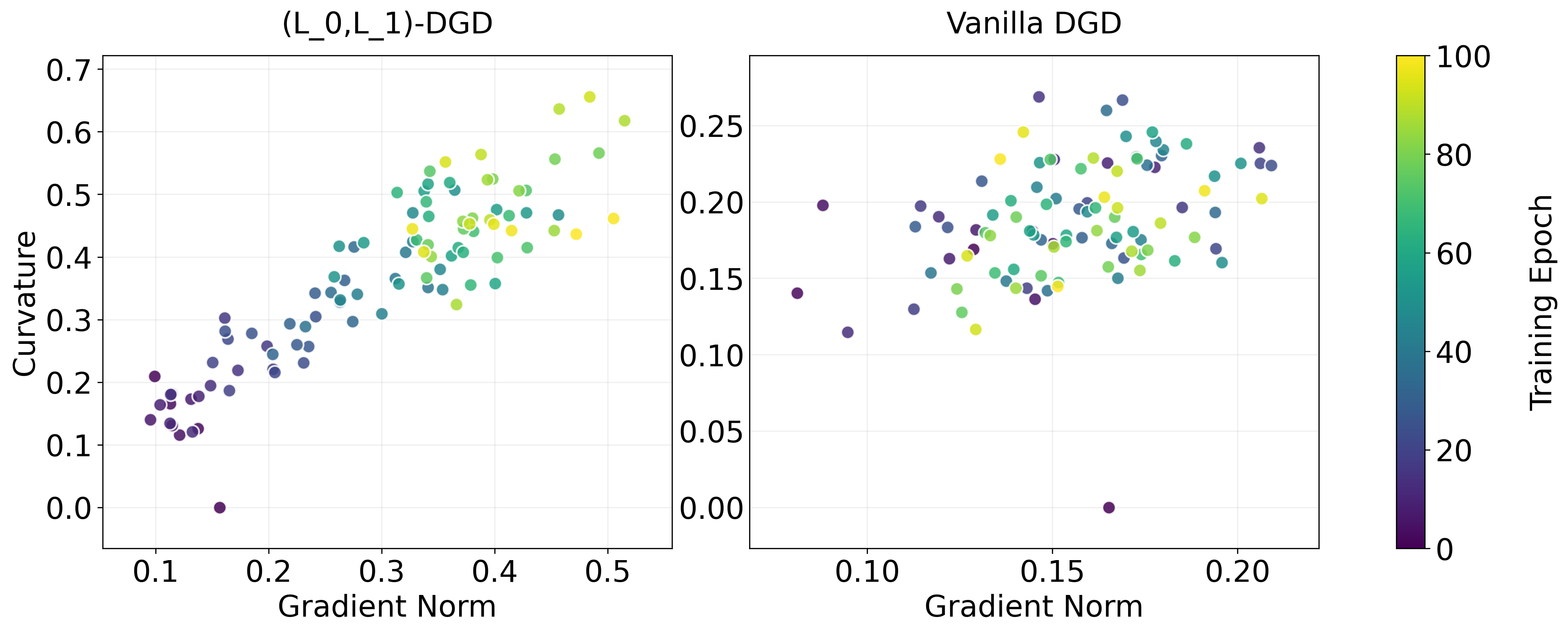}
    \caption{Smoothness dependency on the gradient norm for decentralized CNN with MNIST dataset for averaged model with a fully-connected topology and 10 agents}
    \label{fig:average_comparison_cnn_mnist_10}
\end{figure}

\begin{figure}
    \centering
    \includegraphics[width=1\linewidth]{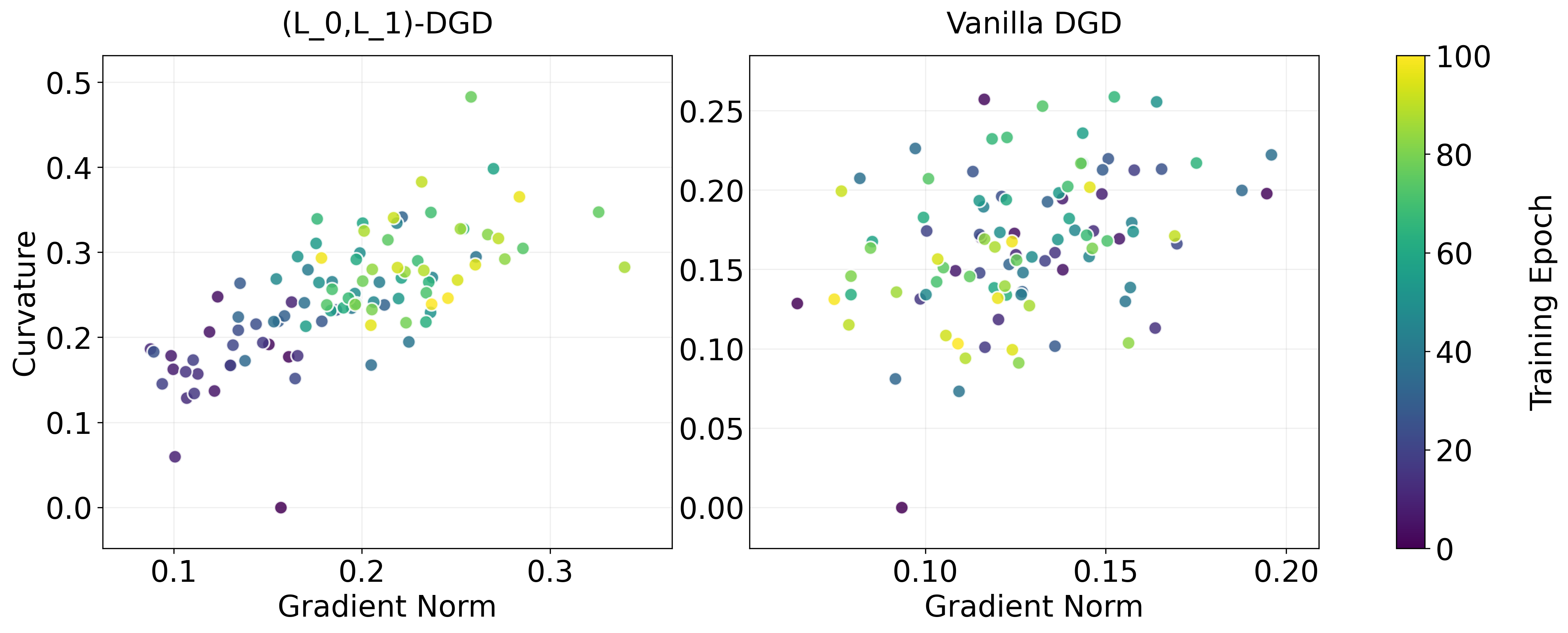}
    \caption{Smoothness dependency on the gradient norm for decentralized CNN with MNIST dataset for averaged model with a fully-connected topology and 20 agents}
    \label{fig:average_comparison_cnn_mnist_20}
\end{figure}

\subsection{Auxiliary Theoretical Analysis}\label{auxiliary_results}
\begin{lemma}\label{lemma_5}
    Let $x\in\mathbb{R}^d$ be the current point in a gradient descent type of update. If for any $y\in\mathbb{R}^d$, the following relationship holds:
    \begin{equation}\label{eq_a1}
        f^i(y)\leq f^i(x)+\langle\nabla f^i(x),y-x\rangle+\frac{z}{L_1^2}\psi(L_1\|y-x\|),
    \end{equation}
    where $z:=L_0+L_1\|\nabla f^i(x)\|>0$, $\psi(\gamma)=e^\gamma-\gamma-1$, then the optimal point $y^*=\mathcal{T}(x)$ satisfies the form $\mathcal{T}(x)=x-\eta\frac{\nabla f(x)}{\|\nabla f(x)\|}$.
\end{lemma}
\begin{proof}
    To optimize $y$ on the right hand side of Eq.~\ref{eq_a1}, we calculate its gradient w.r.t $y$ and equal it to 0 such that the following can be obtained:
    \begin{equation}
        \nabla_y(e^{\|y-x\|}-\|y-x\|-1)+\nabla_y\langle\nabla f^i(x),y-x\rangle=0,
    \end{equation}
    which leads to the following
    \begin{equation}
        (e^{\|y-x\|}-1)\frac{y-x}{\|y-x\|}=-\nabla f^i(x)
    \end{equation}
    Let $b=y-x$ such that the above equality can be rewritten as 
    \begin{equation}\label{eq_a2}
        (e^b-1)\frac{b}{\|b\|}=-\nabla f^i(x).
    \end{equation}
    We consider two scenarios: $\nabla f^i(x)=0$ and $\nabla f^i(x)\neq 0$. When $\nabla f^i(x)=0$, this is trivial to yield that $y=x$. 

    Now we consider $\nabla f^i(x)\neq 0$. We can know that in Eq.~\ref{eq_a2}, the left side represents a vector in the direction of $b=y-x$ and the right side is a vector $-\nabla f^i(x)$. Therefore, $y-x$ must be the in the same direction of $-\nabla f^i(x)$ (or the opposite direction if the scalar term is negative). Let $y-x=\lambda(-\nabla f^i(x))$ for some scalar $\lambda$ such that $\|-\nabla f^i(x)\|=|\lambda|\|\nabla f^i(x)\|$. Substituting this into Eq.~\ref{eq_a2}, we have
    \begin{equation}
        (e^{|\lambda|\|\nabla f^i(x)\|}-1)\frac{|\lambda|\|\nabla f^i(x)\|}{|\lambda|\|\nabla f^i(x)\|}=-\nabla f^i(x).
    \end{equation}
    If $\lambda>0$: we have
    \begin{equation}
        (e^{|\lambda|\|\nabla f^i(x)\|}-1)\frac{\|\nabla f^i(x)\|}{\|\nabla f^i(x)\|}=\nabla f^i(x).
    \end{equation}
    If $\nabla f^i(x)\neq 0$, we can divide both sides by $\nabla f^i(x)$ (considering the direction) to get $\lambda$ with some mathematical manipulation:
    \begin{equation}
        \lambda = \frac{\textnormal{ln}(1+\|\nabla f^i(x)\|)}{\|\nabla f^i(x)\|},
    \end{equation}
    which yields
    \begin{equation}
        y=x-\frac{\textnormal{ln}(1+\|\nabla f^i(x)\|)}{\|\nabla f^i(x)\|}\nabla f^i(x).
    \end{equation}
    If $\lambda<0$, let $\lambda=-\beta$ where $\beta>0$. Then we can obtain:
    \begin{equation}
        e^{\beta\|\nabla f^i(x)\|}=1-\|\nabla f^i(x)\|.
    \end{equation}
    For a real solution for $\beta$ to exist, we need $1-\|\nabla f^i(x)\|>0$, so $\|\nabla f^i(x)\|<1$. In this case, we have $\beta=\frac{\textnormal{ln}(1-\|\nabla f^i(x)\|)}{\|\nabla f^i(x)\|}$. This implies that
    \begin{equation}
        y = x+\frac{\textnormal{ln}(1-\|\nabla f^i(x)\|)}{\|\nabla f^i(x)\|}\nabla f^i(x).
    \end{equation}
    Combining these three cases, and letting $\eta$ equal to $0, \textnormal{ln}(1+\|\nabla f^i(x)\|), -\textnormal{ln}(1-\|\nabla f^i(x)\|)$ under different scenarios completes the proof.
\end{proof}
\begin{lemma}\label{lemma_6}
     Let Assumption~\ref{definition_1} hold. Suppose that $f^i:\mathbb{R}^d\to\mathbb{R}$ is twice continuously differentiable and convex. For any $k\geq 0$, let $x^i_k\in\mathbb{R}^d$ generated by Algorithm~\ref{alg:dgd}. Suppose also that a minimizer of $F$, $x^*$ exists. Then we have the following relationship:
        \begin{equation}
            \|\bar{x}_{k+1}-x^*\|\leq \|\bar{x}_k-x^*\|,
        \end{equation}
    where $\bar{x}_{k}=\frac{1}{N}\sum_{i=1}^Nx^i_k$.
\end{lemma}
\begin{proof}
To ease the analysis, we assume that the step size is constant for now and will compare to the one we have defined in Algorithm~\ref{alg:dgd}.
Recall the update law from Algorithm~\ref{alg:dgd} in the following:
\begin{equation}
    x^i_{k+1}=\sum_{j\in Nb(i)}\pi_{ij}x^j_k-\eta\nabla f^i(x^i_k)
\end{equation}
Based on the definition of ensemble average: $\bar{x}_{k}=\frac{1}{N}\sum_{i=1}^Nx^i_k$, we can obtain the next inequality:
\begin{equation}
    \bar{x}_{k+1}=\bar{x}_k-\alpha\frac{1}{N}\sum_{i=1}^N\nabla f^i(x^i_k)
\end{equation}
We define the consensus error as follows:
\begin{equation}
    \mathbf{q}^i_k=x^i_k-\bar{x}_k, \; \textnormal{with} \;\mathbf{q}_k=[\mathbf{q}^1_k;...;\mathbf{q}^N_k]
\end{equation}
Denote by $F(\mathbf{x}_k)=[f^1(x^1_k);...;f^N(x^N_k)]$. We first bound the consensus error. Based on the definition above, we have
\begin{equation}
    \mathbf{q}_{k+1}=\Pi\mathbf{q}_k-\alpha\bigg(F(\mathbf{x}_k)-\mathbf{1}\otimes\frac{1}{N}\sum_{i=1}^N\nabla f^i(\bar{x}_k)\bigg)
\end{equation}
Using Assumption~\ref{assumption_1}, we have
\begin{equation}\label{eq_62}
    \|\mathbf{q}_{k+1}\|\leq \sqrt{\rho}\|\mathbf{q}_k\|+\alpha\|F(\mathbf{x}_k)-\mathbf{1}\otimes\frac{1}{N}\sum_{i=1}^N\nabla f^i(\bar{x}_k)\|.
\end{equation}
According to Assumption~\ref{definition_1}, it is obtained that
\begin{equation}
    \|\nabla f^i(x^i_k)-\nabla f^i(\bar{x}_k)\|\leq (L_0+L_1\|\nabla f^i(\bar{x}_k)\|)\frac{1}{L_1}.
\end{equation}
Summing the above inequality over $i\in\mathcal{V}$
\begin{equation}\label{eq_64}
    \bigg\|\frac{1}{N}\sum_{i=1}^N\nabla f^i(x^i_k)-F(\bar{x}_k)\bigg\|\leq \frac{L_0+L_1G_k}{L_1}\cdot\frac{1}{N}\sum_{i=1}^N\|\mathbf{q}^i_k\|,
\end{equation}
where $G_k:=\textnormal{max}_i\|\nabla f^i(\bar{x}_k)\|$.
Substituting Eq.~\ref{eq_64} into Eq.~\ref{eq_62} yields:
\begin{equation}\label{eq_65}
    \|\mathbf{q}_{k+1}\|\leq(\sqrt{\rho}+\alpha\frac{L_0+L_1G_k}{L_1})\|\mathbf{q}_{k}\|.
\end{equation}
We next analyze the optimality gap. Recalling 
\begin{equation}
    \bar{x}_{k+1}=\bar{x}_k-\alpha\frac{1}{N}\sum_{i=1}^N\nabla f^i(x^i_k),
\end{equation}
we have 
\begin{equation}
    \|\bar{x}_{k+1}-x^*\|^2=\|\bar{x}_{k}-x^*\|^2-2\alpha\bigg\langle\bar{x}_k-x^*,\frac{1}{N}\sum_{i=1}^N\nabla f^i(x^i_k)\bigg\rangle + \alpha^2\|\frac{1}{N}\sum_{i=1}^N\nabla f^i(x^i_k)\|^2
\end{equation}
We now split the gradient on the second term of the above equality.
\begin{equation}
    \frac{1}{N}\sum_{i=1}^N\nabla f^i(x^i_k)=\nabla F(\bar{x}_k)+\frac{1}{N}\sum_{i=1}^N(f^i(x^i_k)-f^i(\bar{x}_k))
\end{equation}
With convexity, we have
\begin{equation}
    \langle\bar{x}_k-x^*,\nabla F(\bar{x}_k)\rangle\geq F(\bar{x}_k)-F^*
\end{equation}
Therefore, we can get
\begin{equation}
    \|\bar{x}_{k+1}-x^*\|^2\leq\|\bar{x}_{k}-x^*\|^2-2\alpha (F(\bar{x}_k)-F^*) + \alpha^2 \|\nabla F(\bar{x}_k)+\frac{1}{N}\sum_{i=1}^N(f^i(x^i_k)-f^i(\bar{x}_k))\|^2
\end{equation}
According to Eq.~\ref{eq_64}, we know that $\frac{1}{N}\sum_{i=1}^N(f^i(x^i_k)-f^i(\bar{x}_k))$ is bounded by the consensus error. Hence, for the term $-2\alpha (F(\bar{x}_k)-F^*)$ to dominate, we need
\begin{equation}
    \alpha\leq \frac{F(\bar{x}_k-F^*)}{\|\nabla F(\bar{x}_k)\|^2}.
\end{equation}
Using the condition $\|\bar{x}_k-x^*\|\leq \frac{1}{L_1}$ and convexity, we have the following:
\begin{equation}
    F(\bar{x}_k)-F^*\leq \|\nabla F(\bar{x}_k)\|\|\bar{x}_k-x^*\|\leq \frac{\|\nabla F(\bar{x}_k)\|}{L_1},
\end{equation}
which produces $\alpha\leq\frac{1}{L_1\|\nabla F(\bar{x}_k)\|}$. 
To maintain the consensus stability, the following condition should satisfy:
\begin{equation}
    \sqrt{\rho}+\alpha\frac{L_0+L_1G_k}{L_1}<1
\end{equation}
such that $\alpha\leq \frac{(1-\sqrt{\rho})L_1}{L_0+L_1G_k}$. 
Combining with the consensus stability condition in Eq.~\ref{eq_65}, we have
\begin{equation}\label{eq_74}
    \alpha\leq \textnormal{min}\bigg\{\frac{(1-\sqrt{\rho})L_1}{L_0+L_1G_k},\frac{1}{L_1\|\nabla F(\bar{x}_k)\|}\bigg\}
\end{equation}
Based on $(L_0,L_1)$-smoothness, the following is attained
\begin{equation}
    \|\nabla F(x)\|\leq (L_0+L_1\|\nabla F(x)\|)\frac{1}{L_1}
\end{equation}
such that
\begin{equation}
    \|\nabla F(x)\|\leq\frac{L_0}{L_1-L_1\|\nabla F(x)\|}. 
\end{equation}
We solve the above inequality iteratively under $\|x-x^*\|\leq \frac{1}{L_1}$ yielding an implicit bound for $\|\nabla F(x)\|\leq\frac{L_0}{L_1}$. Substituting this into Eq.~\ref{eq_74} yields the step size
\begin{equation}
    \alpha\leq\frac{1-\sqrt{\rho}}{2L_0(2-\sqrt{\rho})}<\frac{1}{2L_0},
\end{equation}
which matches our definition in the Algorithm~\ref{alg:dgd}.
\end{proof}

\begin{corollary}\label{corollary_3}
    Let $f$ be a convex $(L_0,L_1)$-smooth nonlinear function. Then for any $x,y\in\mathbb{R}^d$, we have the following relationship
    \begin{equation}
        \frac{\|\nabla f(y)-\nabla f(x)\|^2}{2(L_0+L_1\|\nabla f(y)\|)+L_1\|\nabla f(y)-\nabla f(x)\|}\leq f(y)-f(x)-\langle\nabla f(x), y-x\rangle.
    \end{equation}
\end{corollary}
\begin{proof}
    The proof follows similarly from the proof for Corollary 2.8 in~\cite{vankov2024optimizing}.
\end{proof}
\begin{lemma}\label{lemma_8}
    Let Assumptions~\ref{assumption_1}, \ref{definition_1} and~\ref{assum_3} hold and $\{\bar{x}_k\}$ be the iterates produced by Algorithm~\ref{alg:dsgd} with step size $\alpha_k=\textnormal{min}\{\hat{\alpha}, \frac{1}{\textnormal{max}_i\{\|\nabla f^i(\bar{x}_k)\|\}},\frac{1-\sqrt{\rho}}{\textnormal{max}_i\{\sqrt{24L_1\|\nabla f^i(\bar{x}_k)\|}\}}\}$, where $\hat{\alpha}=\textnormal{min}\{\frac{1-\sqrt{\rho}}{\textnormal{max}\{20L_0,\sqrt{24L_0}\}},\frac{1}{4C_1}, \frac{1}{\sqrt{K}}\}$. Suppose that $f^i:\mathbb{R}^d\to\mathbb{R}$ is a twice continuously differentiable convex function. The the following relationship holds:
\begin{equation}
\begin{split}
    &\frac{1}{2}\sum_{k=1}^K\frac{1}{N}\sum_{i=1}^N\|x^i_k-\bar{x}_k\|^2\leq \sum_{k=1}^{K}\frac{3\alpha^2_k(\sigma^2+\delta^2)}{(1-\sqrt{\rho})^2}\\&+\sum_{k=1}^K\frac{1}{N}\sum_{i=1}^N\frac{6\alpha^2_k}{(1-\sqrt{\rho})^2}(2L_0+3L_1\|\nabla f^i(\bar{x}_k)\|)(f^i(\bar{x}_k)-f^i_*).
\end{split}
\end{equation}
\end{lemma}
\begin{proof}
    With abuse of notation, we use some upper bold characters to represent vectors after they are expanded. 
Define
\[\mathbf{X}_k=[x^1_k,x^2_k,...,x^N_k]\in\mathbb{R}^{d\times N},\]\[\mathbf{G}_k=[g^1_k,g^2_k,...,g^N_k]\in\mathbb{R}^{d\times N},\]\[\mathbf{H}_k=[\nabla f^1(x^1_k),\nabla f^2(x^2_k),...;\nabla f^N(x^N_k)]\in\mathbb{R}^{d\times N},\]\[\mathbf{Q}=\frac{1}{dN}\mathbf{1}\mathbf{1}_{dN}\in\mathbb{R}^{dN\times dN}\]
Without loss of generality, suppose that the initialization of $\mathbf{X}$ is $\mathbf{X}_0=0$  throughout the rest of analysis. With Algorithm~\ref{alg:dgd}, we have
\begin{equation}
    \mathbf{X}_k = -\alpha\sum_{\tau=0}^{k-1}\mathbf{G}_\tau \mathbf{P}^{k-1-\tau}
\end{equation}
Right multiplying the above equation by $\mathbf{I-Q}$ yields the following relationship
\begin{equation}\label{eq_14}
    \mathbf{X}_k(\mathbf{I-Q}) = -\alpha\sum_{\tau=0}^{k-1}\mathbf{G}_\tau(\mathbf{I-Q})\mathbf{P}^{k-1-\tau},
\end{equation}
which will serve to characterize the optimal error bound. By taking the squared Frobenius norm $\|\cdot\|_F$ and expectation on both sides, we have
\begin{equation}\label{consensus}
    \mathbb{E}[\|\mathbf{X}_k(\mathbf{I-Q})\|_F^2]= \alpha^2\mathbb{E}[\|\sum_{\tau=0}^{k-1}\mathbf{G}_\tau(\mathbf{I-Q})\mathbf{P}^{k-1-\tau}\|_F^2].
\end{equation}
The left side of above equation is equivalent to $\mathbb{E}[\sum_{i=1}^N\|x^i_k-\bar{x}_k\|^2]$. To further analyze the Eq.~\ref{consensus}, we investigate the second term of its right side in the following.
\begin{equation}\label{eq_16}
\begin{split}
    \alpha^2\mathbb{E}[\|\sum_{\tau=0}^{k-1}\mathbf{G}_\tau(\mathbf{I-Q})\mathbf{P}^{k-1-\tau}\|_F^2]&\leq \alpha^2\sum_{\tau=0}^{k-1}\sum_{\tau'=0}^{k-1}\mathbb{E}[\|\mathbf{G}_\tau(\mathbf{I}-\mathbf{Q})\mathbf{P}^{k-1-\tau}\|_F\|\mathbf{G}_\tau'(\mathbf{I}-\mathbf{Q})\mathbf{P}^{k-1-\tau'}\|_F]\\&\leq \sum_{\tau=0}^{k-1}\sum_{\tau'=0}^{k-1}\rho^{k-1-\frac{\tau+\tau'}{2}}\mathbb{E}[\|\mathbf{G}_\tau\|_F\|\mathbf{G}_\tau'\|_F]\\&\leq\sum_{\tau=0}^{k-1}\sum_{\tau'=0}^{k-1}\rho^{k-1-\frac{\tau+\tau'}{2}}\mathbb{E}[\|\mathbf{G}_\tau\|_F^2]\\&\leq \frac{1}{1-\sqrt{\rho}}\sum_{\tau=0}^{k-1}\rho^{\frac{k-1-\tau}{2}}\mathbb{E}[\|\mathbf{G}_\tau\|_F^2]
\end{split}
\end{equation}
We next bound the term $\mathbb{E}[\|\mathbf{G}_\tau\|_F^2]$, which can be rewritten as
\begin{equation}
\begin{split}
    \mathbb{E}[\|\mathbf{G}_\tau\|_F^2]&=\mathbb{E}[\|\mathbf{G}_\tau-\mathbf{H}_\tau+\mathbf{H}_\tau-\mathbf{H}_\tau\mathbf{Q}+\mathbf{H}_\tau\mathbf{Q}\|_F^2]\\&3\mathbb{E}[\|\mathbf{G}_\tau-\mathbf{H}_\tau\|^2_F]+3\mathbb{E}[\|\mathbf{H}_\tau(\mathbf{I-Q})\|^2_F]+3\mathbb{E}[\|\mathbf{H}_\tau \mathbf{Q}\|^2_F]\\&\leq 3N\sigma^2+3N\delta^2+3\mathbb{E}[\|\frac{1}{N}\sum_{i=1}^N\nabla f^i(x^i_\tau)\|^2],
\end{split}
\end{equation}
which is due to Assumption~\ref{assum_3}. With the above inequality in hand, we now know that
\begin{equation}\label{eq_43}
\begin{split}
    \mathbb{E}[\|\sum_{\tau=0}^{k-1}\mathbf{G}_\tau(\mathbf{I-Q})\mathbf{P}^{k-1-\tau}\|_F^2]&\leq\frac{1}{1-\sqrt{\rho}}\sum_{\tau=0}^{k-1}\rho^{\frac{k-1-\tau}{2}}[3N\sigma^2+3N\delta^2+3\mathbb{E}[\|\frac{1}{N}\sum_{i=1}^N\nabla f^i(x^i_\tau)\|^2]]\\&\leq \frac{3N(\sigma^2+\delta^2)}{(1-\sqrt{\rho})^2}+\frac{3N}{(1-\sqrt{\rho})}\sum_{\tau=0}^{k-1}\rho^{\frac{k-1-\tau}{2}}\mathbb{E}[\|\frac{1}{N}\sum_{i=1}^N\nabla f^i(x^i_\tau)\|^2]
\end{split}
\end{equation}
We rewrite the second term at the right hand side of the above inequality as
$\frac{3N}{(1-\sqrt{\rho})}\sum_{\tau=0}^{k-1}\rho^{\frac{k-1-\tau}{2}}\mathbb{E}[\|\frac{1}{N}\sum_{i=1}^N\nabla f^i(x^i_\tau)\|^2]=\frac{3N}{(1-\sqrt{\rho})}\sum_{\tau=0}^{k-1}\rho^{\frac{k-1-\tau}{2}}\mathbb{E}[\|\frac{1}{N}\sum_{i=1}^N(\nabla f^i(x^i_\tau)-\nabla f^i(\bar{x}_\tau)+\nabla f^i(\bar{x}_\tau)-\nabla f^i_*)\|^2]$, where $\nabla f^i_*:=\nabla f^i(x^*)=0$. By using $\|a+b\|^2\leq 2\|a\|^2+2\|b\|^2$ again we can get the following relationship
\begin{equation}\label{eq_44}
\begin{split}
    &\frac{3N}{(1-\sqrt{\rho})}\sum_{\tau=0}^{k-1}\rho^{\frac{k-1-\tau}{2}}\mathbb{E}[\|\frac{1}{N}\sum_{i=1}^N\nabla f^i(x^i_\tau)\|^2]\\&\leq \frac{3N}{(1-\sqrt{\rho})}\sum_{\tau=0}^{k-1}\rho^{\frac{k-1-\tau}{2}}[\frac{2}{N}\sum_{i=1}^N\mathbb{E}[(L_0+L_1\|\nabla f^i(\tau)\|)^2\|x^i_\tau-\bar{x}_\tau\|^2\\&+\frac{2}{N}\sum_{i=1}^N(2L_0+3L_1\|\nabla f^i(\bar{x}_\tau)\|)(f^i(\bar{x}_\tau)-f^i_*)].
\end{split}
\end{equation}
The last inequality is due to Assumption~\ref{definition_1} and Corollary~\ref{corollary_3}. Combining Eq.~\ref{consensus}, Eq.~\ref{eq_43}, Eq.~\ref{eq_44}, and dividing both sides by $\frac{1}{N}$ yields the following:
\begin{equation}
\begin{split}
    \frac{1}{N}\sum_{i=1}^N\|x^i_k-\bar{x}_k\|^2&\leq \frac{3\alpha^2_k(\sigma^2+\delta^2)}{(1-\sqrt{\rho})^2}+\frac{6\alpha^2_k}{1-\sqrt{\rho}}\sum_{\tau=0}^{k-1}\rho^{\frac{k-1-\tau}{2}}[\frac{1}{N}\sum_{i=1}^N\mathbb{E}[(L_0+L_1\|\nabla f^i(\tau)\|)^2\|x^i_\tau-\bar{x}_\tau\|^2\\&+\frac{1}{N}\sum_{i=1}^N(2L_0+3L_1\|\nabla f^i(\bar{x}_\tau)\|)(f^i(\bar{x}_\tau)-f^i_*)]
\end{split}
\end{equation}
Summing over $k\in\{0,1,...,K-1\}$ the above inequality yields the following relationship:
\begin{equation}
\begin{split}
    \sum_{k=1}^K\frac{1}{N}\sum_{i=1}^N\|x^i_k-\bar{x}_k\|^2&\leq \sum_{k=1}^{K}\frac{3\alpha^2_k(\sigma^2+\delta^2)}{(1-\sqrt{\rho})^2}\\&+\sum_{k=1}^K\frac{6\alpha^2_k}{1-\sqrt{\rho}}\sum_{\tau=0}^{k-1}\rho^{\frac{k-1-\tau}{2}}[\frac{1}{N}\sum_{i=1}^N\mathbb{E}[(L_0+L_1\|\nabla f^i(\bar{x}_\tau)\|)^2\|x^i_\tau-\bar{x}_\tau\|^2\\&+\frac{1}{N}\sum_{i=1}^N(2L_0+3L_1\|\nabla f^i(\bar{x}_\tau)\|)(f^i(\bar{x}_\tau)-f^i_*)]\\&\leq \sum_{k=1}^{K}\frac{3\alpha^2_k(\sigma^2+\delta^2)}{(1-\sqrt{\rho})^2}+\\&\sum_{k=1}^K\frac{6\alpha^2_k}{1-\sqrt{\rho}}\frac{1-\rho^{\frac{K-1-k}{2}}}{1-\sqrt{\rho}}[\frac{1}{N}\sum_{i=1}^N(L_0+L_1\|\nabla f^i(\bar{x}_k)\|)^2\|x^i_k-\bar{x}_k\|^2\\&+\frac{1}{N}\sum_{i=1}^N(2L_0+3L_1\|\nabla f^i(\bar{x}_k)\|)(f^i(\bar{x}_k)-f^i_*)]\\&\sum_{k=1}^{K}\frac{3\alpha^2_k(\sigma^2+\delta^2)}{(1-\sqrt{\rho})^2}+\\&\sum_{k=1}^K\frac{6\alpha^2_k}{(1-\sqrt{\rho})^2}[\frac{1}{N}\sum_{i=1}^N(L_0+L_1\|\nabla f^i(\bar{x}_k)\|)^2\|x^i_k-\bar{x}_k\|^2\\&+\frac{1}{N}\sum_{i=1}^N(2L_0+3L_1\|\nabla f^i(\bar{x}_k)\|)(f^i(\bar{x}_k)-f^i_*)]
\end{split}
\end{equation}
The last inequality results in
\begin{equation}
\begin{split}
    &\sum_{k=1}^K\frac{1}{N}\sum_{i=1}^N(1-\frac{6\alpha_k^2}{(1-\sqrt{\rho})^2}(L_0+L_1\|\nabla f^i(\bar{x}_k)\|)\|x^i_k-\bar{x}_k\|^2\leq \sum_{k=1}^{K}\frac{3\alpha^2_k(\sigma^2+\delta^2)}{(1-\sqrt{\rho})^2}\\&+\sum_{k=1}^K\frac{1}{N}\sum_{i=1}^N\frac{6\alpha^2_k}{(1-\sqrt{\rho})^2}(2L_0+3L_1\|\nabla f^i(\bar{x}_k)\|)(f^i(\bar{x}_k)-f^i_*)
\end{split}
\end{equation}
We have the following condition based on the step size $\alpha_k\leq \frac{1-\sqrt{\rho}}{\textnormal{max}_i\{\sqrt{24L_1\|\nabla f^i(\bar{x}_k)\|}\}}$ and $\alpha_k\leq \frac{1-\sqrt{\rho}}{\textnormal{max}\{20L_0,\sqrt{24L_0}\}}$:
\begin{equation}
    1-\frac{6\alpha_k^2}{(1-\sqrt{\rho})^2}(L_0+L_1\|\nabla f^i(\bar{x}_k)\|)\geq \frac{1}{2},
\end{equation}
which results in
\begin{equation}\label{eq_45}
\begin{split}
    &\frac{1}{2}\sum_{k=1}^K\frac{1}{N}\sum_{i=1}^N\|x^i_k-\bar{x}_k\|^2\leq \sum_{k=1}^{K}\frac{3\alpha^2_k(\sigma^2+\delta^2)}{(1-\sqrt{\rho})^2}\\&+\sum_{k=1}^K\frac{1}{N}\sum_{i=1}^N\frac{6\alpha^2_k}{(1-\sqrt{\rho})^2}(2L_0+3L_1\|\nabla f^i(\bar{x}_k)\|)(f^i(\bar{x}_k)-f^i_*).
\end{split}
\end{equation}
This completes the proof.
\end{proof}
    
    
\begin{lemma}\label{lemma_10}
    Given $\bar{x}_k\in\mathbb{R}^d$ produced by Algorithm~\ref{alg:dsgd} such that $F(\bar{x}_k)\leq F(\bar{x}_0)$, for all $k\geq 0$. If $\|\bar{x}_{k+1}-\bar{x}_k\|\leq\textnormal{min}\{1/L_1,1\}$ for any $\bar{x}_{k+1}\in\mathbb{R}^d$, then we have $\|\nabla F(\bar{x}_{k+1})\|\leq 4(L_0/L_1+\|\nabla F(\bar{x}_{k})\|)$.
\end{lemma}
\begin{proof}
    The proof follows from the proof in Lemma 9 in~\cite{zhang2019gradient} and let $x=\bar{x}_k$ and $x^+=\bar{x}_{k+1}$.
\end{proof}
\begin{lemma}\label{lemma_11}
    Let $\{x^i_k\}$ be the iterates of Algorithm~\ref{alg:dsgd} such that $\bar{x}_k=\frac{1}{N}\sum_{i=1}^Nx^i_k$. Let Assumption~\ref{definition_1} hold. The following relationship holds:
    \begin{equation}
    \begin{split}
        \mathbb{E}[\alpha_k\langle\nabla F(\bar{x}_k),\frac{1}{N}\sum_{i=1}^Ng^i_k\rangle]&\geq \frac{1}{2}\mathbb{E}[\alpha_k\|\nabla F(\bar{x}_k)\|^2]+\frac{1}{2}\mathbb{E}[\alpha_k\|\frac{1}{N}\sum_{i=1}^N\nabla f^i(x^i_k)\|]\\&-\frac{1}{2}\mathbb{E}[\alpha_k\frac{1}{N}\sum_{i=1}^N(L_0+L_1\|\nabla f^i(\bar{x}_k)\|)^2\|x^i_k-\bar{x}_k\|^2].
    \end{split}
    \end{equation}
\end{lemma}
\begin{proof}
    We apply the basic equality $\langle a,b\rangle=\frac{1}{2}[\|a\|^2+\|b\|^2-\|a-b\|^2]$ such that we obtain the following relationship
    \begin{equation}
    \begin{split}
        \mathbb{E}[\langle\nabla F(\bar{x}_k),\frac{1}{N}\sum_{i=1}^Ng^i_k\rangle]&=\mathbb{E}[\langle\nabla F(\bar{x}_k),\frac{1}{N}\sum_{i=1}^N\nabla f^i(x^i_k)\rangle]\\&=
        \frac{1}{2}\mathbb{E}[\|\nabla F(\bar{x}_k)\|^2+\|\frac{1}{N}\sum_{i=1}^N\nabla f^i(x^i_k)\|^2]-\|\nabla F(\bar{x}_k)-\frac{1}{N}\sum_{i=1}^N\nabla f^i(x^i_k)\|^2
    \end{split}
    \end{equation}
Using the fact that $\|\nabla F(\bar{x}_k)-\frac{1}{N}\sum_{i=1}^N\nabla f^i(x^i_k)\|^2=\|\frac{1}{N}\sum_{i=1}^N\nabla f^i(\bar{x}_k)-\frac{1}{N}\sum_{i=1}^N\nabla f^i(x^i_k)\|^2$, applying the $(L_0,L_1)$-smoothness assumption, and multiplying both sides with $\alpha_k$ completes the proof.
\end{proof}
\begin{lemma}\label{lemma_12}
    Let $\{x^i_k\}$ be the iterates of Algorithm~\ref{alg:dsgd} such that $\bar{x}_k=\frac{1}{N}\sum_{i=1}^Nx^i_k$. Let Assumption~\ref{definition_1} hold. Then we have the following relationship:
    \begin{equation}
    \begin{split}
        \frac{5L_0+4L_1\|\nabla F(\bar{x}_k)\|}{2}\mathbb{E}[\alpha_k^2\|\frac{1}{N}\sum_{i=1}^Ng^i_k\|^2]&\leq\frac{(5L_0+2L_1\sigma)\alpha_k^2\sigma^2}{N}+\frac{1}{4}\mathbb{E}[\|\nabla F(\bar{x}_k)\|^2\alpha_k] \\&+ (5L_0+4L_1\|\nabla F(\bar{x}_k)\|)\mathbb{E}[\alpha_k^2\|\frac{1}{N}\sum_{i=1}^N\nabla f^i(x^i_k)\|^2]
    \end{split}
    \end{equation}
\end{lemma}
\begin{proof}
    As $\mathbb{E}[\alpha_k^2\|\frac{1}{N}\sum_{i=1}^Ng^i_k\|^2]=\mathbb{E}[\alpha_k^2\|\frac{1}{N}\sum_{i=1}^N(g^i_k-\nabla f^i(x^i_k)+\nabla f^i(x^i_k))\|^2]$, we have that
    \begin{equation}
        \mathbb{E}[\alpha_k^2\|\frac{1}{N}\sum_{i=1}^Ng^i_k\|^2]\leq 2\mathbb{E}[\alpha^2_k\|\frac{1}{N}\sum_{i=1}^N(g^i_k-\nabla f^i(x^i_k))\|^2]+2\mathbb{E}[\alpha^2_k\|\frac{1}{N}\sum_{i=1}^N\nabla f^i(x^i_k)\|^2]
    \end{equation}
    We next bound the term $\mathbb{E}[\alpha^2_k\|\frac{1}{N}\sum_{i=1}^N(g^i_k-\nabla f^i(x^i_k))\|^2]$.
    We investigate two scenarios, including $\|\nabla f^i(\bar{x}_k)\|\geq L_0/L_1 + \sigma$ and $\|\nabla f^i(\bar{x}_k)\|< L_0/L_1 + \sigma$. Following the similar proof from Lemma 13 in~\cite{zhang2019gradient} yields the desirable results. 
\end{proof}

\end{document}